\documentclass[leqno,11pt]{amsart}
\usepackage{amsmath,amscd,amsthm,amsxtra}
\usepackage{epsfig,graphics,color,colortbl}
\usepackage{amssymb,latexsym}
\usepackage{mathrsfs,upgreek} %eucal
\usepackage[poly,all,cmtip]{xy}
\usepackage{arydshln}
\usepackage{adjustbox,multirow}
\usepackage[utf8]{inputenc}
\usepackage{booktabs}
\usepackage{tikz-cd}
\usetikzlibrary{arrows,matrix}
\tikzset{tab/.style={matrix of math nodes,column sep=-.35, row sep=-.35,text height=7pt,text width=7pt,align=center,inner sep=2,font=\footnotesize}}
\usepackage[normalem]{ulem}

\usepackage[colorlinks=true, pdfstartview=FitV, linkcolor=blue, citecolor=blue, urlcolor=blue]{hyperref}

\usepackage{float}
\usepackage{enumerate}

\newcommand{\g}{\mathfrak{g}}

\newcommand{\iso}{\cong}  % isomorphism
\newcommand{\clfw}{\overline{\Lambda}} % classical fundamental weight
\newcommand{\B}{\mathbf{B}}
\newcommand{\abs}[1]{\lvert #1 \rvert} % absolute value
\newcommand{\Abs}[1]{\lVert #1 \rVert} % absolute value
\newcommand{\inner}[2]{\langle #1, #2 \rangle} % inner product
\newcommand{\cosetrepr}[1]{\lfloor #1 \rfloor} % minimal length coset representative
\newcommand{\virtual}[1]{\widetilde{#1}}
\newcommand{\coroot}{\alpha^{\vee}}  % a coroot
\newcommand{\et}{e}  % crystal e operator
\newcommand{\ft}{f}  % crystal f operator
\newcommand{\bopen}{{\color{blue}(}}
\newcommand{\bclose}{{\color{red})}}
\newcommand{\zero}{\mathbf{0}}
\newcommand{\oneb}{\mathbf{1}}

\newcommand{\aff}[1]{\widehat{#1}}  % affine object
\newcommand{\gaff}{\aff{\g}}  % (dual untwisted) affine Lie algebra
\newcommand{\gutw}{\aff{\g}^{\vee}}  % untwisted affine Lie algebra

\newcommand{\ZZ}{\mathbb{Z}}
\newcommand{\QQ}{\mathbb{Q}}
\newcommand{\RR}{\mathbb{R}}

\newcommand{\bon}{\overline{1}}
\newcommand{\btw}{\overline{2}}

\newcommand{\mcB}{\mathcal{B}}
\newcommand{\mcC}{\mathcal{C}}
\newcommand{\mcI}{\mathcal{I}}
\newcommand{\mcJ}{\mathcal{J}}
\newcommand{\mcP}{\mathcal{P}}
\newcommand{\mcR}{\mathcal{R}}

\newcommand{\cc}{\mathbf{c}}

\newcommand{\ii}{\mathbf{i}}
\newcommand{\kk}{\mathbf{k}}

\newcommand{\pp}{\mathbf{p}}

\DeclareMathOperator{\proj}{proj} % Projection
\DeclareMathOperator{\Inv}{Inv} % Inversion set
\DeclareMathOperator{\wt}{wt} % weight

\newtheorem{thm}{Theorem}[section]
\newtheorem{prop}[thm]{Proposition}
\newtheorem{cor}[thm]{Corollary}
\newtheorem{lem}[thm]{Lemma}

\theoremstyle{definition}
\newtheorem{df}[thm]{Definition}
\newtheorem{prob}[thm]{Problem}
\theoremstyle{remark}
\newtheorem{rem}[thm]{Remark}
\newtheorem{ex}[thm]{Example}

\definecolor{darkred}{rgb}{0.7,0,0} % darkred color
\newcommand{\defn}[1]{{\color{darkred}\emph{#1}}} % emphasis of a definition

\definecolor{UQgold}{RGB}{196, 158, 54} % UQ gold
\definecolor{UQpurple}{RGB}{73, 7, 94} % UQ purple

\numberwithin{equation}{section}

\usepackage[colorinlistoftodos]{todonotes}

\title[Special KR crystals]{Special Kirillov--Reshetikhin crystals}

\author[I.-S.~Jang]{Il-Seung Jang}
\address[I.-S. Jang]{Department of Mathematics, Incheon National University, Incheon 22012, Republic of Korea}
\email{ilseungjang@inu.ac.kr}
\urladdr{https://sites.google.com/view/isjang/home}

\author[T.~Scrimshaw]{Travis Scrimshaw}
\address[T.~Scrimshaw]{Department of Mathematics, Hokkaido University, 5 Ch\=ome Kita 8 J\=onishi, Kita Ward, Sapporo, Hokkaid\=o 060-0808}
\email{tcscrims@gmail.com}
\urladdr{https://tscrim.github.io/}


\begin{document}

\begin{abstract}
We give a uniform realization of Kirillov--Reshetikhin crystals $B^{r,s}$, in terms of PBW crystals, for all affine types, $s \ge 1$, and nodes $r$ in the orbit of the $0$ node in the Dynkin diagram.
Our proof is almost uniform except we prove some technical lemmas by first proving the case when the $r$-th fundamental weight $\Lambda_r$ is minuscule by using properties of minuscule crystal and then extend this to the case when $\Lambda_r$ is cominuscule using virtual crystals.
\end{abstract}

\maketitle

\setcounter{tocdepth}{1}
\tableofcontents

%%%%%%%%%%%%%%%%%%%%%%%%%%%%%%%%%%%%%%%%
\section{Introduction}

The affine quantum group $U_q'(\gaff)$ (without derivation) admits finite dimensional representations with an important class of irreducible representations called Kirillov--Reshetikhin (KR) modules with a remarkably rich structure (see, \textit{e.g.},~\cite{ChHe10,HeLe20,Her26,Lec11,Oka26,OS25} and references therein).
One of the distinguishing features of KR modules is they (conjecturally) admit crystal (psuedo)bases~\cite{BS20,KKMMNN91,KKMMNN92,Okado07,OS08,Naoi17,NS21}.
These are denoted by $B^{r,s}$, where $r$ corresponds a fundamental weight of classical Lie algebra and $s$ is a positive integer.
Despite intense study over the past 25 years, it remains an open problem to provide a uniform construction and proof for KR crystals.
This is in sharp contrast to Kashiwara crystals $B(\lambda)$ for highest weight representations (including for the (Drinfel'd--Jimbo) quantum group associated to any Kac--Moody Lie algebra)~\cite{Kamnitzer07,K90,K91,K93,L95-2,LP08,Nakajima03,SalisburyS15II}.

There are combinatorial realizations of KR crystals in nonexceptional types~\cite{FOS09,JK19,Kwon13,Shi02} and in various cases in exceptional types~\cite{BS21,JS10,KMOY07,MMO10,Jang22,Yamane98}, but these constructions are type specific and frequently rely on diagram automorphisms (which notably do not exist for type $E_8^{(1)}$).
On the other hand, there is a uniform construction for $B^{r,1}$ given in terms of (projected level zero and quantum) Littelmann paths~\cite{Kashiwara02,LNSSS17a}, but this does not extend for $s > 1$.
Most of the KR crystals are (conjecturally) perfect~\cite{KKMMNN92,KMOY07,FOS10,Hiroshima21,MMO10,Yamane98}, a technical condition that is equivalent to having an isomorphism $B^{r,s} \otimes B(\lambda) \iso B(\mu)$ for any level $s / c_r$ weight $\lambda$ for specific $c_r$ (depending only on $r$ and level $s / c_r$ weight $\mu$.
This construction (known as the Kyoto path model) in turn implies that we can construct the KR crystal up to the last $s / c_r$ $0$-arrows in a $0$-string by using Demazure crystals~\cite{FSS07,ST12} (see also~\cite{LS19III}).

Among KR crystals, a particularly nice class are those $r$ such that there exists an automorphism $\phi$ of the (affine) Dynkin diagram such that $\phi(r) = 0$ (see Table~\ref{tab:dynkins} below), and such nodes and KR crystals are called special.
These can also be characterized by the KR crystals $B^{r,s}$ such that $B^{r,s} \iso B(s\Lambda_r)$ as classical crystals (crystals for finite dimensional simple Lie algebra $\g$ obtained by removing $0$ from the Dynkin diagram).
Moreover, the classical crystal $B(\Lambda_r)$ is a cominuscule (resp.\ minuscule) representation when $\gaff$ is an untwisted (resp.\ dual untwisted) affine Lie algebra, and in fact, $r$ is special if and only if it is (co)minuscule.
Therefore, its natural to first restrict ourselves to studying special KR crystals.

The main result of this paper is a uniform construction of KR crystals for all special nodes $r$ (Theorem~\ref{thm:main1}).
Our proof is almost uniform, except we have to split it into two cases depending on if the node $r$ is minuscule or cominuscule.
Let us now summarize our construction, which is a uniform version of the KR crystal realizations given in~\cite{JK19,Jan25,Kwon13}.
First, we start with the crystal $B(\infty)$ of the lower half of the quantum group $U_q(\g)^-$, which has a parameterization given by Lusztig~\cite{Lusztig90} (see also~\cite{Saito94}) that describes the leading term of the crystal basis element into PBW monomials.
There is a natural embedding (up to a trivial weight shift) of $B(s\Lambda_r)$ inside of $B(\infty)$, and the image is supported on an explicit set of negative roots such that in the limit $s \to \infty$, it corresponds to the crystal $\B^J$ for the parabolic Verma module of highest weight $0$.
We then promote this to a $U_q'(\gaff)$-crystal by defining the $e_0$ (resp.\ $f_0$) crystal operator to add (resp.\ subtract) $1$ from $\Theta \equiv -\alpha_0 \pmod{\delta}$, where $\delta$ is the null root of $\gaff$, from the Lusztig datum, which is unique  up to trivial swaps because $\Lambda_r$ is minuscule (see Remark~\ref{rem: w^J}).
This gives a combinatorial construction for certain crystals of prefundamental representations~\cite{HJ12,Wang23} (Remark~\ref{eq:B upper J} and Remark~\ref{rem:various}).

For our proof, we use a reconstruction theorem from~\cite{JS10}, which allows us to uniquely determine the KR crystal by three different branching rules and requires two technical statements about special KR crystals (Lemma~\ref{lemma:decomp_by_weight} and Lemma~\ref{lemma:hw_decomp_elements}).
The proof of these lemmas is where we require to split our argument to the case when $r$ is minuscule and not.
Indeed, the minuscule cases are shown by using the simple characterization of (the unique component) $B(s\Lambda_r) \subseteq B(\Lambda_r)^{\otimes s}$ from~\cite[Prop.~7.29]{Scrimshaw20} (see also~\cite[Thm~1.2, Rem.~2.30]{DEKMF22} and~\cite{L94}), whereas the analogous property does not hold for cominuscule nodes (Example~\ref{ex:ordering_failure}).
To prove the cominuscule case, we essentially embed the crystal inside of the corresponding dual type, where the node is minuscule, by using virtual crystals~\cite{K96, OSS03II, OSS03III}.

This paper also contains two other strongly correlated results.
For a minuscule weight $\Lambda_r$, the crystal is described by the orbit of the Weyl group on $\Lambda_r$ (more precisely, the lowering operators correspond to the Demazure or $0$-Hecke product), which is naturally identified with an order ideals of a minuscule heap~\cite{Wildberger03} (see also~\cite[Sec.~11]{Proctor84}) via the Weyl group.
From the characterization of $B(s\Lambda_r) \subseteq B(\Lambda_r)^{\otimes s}$, we can generalize this to a classical crystal on reverse plane partitions on the minuscule heap (Proposition~\ref{prop:direct_limit_finite_type}), and show in the limit $s \to \infty$, this describes the polyhedral realization~\cite{Nakashima99,NZ97} (which arises from~\cite{K93}) of $\B^J$ (Proposition~\ref{prop:heap_string}).

The second is a combinatorial description of the statistic $\varepsilon_i^*$ to determine $B(s\Lambda_r) \subseteq \B^J \subseteq B(\infty)$ when $r$ is special by using the $\ii$-trails introduced by Berenstein and Zelevinsky~\cite[Thm.~3.7]{BZ01}.
More precisely, we describe $\varepsilon_i^*$ as a maximum over paths in a certain directed graph; that is, as a piecewise-linear (i.e., tropical) polynomial.
The cases for types ADE were previously known~\cite{Greene74,JK19,Jang22,Kwon13}, and we complete the picture by determining the paths for the remaining (co)minuscule nodes in types BCD (Theorem~\ref{thm:path formula}).
Our proof is again based on virtual crystals, but now from diagram foldings of the ADE type~\cite{K96} (see also~\cite{OSS03III,OSS03II,NS03,NS05II,Kwon18a}).

The paper is organized as follows.
In Section \ref{sec:preliminaries}, 
we recall affine root systems and study some properties of affine Weyl groups. Also we recall the notion of crystals, introduce the important classes of crystals, and review some properties in Section \ref{sec:crystals}.
In Section \ref{sec:minuscule crystals}, we study some properties of minuscule crystals, which play an important role for proving the main result of this paper.
In Section \ref{sec: construction of special KRs}, we state the main result in this paper, and give its proof. Furthermore, we give some problems related to this result.
In Section \ref{sec:path characterization}, we give the characterization of $\B^{J,s}$ in terms of path combinatorics.

\subsection*{Acknowledgements}

The authors thank Gabriel Frieden, Joel Kamnitzer, Gleb Koshevoy, and Jae-Hoon Kwon for many useful discussions.
The authors thank Anna Pusk\'as for useful discussions about Proposition~\ref{prop:twisted_roots} and the proof in untwisted types.

I.-S.~Jang is supported by the National Research Foundation of Korea (NRF) grant funded by the Korea government(MSIT) (No.~RS-2026-25483668).
T.~Scrimshaw~was partially supported by Grant-in-Aid for JSPS Fellows 21F51028, for Scientific Research for Early-Career Scientists 23K12983, and JSPS KAKENHI Grant Numbers JP23K12983 and JP26K06735.
Part of this work was done while T.~Scrimshaw was at the University of Queensland and OCAMI at Osaka Metropolitan University.
This work benefited from computations done using \textsc{SageMath}~\cite{sage,combinat}.

This work was partly supported by Osaka City University Advanced Mathematical Institute (MEXT Joint Usage/Research Center on Mathematics and Theoretical Physics JPMXP0619217849).
This article is based upon work supported by the National Science Foundation under grant DMS-1929284 while T.~Scrimshaw was in residence at the Institute for Computational and Experimental Research in Mathematics in Providence, RI, during the Categorification and Computation in Algebraic Combinatorics semester program.

%%%%%%%%%%%%%%%%%%%%%%%%%%%%%%%%%%%%%%%%
\section{Affine root systems} \label{sec:preliminaries}

Let $\gaff$ be an affine Kac--Moody Lie algebra with index set $I$, Cartan matrix $\{A_{ij}\}_{i,j\in I}$, weight lattice $\aff{P}$, coweight lattice $\aff{P}^{\vee}$, fundamental weights $\{\Lambda_i\}_{i \in I}$, root system $\aff{\Phi}$ with simple roots $\aff{\Delta} := \{\alpha_i\}_{i \in I}$ and null root $\delta$, simple coroots $\{\coroot_i\}_{i \in I}$, Weyl group $\aff{W}$ with the corresponding simple reflections $\{s_i\}_{i \in I}$, Drinfel'd--Jimbo quantum group $U_q(\aff{\g})$ (taken as a $\QQ(q)$-algebra), and canonical pairing $\langle \ ,\ \rangle \colon \aff{P} \times \aff{P}^{\vee} \to \ZZ$ defined by $\langle \alpha, \alpha_j^{\vee} \rangle = A_{ij}$.
We denote $U_q'(\aff{\g}) := U_q(\aff{\g}') = U_q([\aff{\g}, \aff{\g}])$.
We realized the root lattice $\aff{Q}$ inside the weight lattice with the Cartan matrix defining the linear map expressing the simple roots in terms of the fundamental weights (similarly for the coroot lattice $\aff{Q}^{\vee}$ inside of $\aff{P}^{\vee}$).
The kernel of this map is generated by $\delta$.
Let $\aff{P}^+ := \bigoplus_{i \in I} \ZZ_{\geq0} \Lambda_i$ and $\aff{Q}^+ := \sum_{i \in I} \ZZ_{\geq0} \alpha_i$ denote the positive root and weight cones respectively.
Let $\widehat{\Phi}^{\pm}$ denote the set of positive and negative roots of $\aff{\Phi}$.
For $i, j \in I$, we write $i \sim j$ if $A_{ij} < 0$.

We call a node $r \in I_0$ a \defn{special node} if it is in the orbit of the $0$ node (see Table~\ref{tab:dynkins}).
If $\aff{\g}$ is a untwisted (resp.\ dual untwisted) affine type, then we call $r$ a \defn{cominuscule node} (resp.~\defn{minuscule node}).
Note that if $\aff{\g}$ is simply-laced, then the cominuscule nodes and the minuscule nodes are the special nodes, and if $\aff{\g}$ is not simply-laced, then the minuscule and cominuscule nodes are distinct.
For any simple Lie algebra $\g$, we note that there is a unique dual untwisted affine Lie algebra $\aff{\g}$ whose classical Lie algebra is $\g$, which is given in Table~\ref{table:corresponding_affine}.

\begin{table}[t]
	\resizebox{\columnwidth}{!}{%
		\begin{tabular}{|c||c||c|c|}
			\hline
			type & Dynkin diagram & type & Dynkin diagram \\
			%& cominuscule ($\circ$) & minuscule ($\star$) \\
			\hline \hline
			$A_1^{(1)}$ & 
			\begin{minipage}[c][0.06\textheight]{0.5\linewidth}
			\setlength{\unitlength}{0.17in}
			\begin{picture}(16,3.7)
					\put(7,2.1){\makebox(0,0)[c]{$\circ$}}
					\put(9,2.1){\makebox(0,0)[c]{$\circ$}}
					\put(7.3,2.0){\line(1,0){1.3}}
					\put(7.3,2.2){\line(1,0){1.3}}
					\put(7.0,1.89){\Large $\prec$}
					\put(8.2,1.89){\Large $\succ$}
					\put(7,1.3){\makebox(0,0)[c]{\tiny ${\alpha}_0$}}
					\put(9,1.3){\makebox(0,0)[c]{\tiny ${\alpha}_1$}}
                    %
                    % Kac label
                    \put(7,2.6){\makebox(0,0)[c]{\tiny $1$}}
					\put(9,2.6){\makebox(0,0)[c]{\tiny $1$}}
				\end{picture}
			\end{minipage} &
                $A_2^{(2)}$ & 
                \begin{minipage}[c][0.06\textheight]{0.5\linewidth}
			\setlength{\unitlength}{0.17in}
			\begin{picture}(16,3.7)
					\put(7,2.1){\makebox(0,0)[c]{$\circ$}}
					\put(9,2.1){\makebox(0,0)[c]{$\circ$}}
                        \put(7.3,1.8){\line(1,0){1.3}}
					\put(7.3,2.0){\line(1,0){1.3}}
					\put(7.3,2.2){\line(1,0){1.3}}
                        \put(7.3,2.4){\line(1,0){1.3}}
					%
					%\put(7.0,1.69){\scalebox{2}{$\prec$}}
					\put(7.7,1.69){\scalebox{2}{$\succ$}}
					\put(7,1.3){\makebox(0,0)[c]{\tiny ${\alpha}_0$}}
					\put(9,1.3){\makebox(0,0)[c]{\tiny ${\alpha}_1$}}
				    %
                        % Kac labels
                        %\put(7,2.7){\makebox(0,0)[c]{\tiny $2$}}
                        %\put(9,2.7){\makebox(0,0)[c]{\tiny $1$}}
                        \put(7,2.7){\makebox(0,0)[c]{\tiny $1$}}
                        \put(9,2.7){\makebox(0,0)[c]{\tiny $2$}}
                    \end{picture}
			\end{minipage}
                \\ %$1$ & $1$  \\
			\hline
			$\underset{(n \ge 2)}{A_n^{(1)}}$ & 
			\begin{minipage}[c][0.09\textheight]{0.5\linewidth}
			\setlength{\unitlength}{0.17in}
				\begin{picture}(16,5.5)
					\put(3,2.1){\makebox(0,0)[c]{$\bullet$}}
					\put(5.6,2.1){\makebox(0,0)[c]{$\bullet$}}
					\put(10.5,2.1){\makebox(0,0)[c]{$\bullet$}}
					\put(13,2.1){\makebox(0,0)[c]{$\bullet$}}
					\put(8.1,4.1){\makebox(0,0)[c]{$\circ$}}
					\put(3.7,2.1){\line(1,0){1.3}}
					\put(6,2.1){\line(1,0){1.3}}
					\put(8.7,2.1){\line(1,0){1.3}}
					\put(11.2,2.1){\line(1,0){1.3}}
					\put(7.7,3.9){\line(-3,-1){4.3}}
					\put(8.5,3.9){\line(3,-1){4}}

					\put(8,2.05){\makebox(0,0)[c]{$\cdots$}}
					\put(8,3.2){\makebox(0,0)[c]{\tiny $\alpha_0$}}
					\put(3.2,1.1){\makebox(0,0)[c]{\tiny ${\alpha}_1$}}
					\put(5.6,1.1){\makebox(0,0)[c]{\tiny ${\alpha}_2$}}
					\put(10.4,1.1){\makebox(0,0)[c]{\tiny ${\alpha}_{n-1}$}}
					\put(13,1.1){\makebox(0,0)[c]{\tiny ${\alpha}_{n}$}}
					%
                    % Kac label
                    \put(3,2.6){\makebox(0,0)[c]{\tiny $1$}}
					\put(5.6,2.6){\makebox(0,0)[c]{\tiny $1$}}
					\put(10.5,2.6){\makebox(0,0)[c]{\tiny $1$}}
					\put(13,2.6){\makebox(0,0)[c]{\tiny $1$}}
					\put(8.1,4.6){\makebox(0,0)[c]{\tiny $1$}}
				\end{picture}
			\end{minipage} & $\underset{(n \ge 2)}{A_{2n}^{(2)}}$
            & \begin{minipage}[c][0.06\textheight]{0.5\linewidth}
			\setlength{\unitlength}{0.17in}
			\begin{picture}(16,3.7)
					\put(3,2.1){\makebox(0,0)[c]{$\circ$}}
					\put(5.6,2.1){\makebox(0,0)[c]{$\circ$}}
					\put(10.5,2.1){\makebox(0,0)[c]{$\circ$}}
					\put(13,2.1){\makebox(0,0)[c]{$\circ$}}
					\put(3.7,2){\line(1,0){1.3}}
					\put(3.7,2.2){\line(1,0){1.3}}
					\put(6,2.1){\line(1,0){1.3}}
					\put(8.7,2.1){\line(1,0){1.3}}
					\put(11.2,2){\line(1,0){1.3}}
					\put(11.2,2.2){\line(1,0){1.3}}
					%\put(11.5,1.88){\Large $\prec$}
                        \put(11.5,1.88){\Large $\succ$}
					%\put(4,1.88){\Large $\prec$}
                        \put(4,1.88){\Large $\succ$}

					\put(8,2.05){\makebox(0,0)[c]{$\cdots$}}
					\put(3.2,1.1){\makebox(0,0)[c]{\tiny ${\alpha}_0$}}
					\put(5.6,1.1){\makebox(0,0)[c]{\tiny ${\alpha}_1$}}
					\put(10.4,1.1){\makebox(0,0)[c]{\tiny ${\alpha}_{n-1}$}}
					\put(13,1.1){\makebox(0,0)[c]{\tiny ${\alpha}_{n}$}}
					%
					% Kac label (ge 2)
					\put(3,2.8){\makebox(0,0)[c]{\tiny $1$}}
                    \put(5.6,2.8){\makebox(0,0)[c]{\tiny $2$}}
					\put(10.4,2.8){\makebox(0,0)[c]{\tiny $2$}}
                    \put(13,2.8){\makebox(0,0)[c]{\tiny $2$}}
				\end{picture}
			\end{minipage}
            \\ %$1, 2, \dots, n$  &  $1, 2, \dots, n$ \\
			\hline
			$\underset{(n \ge 3)}{B_n^{(1)}}$ & 
			\begin{minipage}[c][0.10\textheight]{0.5\linewidth}
			\setlength{\unitlength}{0.17in}
				\begin{picture}(16,3.7)
					\put(2.8,0.7){\makebox(0,0)[c]{$\bullet$}}
					\put(2.8,3.3){\makebox(0,0)[c]{$\circ$}}
					\put(5.6,2.1){\makebox(0,0)[c]{$\circ$}}
					\put(10.5,2.1){\makebox(0,0)[c]{$\circ$}}
					\put(13,2.1){\makebox(0,0)[c]{$\circ$}}

					\put(5.3,1.8){\line(-2,-1){2.1}}
					\put(5.3,2.2){\line(-2,1){2.1}}
					\put(6,2.1){\line(1,0){1.3}}
					\put(8.7,2.1){\line(1,0){1.3}}
					\put(11.2,2){\line(1,0){1.3}}
					\put(11.2,2.2){\line(1,0){1.3}}
					\put(11.5,1.88){\Large $\succ$}

					\put(8,2.05){\makebox(0,0)[c]{$\cdots$}}
					\put(2.8,2.5){\makebox(0,0)[c]{\tiny ${\alpha}_0$}}
					\put(2.8,0){\makebox(0,0)[c]{\tiny ${\alpha}_1$}}
					\put(5.6,1.1){\makebox(0,0)[c]{\tiny ${\alpha}_2$}}
					\put(10.4,1.1){\makebox(0,0)[c]{\tiny ${\alpha}_{n-1}$}}
					\put(13,1.1){\makebox(0,0)[c]{\tiny ${\alpha}_{n}$}}
					%
					% Kac label
                    \put(2.8,1.4){\makebox(0,0)[c]{\tiny $1$}}
					\put(2.8,4){\makebox(0,0)[c]{\tiny $1$}}
					\put(5.6,2.8){\makebox(0,0)[c]{\tiny $2$}}
					\put(10.4,2.8){\makebox(0,0)[c]{\tiny $2$}}
					\put(13,2.8){\makebox(0,0)[c]{\tiny $2$}}
				\end{picture}
			\end{minipage} & 
			$\underset{(n \ge 3)}{A_{2n+1}^{(2)}}$ & 
			\begin{minipage}[c][0.10\textheight]{0.5\linewidth}
			\setlength{\unitlength}{0.17in}
				\begin{picture}(16,3.7)
					\put(2.8,0.7){\makebox(0,0)[c]{$\bullet$}}
					\put(2.8,3.3){\makebox(0,0)[c]{$\circ$}}
					\put(5.6,2.1){\makebox(0,0)[c]{$\circ$}}
					\put(10.5,2.1){\makebox(0,0)[c]{$\circ$}}
					\put(13,2.1){\makebox(0,0)[c]{$\circ$}}

					\put(5.3,1.8){\line(-2,-1){2.1}}
					\put(5.3,2.2){\line(-2,1){2.1}}
					\put(6,2.1){\line(1,0){1.3}}
					\put(8.7,2.1){\line(1,0){1.3}}
					\put(11.2,2){\line(1,0){1.3}}
					\put(11.2,2.2){\line(1,0){1.3}}
					\put(11.5,1.88){\Large $\prec$}

					\put(8,2.05){\makebox(0,0)[c]{$\cdots$}}
					\put(2.8,2.5){\makebox(0,0)[c]{\tiny ${\alpha}_0$}}
					\put(2.8,0){\makebox(0,0)[c]{\tiny ${\alpha}_1$}}
					\put(5.6,1.1){\makebox(0,0)[c]{\tiny ${\alpha}_2$}}
					\put(10.4,1.1){\makebox(0,0)[c]{\tiny ${\alpha}_{n-1}$}}
					\put(13,1.1){\makebox(0,0)[c]{\tiny ${\alpha}_{n}$}}
					%
					% Kac label
                    \put(2.8,1.4){\makebox(0,0)[c]{\tiny $1$}}
					\put(2.8,4){\makebox(0,0)[c]{\tiny $1$}}
					\put(5.6,2.8){\makebox(0,0)[c]{\tiny $2$}}
					\put(10.4,2.8){\makebox(0,0)[c]{\tiny $2$}}
					\put(13,2.8){\makebox(0,0)[c]{\tiny $1$}}
				\end{picture}
			\end{minipage}
			\\ %$1$ & $n$ \\
			\hline
			$\underset{(n \ge 2)}{C_n^{(1)}}$ & 
			\begin{minipage}[c][0.06\textheight]{0.5\linewidth}
			\setlength{\unitlength}{0.17in}
			\begin{picture}(16,3.7)
					\put(3,2.1){\makebox(0,0)[c]{$\circ$}}
					\put(5.6,2.1){\makebox(0,0)[c]{$\circ$}}
					\put(10.5,2.1){\makebox(0,0)[c]{$\circ$}}
					\put(13,2.1){\makebox(0,0)[c]{$\bullet$}}
					\put(3.7,2){\line(1,0){1.3}}
					\put(3.7,2.2){\line(1,0){1.3}}
					\put(6,2.1){\line(1,0){1.3}}
					\put(8.7,2.1){\line(1,0){1.3}}
					\put(11.2,2){\line(1,0){1.3}}
					\put(11.2,2.2){\line(1,0){1.3}}
					\put(11.5,1.88){\Large $\prec$}
					\put(4,1.88){\Large $\succ$}

					\put(8,2.05){\makebox(0,0)[c]{$\cdots$}}
					\put(3.2,1.1){\makebox(0,0)[c]{\tiny ${\alpha}_0$}}
					\put(5.6,1.1){\makebox(0,0)[c]{\tiny ${\alpha}_1$}}
					\put(10.4,1.1){\makebox(0,0)[c]{\tiny ${\alpha}_{n-1}$}}
					\put(13,1.1){\makebox(0,0)[c]{\tiny ${\alpha}_{n}$}}
					%
					% Kac label
                    \put(3,2.8){\makebox(0,0)[c]{\tiny $1$}}
					\put(5.6,2.8){\makebox(0,0)[c]{\tiny $2$}}
					\put(10.4,2.8){\makebox(0,0)[c]{\tiny $2$}}
                    \put(13,2.8){\makebox(0,0)[c]{\tiny $1$}}
				\end{picture}
			\end{minipage} & $\underset{(n \ge 2)}{D_{n+1}^{(2)}}$ &
			\begin{minipage}[c][0.06\textheight]{0.5\linewidth}
			\setlength{\unitlength}{0.17in}
			\begin{picture}(16,3.7)
					\put(3,2.1){\makebox(0,0)[c]{$\circ$}}
					\put(5.6,2.1){\makebox(0,0)[c]{$\circ$}}
					\put(10.5,2.1){\makebox(0,0)[c]{$\circ$}}
					\put(13,2.1){\makebox(0,0)[c]{$\bullet$}}
					\put(3.7,2){\line(1,0){1.3}}
					\put(3.7,2.2){\line(1,0){1.3}}
					\put(6,2.1){\line(1,0){1.3}}
					\put(8.7,2.1){\line(1,0){1.3}}
					\put(11.2,2){\line(1,0){1.3}}
					\put(11.2,2.2){\line(1,0){1.3}}
					\put(11.5,1.88){\Large $\succ$}
					\put(4,1.88){\Large $\prec$}

					\put(8,2.05){\makebox(0,0)[c]{$\cdots$}}
					\put(3.2,1.1){\makebox(0,0)[c]{\tiny ${\alpha}_0$}}
					\put(5.6,1.1){\makebox(0,0)[c]{\tiny ${\alpha}_1$}}
					\put(10.4,1.1){\makebox(0,0)[c]{\tiny ${\alpha}_{n-1}$}}
					\put(13,1.1){\makebox(0,0)[c]{\tiny ${\alpha}_{n}$}}
					%
					% Kac label
					\put(3,2.8){\makebox(0,0)[c]{\tiny $1$}}
					\put(5.6,2.8){\makebox(0,0)[c]{\tiny $1$}}
					\put(10.5,2.8){\makebox(0,0)[c]{\tiny $1$}}
					\put(13,2.8){\makebox(0,0)[c]{\tiny $1$}}
				\end{picture}
			\end{minipage} \\ %$n$ & $1$  \\
			\hline
			$\underset{(n \ge 4)}{D_n^{(1)}}$ &
			\begin{minipage}[c][0.11\textheight]{0.5\linewidth}
				\setlength{\unitlength}{0.17in}
				\begin{picture}(16,3.7)
					\put(2.8,0.7){\makebox(0,0)[c]{$\bullet$}}
					\put(2.8,3.3){\makebox(0,0)[c]{$\circ$}}
					\put(5.6,2){\makebox(0,0)[c]{$\circ$}}
					\put(10.4,2){\makebox(0,0)[c]{$\circ$}}
					\put(13.1,3.3){\makebox(0,0)[c]{$\bullet$}}
					\put(13.1,0.7){\makebox(0,0)[c]{$\bullet$}}
					\put(5.3,1.8){\line(-2,-1){2.1}}
					\put(5.3,2.2){\line(-2,1){2.1}}
					\put(6,2){\line(1,0){1.3}}
					\put(8.7,2){\line(1,0){1.3}}
					\put(10.7,2.2){\line(2,1){2}}
					\put(10.7,1.8){\line(2,-1){2}}

					\put(8,1.95){\makebox(0,0)[c]{$\cdots$}}
					\put(2.8,2.5){\makebox(0,0)[c]{\tiny ${\alpha}_0$}}
					\put(2.8,0.0){\makebox(0,0)[c]{\tiny ${\alpha}_1$}}
					\put(5.6,1){\makebox(0,0)[c]{\tiny ${\alpha}_2$}}
					\put(10.4,1){\makebox(0,0)[c]{\tiny ${\alpha}_{n-2}$}}
					\put(13.1,2.5){\makebox(0,0)[c]{\tiny ${\alpha}_{n-1}$}}
					\put(13.1,0.0){\makebox(0,0)[c]{\tiny ${\alpha}_{n}$}}
					%
					% Kac label (ge 2)
                    \put(2.8,1.4){\makebox(0,0)[c]{\tiny $1$}}
					\put(2.8,4){\makebox(0,0)[c]{\tiny $1$}}
					\put(5.6,2.7){\makebox(0,0)[c]{\tiny $2$}}
					\put(10.4,2.7){\makebox(0,0)[c]{\tiny $2$}}
                    \put(13.1,4){\makebox(0,0)[c]{\tiny $1$}}
					\put(13.1,1.4){\makebox(0,0)[c]{\tiny $1$}}
				\end{picture}
			\end{minipage} & $E_6^{(1)}$ & 
			\begin{minipage}[c][0.15\textheight]{0.5\linewidth}
			\setlength{\unitlength}{0.17in}
				\begin{picture}(16,8.5)
					\put(3,2.1){\makebox(0,0)[c]{$\bullet$}}
					\put(5.6,2.1){\makebox(0,0)[c]{$\circ$}}
					\put(8,2.1){\makebox(0,0)[c]{$\circ$}}
					\put(10.5,2.1){\makebox(0,0)[c]{$\circ$}}
					\put(13, 2.1){\makebox(0,0)[c]{$\bullet$}}
					\put(8,4.5){\makebox(0,0)[c]{$\circ$}}
					\put(8,7){\makebox(0,0)[c]{$\circ$}}

					\put(8, 4){\line(0, -3){1.3}}
					\put(8, 6.5){\line(0, -3){1.3}}
					\put(3.7,2.1){\line(1,0){1.3}}
					\put(6,2.1){\line(1,0){1.3}}
					\put(8.7,2.1){\line(1,0){1.3}}
					\put(11.2,2.1){\line(1,0){1.3}}
					\put(8.7,7){\makebox(0,0)[c]{\tiny $\alpha_0$}}
					\put(3.2,1.1){\makebox(0,0)[c]{\tiny ${\alpha}_1$}}
					\put(8.7,4.5){\makebox(0,0)[c]{\tiny $\alpha_2$}}
					\put(5.6,1.1){\makebox(0,0)[c]{\tiny ${\alpha}_3$}}
					\put(8,1.1){\makebox(0,0)[c]{\tiny $\alpha_4$}}
					\put(10.4,1.1){\makebox(0,0)[c]{\tiny ${\alpha}_5$}}
					\put(13,1.1){\makebox(0,0)[c]{\tiny ${\alpha}_6$}}
					%
					%Kac labels
                    \put(3,2.7){\makebox(0,0)[c]{\tiny $1$}}
					\put(7.4,4.5){\makebox(0,0)[c]{\tiny $2$}}
					\put(5.6,2.7){\makebox(0,0)[c]{\tiny $2$}}
					\put(8.5,2.7){\makebox(0,0)[c]{\tiny $3$}}
					\put(10.4,2.7){\makebox(0,0)[c]{\tiny $2$}}
                    \put(7.4,7){\makebox(0,0)[c]{\tiny $1$}}
                    \put(13, 2.7){\makebox(0,0)[c]{\tiny $1$}}
				\end{picture}
			\end{minipage} \\ %$1,\,6$ & $1,\,6$ \\  \\ %$1,\,n-1,\,n$ & $1,\,n-1,\,n$ \\
			\hline
			$E_7^{(1)}$ & 
			\begin{minipage}[c][0.10\textheight]{0.5\linewidth}
			\setlength{\unitlength}{0.17in}
				\begin{picture}(16,5.5)
					\put(0.7,2.1){\makebox(0,0)[c]{$\circ$}}
					\put(3,2.1){\makebox(0,0)[c]{$\circ$}}
					\put(5.6,2.1){\makebox(0,0)[c]{$\circ$}}
					\put(8,2.1){\makebox(0,0)[c]{$\circ$}}
					\put(10.5,2.1){\makebox(0,0)[c]{$\circ$}}
					\put(13, 2.1){\makebox(0,0)[c]{$\circ$}}
					\put(15.5,2.1){\makebox(0,0)[c]{$\bullet$}}
					\put(8,4.5){\makebox(0,0)[c]{$\circ$}}
					\put(8, 4){\line(0, -3){1.3}}
					\put(1.2,2.1){\line(1,0){1.3}}
					\put(3.7,2.1){\line(1,0){1.3}}
					\put(6,2.1){\line(1,0){1.3}}
					\put(8.7,2.1){\line(1,0){1.3}}
					\put(11.2,2.1){\line(1,0){1.3}}
					\put(13.7,2.1){\line(1,0){1.3}}
					\put(0.8,1.1){\makebox(0,0)[c]{\tiny ${\alpha}_0$}}
					\put(3.2,1.1){\makebox(0,0)[c]{\tiny ${\alpha}_1$}}
					\put(8.7,4.5){\makebox(0,0)[c]{\tiny $\alpha_2$}}
					\put(5.6,1.1){\makebox(0,0)[c]{\tiny ${\alpha}_3$}}
					\put(8,1.1){\makebox(0,0)[c]{\tiny $\alpha_4$}}
					\put(10.4,1.1){\makebox(0,0)[c]{\tiny ${\alpha}_5$}}
					\put(13,1.1){\makebox(0,0)[c]{\tiny ${\alpha}_6$}}
					\put(15.5, 1.1){\makebox(0,0)[c]{\tiny ${\alpha}_7$}}
					%
					%Kac labels
                    \put(0.7,2.7){\makebox(0,0)[c]{\tiny $1$}}
					\put(3.2,2.7){\makebox(0,0)[c]{\tiny $2$}}
					\put(7.5,4.5){\makebox(0,0)[c]{\tiny $2$}}
					\put(5.6,2.7){\makebox(0,0)[c]{\tiny $3$}}
					\put(8.5,2.7){\makebox(0,0)[c]{\tiny $4$}}
					\put(10.4,2.7){\makebox(0,0)[c]{\tiny $3$}}
					\put(13,2.7){\makebox(0,0)[c]{\tiny $2$}}
                    \put(15.5,2.7){\makebox(0,0)[c]{\tiny $1$}}
				\end{picture}
			\end{minipage} & $E_8^{(1)}$ & 
			\begin{minipage}[c][0.10\textheight]{0.5\linewidth}
			\setlength{\unitlength}{0.15in}
				\begin{picture}(16,5.5)
					\put(0.7,2.1){\makebox(0,0)[c]{$\circ$}}
					\put(3,2.1){\makebox(0,0)[c]{$\circ$}}
					\put(5.6,2.1){\makebox(0,0)[c]{$\circ$}}
					\put(8,2.1){\makebox(0,0)[c]{$\circ$}}
					\put(10.5,2.1){\makebox(0,0)[c]{$\circ$}}
					\put(13, 2.1){\makebox(0,0)[c]{$\circ$}}
					\put(15.5,2.1){\makebox(0,0)[c]{$\circ$}}
					\put(18,2.1){\makebox(0,0)[c]{$\circ$}}
					\put(13,4.5){\makebox(0,0)[c]{$\circ$}}
					\put(13,4){\line(0, -3){1.3}}
					\put(1.2,2.1){\line(1,0){1.3}}
					\put(3.7,2.1){\line(1,0){1.3}}
					\put(6,2.1){\line(1,0){1.3}}
					\put(8.7,2.1){\line(1,0){1.3}}
					\put(11.2,2.1){\line(1,0){1.3}}
					\put(13.7,2.1){\line(1,0){1.3}}
					\put(16.2,2.1){\line(1,0){1.3}}
					\put(0.8,1.1){\makebox(0,0)[c]{\tiny ${\alpha}_0$}}
					\put(3.2,1.1){\makebox(0,0)[c]{\tiny ${\alpha}_1$}}
					\put(13.8,4.5){\makebox(0,0)[c]{\tiny $\alpha_8$}}
					\put(5.6,1.1){\makebox(0,0)[c]{\tiny ${\alpha}_2$}}
					\put(8,1.1){\makebox(0,0)[c]{\tiny $\alpha_3$}}
					\put(10.4,1.1){\makebox(0,0)[c]{\tiny ${\alpha}_4$}}
					\put(13,1.1){\makebox(0,0)[c]{\tiny ${\alpha}_5$}}
					\put(15.5, 1.1){\makebox(0,0)[c]{\tiny ${\alpha}_6$}}
					\put(18, 1.1){\makebox(0,0)[c]{\tiny ${\alpha}_7$}}
					%
					% Kac labels
					\put(0.8,2.7){\makebox(0,0)[c]{\tiny $1$}}
					\put(3.2,2.7){\makebox(0,0)[c]{\tiny $2$}}
                    \put(5.6,2.7){\makebox(0,0)[c]{\tiny $3$}}
                    \put(8,2.7){\makebox(0,0)[c]{\tiny $4$}}
                    \put(10.4,2.7){\makebox(0,0)[c]{\tiny $5$}}
                    \put(13.4,2.7){\makebox(0,0)[c]{\tiny $6$}}
					\put(12.4,4.5){\makebox(0,0)[c]{\tiny $3$}}
					\put(15.5,2.7){\makebox(0,0)[c]{\tiny $4$}}
                    \put(18,2.7){\makebox(0,0)[c]{\tiny $2$}}
				\end{picture}
			\end{minipage}  \\ %$7$ & $7$ \\
			\hline
			$F_4^{(1)}$ & 
			\begin{minipage}[c][0.06\textheight]{0.5\linewidth}
			\setlength{\unitlength}{0.17in}
			\begin{picture}(16,3.9)
					\put(4,2.1){\makebox(0,0)[c]{$\circ$}}
					\put(6,2.1){\makebox(0,0)[c]{$\circ$}}
					\put(8,2.1){\makebox(0,0)[c]{$\circ$}}
					\put(10,2.1){\makebox(0,0)[c]{$\circ$}}
					\put(12,2.1){\makebox(0,0)[c]{$\circ$}}
					\put(4.3,2.1){\line(1,0){1.3}}
					\put(6.3,2.1){\line(1,0){1.3}}
					\put(8.3,2.0){\line(1,0){1.3}}
					\put(8.3,2.2){\line(1,0){1.3}}
					\put(10.3,2.1){\line(1,0){1.3}}
					\put(8.6,1.88){\Large $\succ$}
					\put(4,1.3){\makebox(0,0)[c]{\tiny ${\alpha}_0$}}
					\put(6,1.3){\makebox(0,0)[c]{\tiny ${\alpha}_1$}}
					\put(8,1.3){\makebox(0,0)[c]{\tiny ${\alpha}_2$}}
					\put(10,1.3){\makebox(0,0)[c]{\tiny ${\alpha}_3$}}
					\put(12,1.3){\makebox(0,0)[c]{\tiny ${\alpha}_4$}}
					%
					% Kac labels
					\put(4,2.7){\makebox(0,0)[c]{\tiny $1$}}
                    \put(6,2.7){\makebox(0,0)[c]{\tiny $2$}}
					\put(8,2.7){\makebox(0,0)[c]{\tiny $3$}}
					\put(10,2.7){\makebox(0,0)[c]{\tiny $4$}}
					\put(12,2.7){\makebox(0,0)[c]{\tiny $2$}}
				\end{picture}
			\end{minipage} & $E_6^{(2)}$ &
			\begin{minipage}[c][0.06\textheight]{0.5\linewidth}
			\setlength{\unitlength}{0.17in}
			\begin{picture}(16,3.9)
					\put(4,2.1){\makebox(0,0)[c]{$\circ$}}
					\put(6,2.1){\makebox(0,0)[c]{$\circ$}}
					\put(8,2.1){\makebox(0,0)[c]{$\circ$}}
					\put(10,2.1){\makebox(0,0)[c]{$\circ$}}
					\put(12,2.1){\makebox(0,0)[c]{$\circ$}}
					\put(4.3,2.1){\line(1,0){1.3}}
					\put(6.3,2.1){\line(1,0){1.3}}
					\put(8.3,2.0){\line(1,0){1.3}}
					\put(8.3,2.2){\line(1,0){1.3}}
					\put(10.3,2.1){\line(1,0){1.3}}
					\put(8.6,1.88){\Large $\prec$}
					\put(4,1.3){\makebox(0,0)[c]{\tiny ${\alpha}_0$}}
					\put(6,1.3){\makebox(0,0)[c]{\tiny ${\alpha}_1$}}
					\put(8,1.3){\makebox(0,0)[c]{\tiny ${\alpha}_2$}}
					\put(10,1.3){\makebox(0,0)[c]{\tiny ${\alpha}_3$}}
					\put(12,1.3){\makebox(0,0)[c]{\tiny ${\alpha}_4$}}
					%
					% Kac labels
                    \put(4,2.7){\makebox(0,0)[c]{\tiny $1$}}
					\put(6,2.7){\makebox(0,0)[c]{\tiny $2$}}
					\put(8,2.7){\makebox(0,0)[c]{\tiny $3$}}
					\put(10,2.7){\makebox(0,0)[c]{\tiny $2$}}
					\put(12,2.7){\makebox(0,0)[c]{\tiny $1$}}
				\end{picture}
			\end{minipage}
			 \\ % {\rm none} & {\rm none}  \\
			\hline
			$G_2^{(1)}$ & 
			\begin{minipage}[c][0.06\textheight]{0.5\linewidth}
			\setlength{\unitlength}{0.17in}
			\begin{picture}(16,3.7)
					\put(6,2.1){\makebox(0,0)[c]{$\circ$}}
					\put(8,2.1){\makebox(0,0)[c]{$\circ$}}
					\put(10,2.1){\makebox(0,0)[c]{$\circ$}}
					\put(6.3,2.1){\line(1,0){1.3}}
					\put(8.3,1.9){\line(1,0){1.3}}
					\put(8.3,2.1){\line(1,0){1.3}}
					\put(8.3,2.3){\line(1,0){1.3}}
					\put(8.5,1.84){\LARGE $\succ$}
					\put(6,1.3){\makebox(0,0)[c]{\tiny ${\alpha}_0$}}
					\put(8,1.3){\makebox(0,0)[c]{\tiny ${\alpha}_1$}}
					\put(10,1.3){\makebox(0,0)[c]{\tiny ${\alpha}_2$}}
					%
					% Kac labels
                    \put(6,2.7){\makebox(0,0)[c]{\tiny $1$}}
					\put(8,2.7){\makebox(0,0)[c]{\tiny $2$}}
					\put(10,2.7){\makebox(0,0)[c]{\tiny $3$}}
				\end{picture}
			\end{minipage} & $D_4^{(3)}$ & 
			\begin{minipage}[c][0.06\textheight]{0.5\linewidth}
			\setlength{\unitlength}{0.17in}
			\begin{picture}(16,3.7)
					\put(6,2.1){\makebox(0,0)[c]{$\circ$}}
					\put(8,2.1){\makebox(0,0)[c]{$\circ$}}
					\put(10,2.1){\makebox(0,0)[c]{$\circ$}}
					\put(6.3,2.1){\line(1,0){1.3}}
					\put(8.3,1.9){\line(1,0){1.3}}
					\put(8.3,2.1){\line(1,0){1.3}}
					\put(8.3,2.3){\line(1,0){1.3}}
					\put(8.5,1.84){\LARGE $\prec$}
					\put(6,1.3){\makebox(0,0)[c]{\tiny ${\alpha}_0$}}
					\put(8,1.3){\makebox(0,0)[c]{\tiny ${\alpha}_1$}}
					\put(10,1.3){\makebox(0,0)[c]{\tiny ${\alpha}_2$}}
					%
					% Kac labels
                    \put(6,2.7){\makebox(0,0)[c]{\tiny $1$}}
					\put(8,2.7){\makebox(0,0)[c]{\tiny $2$}}
					\put(10,2.7){\makebox(0,0)[c]{\tiny $1$}}
				\end{picture}
			\end{minipage}
			\\ % {\rm none} & {\rm none}  \\
			\hline
		\end{tabular}
	}
	\caption{The affine Dynkin diagrams with special nodes marked as $\bullet$.}
	\label{tab:dynkins}
\end{table}

For any $J \subseteq I$, we define the Levi subalgebra $\aff{\g}_J$ as being the subalgebra whose Dynkin diagram is induced from that of $\aff{\g}$ with index set $J$.
Note that $\aff{\g}_J$ is a finite-dimensional semisimple Lie algebra for any $J \neq I$.
Let $\aff{\Phi}_J$ denote the corresponding root system with simple roots $\aff{\Delta}_J := \{\alpha_i \mid i \in J \}$.
For $J \neq I$, let $w_J$ denote the longest element of the corresponding Weyl group $\aff{W}_J$.
For $w \in \aff{W}$, the \defn{inversion set} is $\Inv(w) = \{\beta \in \Phi^+ \mid w \beta \in \Phi^-\}$, and the \defn{length} of $w$ is $\ell(w) := \abs{\Inv(w)}$.
Let $R(w) := \{(i_1, \dotsc, i_{\ell(w)}) \in I^{\ell(w)} \mid s_{i_1} \cdots s_{i_{\ell(w)}} = w \}$ denote the set of \defn{reduced expressions} for $w$ (see, \textit{e.g.},~\cite{Hum90} for more information).

When $J = I_0 := I \setminus \{0\}$, we call the corresponding Lie algebra the \defn{classical Lie algebra} and denote it by $\g := \aff{\g}_{I_0}$.
Let $\Phi$, $P$, $Q$, and $W$ denote the corresponding root system, weight lattice, root lattice, and Weyl group, respectively.
The corresponding fundamental weights will be the \defn{classical fundamental weights} $\{\clfw_i\}_{i \in I_0}$.
Let $\theta$ (resp.~$\theta_s$) denote the highest (resp.~short) root of $\Phi$.
Define
\begin{equation} \label{eq: def of Theta}
\Theta = \begin{cases}
\theta & \text{if $\aff{\g}$ is of untwisted type,} \\
\theta_s & \text{otherwise.}
\end{cases}
\end{equation}
Let $w_0$ denote the longest element of the corresponding Weyl group $W$.
For $J \subseteq I_0$, we set $\g_J := \aff{\g}_J$, $W_J := \aff{W}_J$, etc.

\begin{rem}
From this point forward, we will always denote the corresponding dual untwisted affine Lie algebra to $\g$ by $\aff{\g}$, and we will denote the corresponding untwisted affine Lie algebra by $\gutw$.
\end{rem}

\begin{table}[t]
	\[
	\begin{array}{cccccccc}
	\toprule
	\g & A_n & B_n & C_n & D_n & E_n & F_4 & G_2
	\\\midrule
	\aff{\g} & A_n^{(1)} & D_{n+1}^{(2)} & A_{2n+1}^{(2)} & D_n^{(1)} & E_n^{(1)} & E_6^{(2)} & D_4^{(3)}
	\\\bottomrule
	\end{array}
	\]
	\caption{The finite types and their corresponding dual untwisted type.}
	\label{table:corresponding_affine}
\end{table}

There are only two nodes that are cominuscule but not minuscule: $r = 1$ in type $B_n$ and $r = n$ in type $C_n$.
Indeed, these come from the orbit of the $0$ node in the corresponding \emph{un}twisted affine types $B_n^{(1)}$ and $C_n^{(1)}$, respectively.
Likewise, the minuscule node in type $B_n$ and type $C_n$ are not cominuscule.

Next, we consider $J = I_0 \setminus \{r\}$ for some $r \in I_0$.
The stabilizer of $\clfw_r$ is given by the parabolic subgroup $W_J$.
Let $W^J$ denote the set of minimal length coset representatives of $W / W_J$. 
Take the maximal element in $W / W_J$, which is unique (see, \textit{e.g.},~\cite{Hum90}), and 
we denote $w^J := \cosetrepr{w_0}$.
Note that the inversion set $\Inv(w^J) = (\Phi^J)^+ := \Phi^+ \setminus \Phi_J$ since $w_0 = w^J w_J$, where $w_J$ is the longest element of $W_J$.

\begin{rem} \label{rem: w^J}
When $r$ is a minuscule node, any minimal length coset representative $\cosetrepr{w} \in W^J$ is \defn{fully-commutative}~\cite[Prop.~2.1]{Stembridge01II} (see also~\cite{Littelmann98}), which means the reduced expression is unique up to commutation relations (also known as $2$-term braid moves).
Furthermore, we can write $w^J = s_{i_1} \cdots s_{i_\ell}$ such that $s_{i_{\ell}} = s_r$ and $s_{i_1} = s_{\tau^{-1}(0)}$.
Here $\tau$ is an affine Dynkin diagram automorphism obtained from a translation by fundamental $r$-th coweight.
%In~\cite{JK19, Jang22}, the authors give explicit reduced expressions for $w^J$ and $w_J$, which are unique up to $2$-braid moves in the sense of \cite[Rem.~5.2]{JK21}.
%By using \cite[Thm.~4.5]{SST18} with the reduced expression of $w^J$ (see, \textit{e.g.},~\cite[Prop.~3.2]{Jang22}), one can describe the crystal operators explicitly.
%For example, see~\cite[Prop.~3.2]{JK19} for type $D_n^{(1)}$ and \cite[Prop.~4.3]{Jang22} for type $E_n^{(1)}$.
%Note that the reduced expressions of $w_0 = w^J w_J$ are not necessarily equal to the ones in~\cite{SST18} (\textit{cf.}~\cite[Rem.~5.4]{JK21}; see also~\eqref{eq:rex for KR in types BC}).
\end{rem}

When $\g$ has at least one minuscule node, we can order $I_0 = (i_1, \dotsc, i_n)$ such that $i_k$ is a minuscule node of the subdiagram with nodes $I_{(k)} := (i_k, \dotsc, i_n)$ (\textit{i.e.}, for the corresponding Levi subalgebra $\g_{I_{(k)}}$). 
Subsequently, we can use this to form the reduced decompositions that are the ``nice decompositions'' given by Littelmann~\cite{Littelmann98} by
\begin{equation}
\label{eq:nice_decomp}
w_0 = w^{(1)} w^{(2)} \dotsm w^{(n)},
\end{equation}
where $w^{(k)}$ is any reduced expression for the minimal length coset representative of the longest element of $W_{I_{(k)}}$ modulo $W_{I_{(k+1)}}$, where $W_{I_{(n+1)}} = \mathbf{1}$ is the trivial group (for example, $w^{(1)} = \cosetrepr{w_0}$).

\begin{ex}
\label{ex:nice_decomposition}
Let $\g$ be of type $E_7$.
Then $w_0 = w^{(1)} w^{(2)} w^{(3)} w^{(4)} w^{(5)} w^{(6)} w^{(7)}$ with
\begin{align*}
w^{(1)} & = s_7 s_6 s_5 s_4 s_3 s_2 s_4 s_5 s_6 s_7 s_1 s_3 s_4 s_5 s_6 s_2 s_4 s_5 s_3 s_4 s_1 s_3 s_2 s_4 s_5 s_6 s_7,
\\ w^{(2)} & = s_1 s_3 s_4 s_5 s_6 s_2 s_4 s_5 s_3 s_4 s_1 s_3 s_2 s_4 s_5 s_6,
\\ w^{(3)} & = s_2 s_4 s_5 s_3 s_4 s_2 s_1 s_3 s_4 s_5,
\\ w^{(4)} & = s_1 s_3 s_4 s_2,
\qquad w^{(5)} = s_1 s_3 s_4,
\qquad w^{(6)} = s_1 s_3,
\qquad w^{(7)} = s_1.
\end{align*}
This corresponds to ordering the indices $I_0 = (7, 6, 5, 2, 4, 3, 1)$ and $w^{(k)} \dotsm w^{(7)}$ is a nice decomposition of the longest element for the groups (of type)
\[
(W_{I_{(k)}})_{k=1}^7 = (E_7, E_6, D_5, A_4, A_3, A_2, A_1)
\qquad
\begin{tikzpicture}[baseline=.4cm]
\foreach \i in {1,2,3}
  \node (\i) at (\i,0) {$\i$};
\foreach \i in {5,6,7}
  \node (\i) at (\i-1,0) {$\i$};
\node (4) at (3,1) {$4$};
\draw[-] (1) -- (2) -- (3) -- (5) -- (6) -- (7);
\draw[-] (3) -- (4);
\end{tikzpicture}
\]
(with some nodes relabeled from Table~\ref{tab:dynkins}), where the nodes of the Dynkin diagram are labeled according to the order they are removed.
\end{ex}

Define an automorphism of the Dynkin diagram by $\tau_r = t_{\clfw_r^{\vee}} (w^J)^{-1}$ on the simple roots (see, \textit{e.g.},~\cite{Bourbaki02}), where $t_{\clfw_r^{\vee}}$ is the translation by $\clfw_r^{\vee}$ in the extended affine Weyl group $W \ltimes P \iso \widehat{W} \ltimes \Pi$ with $\Pi$ being a particular set of Dynkin diagram automorphisms (\textit{cf.}~\cite{kac90}).
In particular, we have $\tau_r(0) = r$.

We will also require the following fact.
We believe it is well-known to experts (\textit{cf.}~\cite[Section 3.2]{JKP25} and references therein for more details), but the authors could not find an explicit reference in the literature.

\begin{prop} \label{prop:twisted_roots}
Let $\aff{\g}$ be an affine Lie algebra, and choose a special node $r$.
Let $J = I_0 \setminus \{r\}$, and recall $w^J = \cosetrepr{w_0} \in W / W_J$.
Then for all $i \in I_0$ we have
\[
w^J \left(  \alpha_i \right) = \begin{cases} \alpha_{\tau_r^{-1}(i)} & \text{if } i \neq r, \\ -\Theta & \text{if } i = r. \end{cases}
\]
\end{prop}

\begin{proof}
We first prove this when $\aff{\g}$ is an untwisted affine Lie algebra.
Since $w^J$ is a minimal length coset representative, we have the $(w^J)^{-1} \left( \alpha_i \right) \in \Phi^+$ for all $i \in J$ as otherwise $\ell(w^J s_i) < \ell(w^J)$.

Thus
\[
t_{\clfw_r^{\vee}} (w^J)^{-1} \left( \Delta \right) = \tau_r \left( \Delta \right) \subseteq \Delta \cup \{-\theta + \delta\},
\]
where $\Delta = \aff{\Delta}_{I_0}$.
Hence, we have
\[
(w^J)^{-1}(\Delta_J) \subseteq \proj(\Delta \cup \{-\theta + \delta\}),
\]
where $\proj$ is the projection onto the real roots.
Moreover, we have
\[
t_{\clfw_r^{\vee}} (w^J)^{-1} \left( \alpha_r \right) = -\theta + \delta
\qquad \Longrightarrow \qquad
(w^J)^{-1} \left( \alpha_r \right) = -\theta.
\]
by definition of $t_{\clfw_r^{\vee}}$.

Finally, when $\aff{\g}$ is of dual untwisted type, the result follows from interchanging long roots and short roots in the root system.
\end{proof}

\begin{ex} \label{ex: prop 4.1}
Let us consider the case of type $E_6^{(1)}$ with $r = 1$.
In this case, one reduced expression is
\begin{equation*}
	w^J = s_6 s_5 s_4 s_3 s_1 s_2 s_4 s_3 s_5 s_4 s_2 s_6 s_5 s_4 s_3 s_1
\end{equation*}
(\textit{cf.}~Remark \ref{rem: w^J}), and $\tau_1 \colon I \rightarrow I$ is given by 
\begin{equation*}
	\tau_1(0) = 1,\quad\!
	\tau_1(1) = 6,\quad\!
	\tau_1(2) = 3,\quad\!
	\tau_1(3) = 5,\quad\!
	\tau_1(4) = 4,\quad\!
	\tau_1(5) = 2,\quad\!
	\tau_1(6) = 0.
\end{equation*}
Thus $\tau_1^{-1}(0) = 6$ in this case, and we also have $\tau_1^{-1} = w^J t_{-\clfw_1^{\vee}}$.
Then one can check that $w^J(\alpha_i) = \alpha_{\tau_1^{-1}(i)}$ for $i \in J$; for example,
\begin{align*}
	w^J(\alpha_6) &= {\color{black!30} s_6 s_5 s_4} s_3 s_1 {\color{black!30} s_2} s_4 s_3 s_5 s_4 {\color{black!30} s_2} s_6 s_5 {\color{black!30} s_4 s_3 s_1}(\alpha_6) = \alpha_1 = \alpha_{\tau_1^{-1}(1)},
\end{align*}
where we apply the relation $s_is_j(\alpha_i) = \alpha_j$ (resp.\ $s_i(\alpha_j) = \alpha_j$) if $A_{ij} = -1$ (resp.\ $A_{ij} = 0$) at each such a pair marked as black color (resp.\ gray color).
%Furthermore, a direct computation shows $w^J(\alpha_1) = -\theta$.
\end{ex}

%===================
\section{Background on crystals} \label{sec:crystals}

\subsection{Crystals}
We recall the definition of a $U_q(\aff{\g})$-crystal~\cite{K90,K91}.
We also refer the reader to~\cite{BS17} for more information on crystals.

\begin{df}
\label{def:abstract_crystal}
An \defn{abstract $U_q(\aff{\g})$-crystal} is a set $B$ with, for each $i \in I$, \defn{crystal operators} $\et_i, \ft_i \colon B \to B \sqcup \{ \zero \}$, statistics $\varepsilon_i, \varphi_i \colon B \to \ZZ \sqcup \{-\infty\}$, and \defn{weight function} $\wt \colon B \to \aff{P}$ such that for all $b \in B$ and $i \in I$, the following holds
\begin{enumerate}
\item $\inner{\wt(b)}{\coroot_i} + \varepsilon_i(b) = \varphi_i(b)$;
\item $\wt(e_i b) = \wt(b) + \alpha_i$;
\item $\varepsilon_i(b) = -\infty$ implies $\et_i b = \zero$ and $\varphi_i(b) = -\infty$ implies $\ft_i b = \zero$;
\item $\et_i b = b'$ if and only if $b = \ft_i b'$ for some $b' \in B$;
\item $\varepsilon_i(\ft_i b) = \varepsilon_i(b) + 1$ if $\ft_i b \neq \zero$.
\end{enumerate}
We will also extend $\et_i \zero = \ft_i \zero = \zero$.
\end{df}

For two abstract $U_q(\aff{\g})$-crystals $B$ and $C$, a \defn{crystal morphism} $\psi \colon B \to C$ is a set-theoretic map $\psi \colon B \to C \sqcup \{\zero\}$ such that
\begin{enumerate}
\item commutes with all nonzero crystal operators: for all $i \in I$, we have $\psi(\et_i b) = \et_i \psi(b)$ when $\et_i b \in B$ and $\psi(b), \psi(\et_i b) \in C$ and similarly for $\ft_i$; and
\item preserves $\eta = \varepsilon_i, \varphi_i, \wt$ ($i \in I$): we have $\eta\bigl( \psi(b) \bigr) = \eta(b)$ whenever $\psi(b) \in C$.
\end{enumerate}
We call $\psi$ an \defn{embedding} (resp.\ \defn{isomorphism}) if it is an injective (resp.\ bijective) map of the sets $B \to C$ (specifically, $\zero$ is not in the image).
A morphism is called \defn{strict} if it commutes with $\et_i$ and $\ft_i$ for all $i \in I$.
When $B$ is isomorphic to $C$, we denote this by $B \iso C$.
We call $B$ a \defn{subcrystal} of $C$ if $B$ is a subset of $C$ such that the inclusion map is a crystal embedding, and we simply write this as $B \subseteq C$. 
We say $B$ is a \defn{full subcrystal} if the embedding $\psi$ is strict.
An abstract $U_q(\aff{\g})$-crystal $B$ is \defn{connected} if its only full subcrystal is $B$.

For $J \subseteq I$, there is a natural definition of a $U_q(\aff{\g}_J)$-crystal by restricting the index set to $J$.
As such, there are natural restriction functors from abstract $U_q(\aff{\g})$-crystals to $U_q(\aff{\g}_J)$-crystals.

An abstract $U_q(\aff{\g})$-crystal $B$ is an \defn{upper regular} (resp.\ \defn{lower regular}) crystal if
\[
\varepsilon_i(b) = \abs{\{ k \mid \et_i^k(b) \neq \zero \}} < \infty,
\qquad\qquad
\varphi_i(b) = \abs{\{ k \mid \ft_i^k(b) \neq \zero \}} < \infty,
\]
respectively, for all $b \in \mcB$.
This is sometimes referred to as an upper (resp.\ lower) seminormal crystal.
We say $B$ is \defn{regular} if it is upper and lower regular.
An element $b \in B$ is \defn{highest weight} (resp.\ \defn{lowest weight}) if $\et_i b = \zero$ (resp.\ $\ft_i b = \zero$) for all $i \in I$.
For any $J \subseteq I$, we say $b \in B$ is \defn{$J$-highest weight} (resp.\ \defn{$J$-lower weight}) if it is highest (resp.\ lowest) weight when $B$ is considered as a $U_q(\aff{\g}_J)$-crystal.

When an abstract $U_q(\aff{\g})$-crystal $B$ is one part of the crystal basis\footnote{A crystal basis is a pair $(L, B)$, where $L$ is the crystal lattice.}~\cite{K90,K91} for a $U_q(\aff{\g})$ with $\et_i, \ft_i$ the Kashiwara operators, we call $B$ a \defn{$U_q(\aff{\g})$-crystal}.
All of the $U_q(\aff{\g})$-crystals considered in this manuscript will be upper regular as they will come from the category $\mathcal{O}_q^{\mathrm{int}}$ of integrable $U_q(\aff{\g})$-modules or the negative half of $U_q(\aff{\g})$.

\begin{ex} \label{ex:crystal T}
An abstract $U_q(\aff{\g})$-crystal that is not a $U_q(\aff{\g})$-crystal is $T_{\lambda} = \{t_{\lambda}\}$, where $\lambda \in P$, whose crystal structure is given by
\[
\wt(t_{\lambda}) = \lambda,
\qquad\qquad
\varepsilon_i(t_{\lambda}) = \varphi_i(t_{\lambda}) = -\infty,
\]
for all $i \in I$.
Therefore, we must also have $\et_i t_{\lambda} = \ft_i t_{\lambda} = \zero$ for all $i \in I$.
\end{ex}

\begin{ex}
\label{ex:elem_crystal}
For $i \in I$, let $B_i := \{b_i(k) \mid k \in \ZZ\}$ denote the \defn{elementary $i$-crystal}, where the abstract $U_q(\aff{\g})$-crystal structure is given by $f_i b_i(k) = b_i(k-1)$ with $\wt(b_i(k)) = k\alpha_i$ and statistics $\varepsilon_i(b_i(k)) = -k$ and $\varepsilon_j\bigl(b_i(k)\bigr) = -\infty$ for all $j \neq i$.
By the crystal axioms, we have $\et_i b_i(k) = b_i(k+1)$ and $\et_j b_i(k) = \ft_j b_i(k) = \zero$ and the statistics $\varphi_i\bigl(b_i(k)\bigr) = k$ and $\varphi_j\bigl(b_i(k)\bigr) = -\infty$ for all $j \neq i$.
\end{ex}

Let $B$ and $C$ be two abstract $U_q(\aff{\g})$-crystals.
The direct sum of $B$ and $C$ is the abstract $U_q(\aff{\g})$-crystal $B \oplus C$ whose underlying set is the disjoint union $B \sqcup C$ with the natural inherited crystal structure from $B$ and $C$.
Hence, the category of (abstract) (upper/lower seminormal) $U_q(\aff{\g})$-crystals has many of the properties of an abelian category.
The tensor product of $B$ and $C$ is the abstract $U_q(\aff{\g})$-crystal $B \otimes C$ whose underlying set is the Cartesian product $B \times C$ with crystal operators
\[
\et_i(b\otimes c) := \begin{cases}
b \otimes (\et_i c) &\text{if }\varepsilon_i(b) \leq \varphi_i(c), \\
(\et_i b) \otimes c &\text{if }\varepsilon_i(b) > \varphi_i(c),
\end{cases}
\quad
\ft_i(b\otimes c) := \begin{cases}
b \otimes (\ft_i c) &\text{if }\varepsilon_i(b) < \varphi_i(c), \\
(\ft_i b) \otimes c &\text{if }\varepsilon_i(b) \geq \varphi_i(c),
\end{cases}
\]
weight function $\wt(b \otimes c) = \wt(b) + \wt(c)$, and statistics
\begin{align*}
\varepsilon_i(b \otimes c) & = \max\bigl( \varepsilon_i(c), \varepsilon_i(c) + \varepsilon_i(b) - \varphi_i(c) \bigr),
\\
\varphi_i(b \otimes c) & = \max\bigl( \varphi_i(b), \varphi_i(c) + \varphi_i(b) - \varepsilon_i(b) \bigr).
\end{align*}
From~\cite{K90,K91}, this is an abstract $U_q(\aff{\g})$-crystal and makes the category of $U_q(\aff{\g})$-crystals into a tensor category.
In particular, the tensor product is associative.
Furthermore the tensor product of two (upper/lower regular) $U_q(\aff{\g})$-crystals is again a (upper/lower regular) $U_q(\aff{\g})$-crystal, which makes upper/lower regular $U_q(\aff{\g})$-crystals into a full tensor subcategory.

\begin{rem}
In this paper, we follow~\cite{BS17}, which uses the opposite tensor product convention of Kashiwara~\cite{K90,K91}.
\end{rem}

From the definition, we can realize an abstract $U_q(\aff{\g})$-crystal $B$ as an $I$-edge-colored weighted directed graph called the \defn{crystal graph} (up to the statistics $\varepsilon_i, \varphi_i$).
In particular, any \defn{$i$-string} (a connected full subcrystal of $B$ considered as a $U_q(\aff{\g}_{\{i\}})$-crystal) through an element $b$ in a regular crystal is depicted as
\[
\et_i^{\varepsilon_i(b)} b \xrightarrow[\hspace{20pt}]{i} \cdots \xrightarrow[\hspace{20pt}]{i} \et_i b \xrightarrow[\hspace{20pt}]{i} b \xrightarrow[\hspace{20pt}]{i} \ft_i b \xrightarrow[\hspace{20pt}]{i} \cdots \xrightarrow[\hspace{20pt}]{i} \ft_i^{\varphi_i(b)} b.
\]
For nonregular crystals, the $i$-strings might be (semi)infinite and the number of times we can apply $\et_i$ or $\ft_i$ to $b$ might differ from $\varepsilon_i(b)$ and $\varphi_i(b)$, respectively.

Kashiwara showed~\cite{K90,K91} that for every affine Lie algebra (more generally, any symmetrizable Kac--Moody Lie algebra), their irreducible highest weight representations $V(\lambda)$ for $\lambda \in \aff{P}^+$ admit a crystal basis\footnote{We will not differentiate between crystals and crystal bases here, but instead are using it to indicate a connection to representation theory. Hence, we simply refer the reader to, \textit{e.g.},~\cite{K90,K91}, for a precise definition of a crystal basis.} $B(\lambda)$ that is a regular crystal with a unique highest weight element $u_{\lambda}$.
The \defn{depth} of an element $b \in B(\lambda)$ is the number of crystal raising operators needed to reach $u_{\lambda}$; this is well defined and can be computed from $\lambda - \wt(b)$.
Furthermore, Kashiwara also showed $B(\lambda)$ carries a natural action of the Weyl group given by
\[
s_i b = \begin{cases} \et^{-k}_i b & \text{if } k \leq 0, \\  \ft^k_i b & \text{if } k \geq 0, \end{cases}
\qquad
\text{ where }
k := \inner{\wt(b)}{\alpha_i^{\vee}}.
\]
Thus, this corresponds to reflecting an element across the middle of its $i$-string.

All Cartan subalgebras of $\aff{\g}$ are conjugate to each other by elements of the Weyl group, thus we can twist the root system by any element $w \in W$.
For $B(\lambda)$, this produces an isomorphic crystal as the crystal is independent of the choice of Cartan subalgebra.
Thus, for any root $\beta \in \Phi$ and $b \in B$ for any crystal $B$, we can define a crystal raising and lowering operator by
\[
e_{\beta} b := w^{-1} e_i w b,
\qquad\qquad
f_{\beta} b := w^{-1} f_i w b,
\]
for some $w \in W$ and $i \in I$ such that $w \beta = \alpha_i$.
However, such root operators can depend upon the choice of $w$ and $i$, which was noted in~\cite[Rem.~7.4.2]{K94}.
The earliest appearance of this idea that the authors are aware of is~\cite{K94}, but it has been employed in~\cite{GJKKK14} to construct the additional crystal operators for the queer Lie superalgebra.
Some more recent uses of general root operators was undertaken in~\cite{LL21,Patimo21,PT23}.
This is slightly different than the general root operators in~\cite{L95-2}, which do not depend on a Weyl group element.

\begin{ex}
Consider the crystal $B(\clfw_1+\clfw_2)$ in type $A_2$:
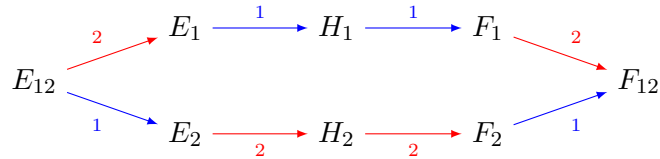
\[
\begin{tikzpicture}[>=latex, xscale=2, yscale=.7]
\node (e12) at (-2, 0) {$E_{12}$};
\node (e1) at (-1,1) {$E_1$};
\node (e2) at (-1,-1) {$E_2$};
\node (h1) at (0,1) {$H_1$};
\node (h2) at (0,-1) {$H_2$};
\node (f1) at (1,1) {$F_1$};
\node (f2) at (1,-1) {$F_2$};
\node (f12) at (2,0) {$F_{12}$};
\draw[->,red] (e12) -- node[midway,anchor=south east] {\tiny $2$} (e1);
\draw[->,blue] (e12) -- node[midway,anchor=north east] {\tiny $1$} (e2);
\draw[->,red] (e2) -- node[midway,anchor=north] {\tiny $2$} (h2);
\draw[->,blue] (e1) -- node[midway,anchor=south] {\tiny $1$} (h1);
\draw[->,red] (h2) -- node[midway,anchor=north] {\tiny $2$} (f2);
\draw[->,blue] (h1) -- node[midway,anchor=south] {\tiny $1$} (f1);
\draw[->,red] (f1) -- node[midway,anchor=south west] {\tiny $2$} (f12);
\draw[->,blue] (f2) -- node[midway,anchor=north west] {\tiny $1$} (f12);
\end{tikzpicture}
\]
We have $\beta := \alpha_1 + \alpha_2 = s_1 \alpha_2 = s_2 \alpha_1$, and so we can define $f_{\beta} = s_1 f_2 s_1$ or $f'_{\beta} = s_2 f_1 s_2$.
Therefore, we have $f_{\beta} E_{12} = H_2$ but $f'_{\beta} E_{12} = H_1$.
\end{ex}

Kashiwara also showed that the lower half of the quantum group $U^-_q(\aff{\g})$ admits a crystal basis denoted $B(\infty)$.
Moreover, he showed that $B(\infty)$ is isomorphic to the direct (injective) limit of $B(\lambda)$ as $\lambda \to \infty$ by identifying the highest weight elements up to a trivial weight shift given by tensoring with $T_{\lambda}$.
More precisely, there exists a directed system on $\{T_{-\lambda} \otimes B(\lambda) \mid \lambda \in \aff{P}^+\}$, which means the diagram
\[
\begin{tikzpicture}[>=latex]
\node (0) at (0,0) {$T_{-\lambda} \otimes B(\lambda)$};
\node (1) at (5,0) {$T_{-\lambda - \mu} \otimes B(\lambda + \mu)$};
\node (2) at (5,-2) {$T_{-\lambda - \mu - \nu} \otimes B(\lambda + \mu + \nu)$};
\draw[->] (0) -- node[midway,above] {$\psi_{\lambda,\mu}$} (1);
\draw[->] (1) -- node[midway,right] {$\psi_{\lambda+\mu,\nu}$} (2);
\draw[->] (0) -- node[midway,below] {$\psi_{\lambda,\mu+\nu}$} (2);
\end{tikzpicture}
\]
commutes for all $\lambda, \mu, \nu \in \aff{P}^+$, where $\psi_{\lambda,\mu}$ being the crystal embedding defined by $\psi_{\lambda,\mu}(t_{-\lambda} \otimes u_{\lambda}) = t_{-\lambda-\mu} \otimes u_{\lambda+\mu}$.
Therefore, the result of Kashiwara is
\begin{equation}
\label{eq:direct_limit_construction}
B(\infty) \iso \varinjlim_{\lambda \to \infty} B(\lambda)
\end{equation}
from this directed system with $u_{\infty} \mapsto [t_0 \otimes u_0]$.
We also refer the reader to~\cite[Ch.~7]{K90f} for more information.

%%===================
\subsection{Polyhedral model}

For the remainder of this section, unless otherwise noted, we restrict our discussion to $\g$, which will be a finite dimensional (semi)simple Lie algebra.
Let $N := \ell(w_0)$, and recall $R(w)$ is the set of reduced expressions of $w$.

We now describe a way to parameterize of elements in $B(\mu)$ for $\mu \in P^+ \sqcup \{\infty\}$ by using integer points in particular polyhedron, which depends on a particular choice of a reduced expression $\ii \in R((w^J)^{-1})$, where $J$ is such that $W_J$ is the parabolic subgroup corresponding to the stabilizer of $\mu$ (if $\mu = \infty$, then $J = \emptyset$ and $w^J = w_0$).
In particular, for $\mu = \clfw_r$, we have $J = I_0 \setminus \{r\}$.
This is known as the \defn{polyhedral model} and can be described as follows.

Recall the elementary crystal $B_i$ from Example~\ref{ex:elem_crystal}.
By~\cite{K93}, for any $i \in I$, there is a unique embedding of crystals
\begin{equation}
	\label{eq:kashiwara_embedding}
	\Psi_i \colon B(\infty) \hookrightarrow B_i \otimes B(\infty),
	\qquad\qquad
	u_{\infty} \mapsto b_i(0) \otimes u_{\infty},
\end{equation}
that we call the \defn{Kashiwara embedding}.
By iterating this and considering the inclusion $T_{-\mu} \otimes B(\mu) \hookrightarrow B(\infty)$ from~\eqref{eq:direct_limit_construction}, we have an embedding~\cite{KS97}
\[
\Psi_{\ii} \colon B(\mu) \hookrightarrow B_{i_1} \otimes \cdots \otimes B_{i_N} \otimes u_{\mu}
\]
for any $\ii = (i_1, \dotsc, i_N) \in R((w^J)^{-1})$.\footnote{These results hold for any Kac--Moody Lie algebra $\g$, but in that case, we need $\ii$ to be a sequence $\ii = (i_j \in I)_{j=1}^{\infty}$ with $i_j \neq i_{j+1}$ for all $j = 1,2,\ldots$ and $\abs{\{j = 1,2,\ldots \mid i_j = i\}} = \infty$ for all $i \in I$.}
Furthermore, we can equate the image $b \mapsto b_{i_1}(-k_1) \otimes \cdots \otimes b_{i_N}(-k_N) \otimes u_{\mu}$ with the integer point $(k_1, \dotsc, k_N) \in \ZZ_{\geq 0}^N$ inside a particular polyhedral cone.
For more information, see~\cite{Nakashima99,NZ97}. % with some recent progress in~\cite{Kanakubo23}.
Note that we have negated the points in the cone to simplify our exposition below.

The crystal structure for the polyhedral model is given by repeatedly applying the tensor product rule, which in this case can be described more precisely as follows (\textit{cf.}~\cite[Prop.~2.1.1]{KN94}).
Consider $\kk = (k_1, \dotsc, k_N)$.
Define
$
a_i(m) := k_{N+1-m} + \sum_{j=1}^{m-1} A_{i,i_{N+1-j}} k_{N+1-j}.
$
Then
\begin{gather*}
\varepsilon_i(\kk) %= \max_{1 \leq k \leq n}\left( \sum_{j=1}^k \varepsilon_i(b_j) - \sum_{j=1}^{k-1}\varphi_i(b_j) \right)
  = \max_{1 \leq m \leq N}\bigl( a_i(m) \bigr),
\hspace{30pt}
\varphi_i(\kk) %= \max_{1 \leq k \leq N} \left( \varphi_i(b_N) + \sum_{j=k}^{N-1} \big( \varphi_i(b_j) - \varepsilon_i(b_{j+1}) \big) \right)
   = \max_{1 \leq m \leq N}\bigl( \lambda_i + a_i(m) \bigr),
\\
\wt(\kk) = -k_1 \alpha_{i_1} - \cdots - k_N \alpha_{i_N},
\end{gather*}
where $\lambda_i = -\sum_{j=1}^N A_{i,i_j} k_j$.
The crystal operators are determined as follows.
\begin{itemize}
\item If $a_i(m) > a_i(j)$ for $1 \leq j < m$ and $a_i(m) \geq a_i(j)$ for $m < j \leq N$:
\[
e_i(\kk) = (k_1, \dotsc, k_m-1, \dotsc, k_N). %b_N \otimes \cdots \otimes e_i b_k \otimes \cdots \otimes b_1.
\]
\item If $a_i(m) \geq a_i(j)$ for $1 \leq j < m$ and $a_i(k) > a_i(j)$ for $m < j \leq N$:
\[
f_i(\kk) = (k_1, \dotsc, k_m+1, \dotsc, k_N). % b_N \otimes \cdots \otimes f_i b_k \otimes \cdots \otimes b_1.
\]
\end{itemize}
%For any $\ii \in R((w^J)^{-1})$, we denote the crystal by $P_{\ii}(\mu)$.

%===================
\subsection{String data}

There is another general set of points in $\ZZ_{\geq 0}^N$ that we can use to uniquely parameterize the elements in $B(\infty)$ (and hence $B(\lambda)$) associated to a fixed reduced word $\ii \in R(w_0)$.
The \defn{$\ii$-string data} (when $\ii$ is clear, we will simply call this the string data) of $b \in B(\infty)$ is the sequence $(d_1, d_2, \dotsc, d_N)$ constructed recursively by
\begin{equation}
\label{eq:string_data_def}
d_1 = \varepsilon_{i_1}(b), \qquad
d_2 = \varepsilon_{i_2}(\et_{i_1}^{d_1} b), \qquad
\dotsc,
\qquad
d_N = \varepsilon_{i_N}(\et_{i_{N-1}}^{d_{N-1}} \cdots \et_{i_1}^{d_1} b).
\end{equation}
However, not every point in $\ZZ_{\geq 0}^N$ corresponds to an element in $B(\infty)$ (for example, $(1,0,1) \in \ZZ^3$ in type $A_2$ for any $\ii \in R(w_0)$), but all possible string data is known that to be given by integer points in a rational cone that can be computed~\cite{Littelmann98}, which is called the Littelmann string cone.
The crystal structure can also be directly computed by piecewise-linear functions from~\cite{BFZ96}, but we do not describe this here as we do not use it.
We refer the reader to~\cite{KS22,KKN24} for more details and recent progress on understanding the Littelmann string cones.

%===================
\subsection{Star crystal structure}

There exists a map $\ast \colon B(\infty) \to B(\infty)$ that is induced from an anti-involution of the quantum group $U_q(\g)$~\cite{K93,Lusztig90}.
However, we can give an intrinsic crystal theoretical description using the Kashiwara embedding~\eqref{eq:kashiwara_embedding}.

Fix some $i \in I$, then the \defn{star crystal operators} $\et_i^{\ast}, \ft_i^{\ast} \colon B(\infty) \to B(\infty) \sqcup \{\zero\}$ and statistics $\varepsilon_i^*, \varphi_i^* \colon B(\infty) \to \ZZ \sqcup \{-\infty\}$ are defined as follows.
For some $b \in B(\infty)$, suppose $\Psi_i(b) = b_i(-k) \otimes b'$, and then define
\begin{align*}
\et_i^{\ast}(b) & = \begin{cases} \zero & \text{if } k = 0, \\ \Psi_i^{-1}\bigl( b_i(-k+1) \otimes b') & \text{if } k > 0, \end{cases}
&
\ft_i^{\ast}(b) & = \Psi_i^{-1}(b_i(-k-1) \otimes b'),
\\
\varepsilon_i^{\ast}(b) &= k,
&
\varphi_i^{\ast}(b) &= \langle \wt(b), \alpha_i^{\vee} \rangle + k.
\end{align*}
It is known that this gives an additional crystal structure on $B(\infty)$ that we call the \defn{star crystal structure}~\cite{K93,KS97,Lusztig90}.
Furthermore, the star crystal structure on $B(\infty)$ is isomorphic to $B(\infty)$ (under the usual crystal structure), where the map $\ast$ is the isomorphism given by mapping $u_{\infty} \to u_{\infty}$, and we call this the \defn{bicrystal structure} of $B(\infty)$.
More precisely, for any $\ii \in R(w_0)$, we have
\begin{equation}
\label{eq:star_defn}
b^* := \ast(b) = \bigl( \ft_{i_1}^{\ast} \bigr)^{d_1} \cdots \bigl(\ft_{i_N}^{\ast} \bigr)^{d_N} \et_{i_N}^{d_N} \cdots \et_{i_1}^{d_1} b,
\end{equation}
where $(d_1, \dotsc, d_N)$ is the $\ii$-string data of $b$ from~\eqref{eq:string_data_def}.
As a consequence of~\eqref{eq:star_defn}, the point $\Psi_{\ii}(b) \in \ZZ_{\geq0}^N$ from the $\ii$-polyhedral model is the $\ii$-string data of $b^*$.

%===================
\subsection{PBW crystals}

Let us briefly recall a PBW basis and the crystal of Lusztig data, which is isomorphic to $B(\infty)$~\cite{Lusztig90,Lusztig93,Saito94}.

For $\ii = (i_1, \dotsc, i_N) \in R(w_0)$, a total order $\prec$ on the positive roots of $\g$ is constructed by
\begin{equation}\label{eq:beta_k}
\Phi^+=\{\beta_1:=\alpha_{i_1} \prec \beta_2:=s_{i_1}(\alpha_{i_2}) \prec \cdots \prec \beta_N := s_{i_1}\cdots s_{i_{N-1}}(\alpha_{i_N})\}
\end{equation}
Any such total order $\prec$ is called a \defn{convex order} on $\Phi^+$, which is characterized (\textit{e.g.}, see~\cite{Papi}) by the property that for $\alpha, \beta \in \Phi^+$ such that $\alpha \prec \beta$ and $\alpha + \beta \in \Phi^+$, then $\alpha \prec \alpha + \beta \prec \beta$.

For $i\in I$, let $T_i$ be the $\QQ(q)$-algebra automorphism of $U_q(\g)$, which is denoted by $T''_{i,1}$ in~\cite{Lusztig93}.
For $1\leq k\leq N$, define $F_{\beta_k} := T_{i_{1}}T_{i_2}\cdots T_{i_{k-1}}(F_{i_{k}})$,
and for $\cc = (c_{\beta_1},\ldots,c_{\beta_N})\in\mathbb{Z}_+^N$, let
\begin{equation}\label{eq:PBW vector}
\begin{split}
b _\ii(\bf c) & = 
F_{\beta_1}^{(c_{\beta_1})} F_{\beta_2}^{(c_{\beta_2})}\cdots F_{\beta_N}^{(c_{\beta_N})},
\end{split}
\end{equation}
where $F_{\beta_k}^{(c)}$ is the divided power of $F_{\beta_k}$.
Then the set $B_\ii:=\{ b _\ii(\cc) \mid \cc \in \ZZ_{\geq 0}^{N}\}$ is a $\QQ(q)$-basis of $U^-_q(\g)$ called a \defn{PBW basis}.

Let $A_0$ be the subring of $\QQ(q)$ consisting of rational functions regular at $q=0$.
The $A_0$-lattice $L(\infty)$ of $U^-_q(\g)$ generated by $B_\ii$ is independent of the choice of $\ii$ and invariant under the Kashiwara operators $\et_i$, $\ft_i$.
Moreover, the induced crystal (basis) $\pi(B_\ii)$ under the canonical projection $\pi \colon L(\infty) \rightarrow L(\infty) / q L(\infty)$ is isomorphic to $B(\infty)$. 
We identify 
\[
\B_\ii := \ZZ_{\geq 0}^N
\] 
with the crystal $\pi(B_\ii) \iso B(\infty)$ under the map $\cc \mapsto b_\ii(\cc)$.
We call $\cc \in \B_\ii$ an \defn{$\ii$-Lusztig data} associated to $\ii$.
Hence, $\B_\ii$ is called the \defn{crystal of $\ii$-Lusztig data} or, when there is no confusion for $\ii$, the \defn{PBW crystal}.
Note that all possible Lusztig data are the integer points of the polyhedron $\RR^N_{\geq 0}$, which is $\ZZ_{\geq 0}^N$, as opposed to the more subtle string data.

Let us fix a reduced expression $\ii \in R(w_0)$, and let $\beta_1 \prec \cdots \prec \beta_N$ be the corresponding convex order.
We use the results of~\cite{Lusztig90,K90,K91} to compute the crystal operators on the $\ii$-Lusztig data.
We first have
\begin{subequations}
\begin{align}
\ft_{i}(c_1, \dotsc, c_N) &= (c_1+1,\ldots,c_N), &&
\text{when $\beta_1=\alpha_i$},\\
\ft^{\ast}_{i}(c_1, \dotsc, c_N) &= (c_1,\ldots,c_N+1), && \text{when $\beta_N=\alpha_i$}.
\end{align}
\end{subequations}
The remaining crystal structure, so for $\ft_i$ with $i \neq i_1$, is given by using the braid moves to change the Lusztig data so that $i$ is the first entry (such a reduced expression exists because $w_0$ is the longest element).
More precisely, the action of the braid moves is given by explicit piecewise-linear functions; see~\cite[Prop.~2.1.2]{Saito94} and ~\cite[Prop.~3.6]{BZ01}.
However, when $\ii$ is a reduced expression in~\eqref{eq:nice_decomp}, we can provide a more explicit description of the crystal operators.
The reduced expressions~\eqref{eq:nice_decomp} were shown to be \defn{simply-braided} by Salisbury, Schultze, and Tingley~\cite[Lemma~4.17]{SST18}, which is a property that allows us to define crystal operators on Lusztig data by a simple bracketing rule by~\cite[Thm.~4.5]{SST18} (\textit{cf.}~\cite[Prop.~3.2, Rem.~3.3]{JK19}).

It is also known that for any $\lambda \in P^+$, to parameterize elements of $B(\lambda)$, it is sufficient to restrict to the Lusztig data according to $\ii \in R\bigl((w^J)^{-1}\bigr)$, where $J \subseteq I_0$ is such that $W_J$ is the stabilizer of $\lambda$.
Indeed, if we consider some $\ii^+ = (i_1, \dotsc, i_N) \in R(w_0)$ such that the prefix is $\ii = (i_1, \dotsc, i_{\ell})$ (with $\ell = \ell(w^J) = \ell((w^J)^{-1})$ here), then for any $t_{-\lambda} \otimes b \in T_{-\lambda} \subseteq B(\lambda) \subseteq B(\infty)$ with $\ii^+$ Lusztig data $\cc \in \ZZ_{\geq 0}^N$, we must have $c_j = 0$ for all $j > \ell$.

%%%%%%%%%%%%%%%%%%%%%%%%%%%%%%%%%%%%%%%%
\section{Minuscule crystals} \label{sec:minuscule crystals}

\subsection{Posets, heaps, and reverse plane partitions}

Throughout this paper, unless otherwise stated, we will assume all of our posets to be finite.
Let $(\mcP, \leq)$ be a (finite) poset, and we will often write this just as $\mcP$ when there is no danger of confusion about the ordering.
A \defn{cover relation} in $\mcP$ are elements $\alpha < \beta$ such that there is no $\gamma \in \mcP_r$ such that $\alpha < \gamma < \beta$, and this is denoted by $\alpha \lessdot \beta$.
A \defn{weak chain} of length $\ell$ in $\mcP$ is a sequence of elements $p_1 \leq p_2 \leq \cdots \leq p_{\ell}$.
A \defn{lattice} is a poset with a well defined meet $\wedge$ (or infimum) and join $\vee$ (or supremum) for every pair of elements.
We draw the Hasse diagram of a poset (the digraph of covering relations) so that smaller elements are below larger elements unless otherwise indicated and omit the directional arrows.
For more information on posets, we refer the reader to~\cite[Ch.~3]{ECI}.

An \defn{order ideal} is a subset $\mcI \subset \mcP$ such that if $x \in \mcI$ and $x' \leq x$, then $x' \in \mcI$.
The set of all order ideals $\mcJ(\mcP)$ is a lattice under containment with the join (resp.\ meet) being union (resp.\ intersection) of the order ideals.
A \defn{join-irreducible element} $x$ in a lattice is an element that covers precisely one element; equivalently for $x = a \vee b$, then either $a = x$ or $b = x$.
For any join-irreducible element $\mcI \in \mcJ(\mcP)$, there is a unique element $x \in \mcI \subseteq \mcP$ such that $\mcI \setminus \{x\} \in \mcJ(\mcP)$.
Hence, we can recover $\mcP$ from $\mcJ(\mcP)$ as the subposet of join-irreducible elements of $\mcJ(\mcP)$.

%\begin{ex}
%For succinctness in the following example, we write sets as words; so $12 = \{1,2\}$.
%For the boolean lattice on 2 elements $\mcB_2$ and is order ideal lattice $\mcJ(\mcB_2)$, the Hasse diagrams are
%\[
%\begin{tikzpicture}[baseline=0]
%\node (12) at (0,1) {$12$};
%\node (1) at (-1,0) {$1$};
%\node (2) at (1,0) {$2$};
%\node (0) at (0,-1) {$\emptyset$};
%\draw[-] (12) -- (1);
%\draw[-] (12) -- (2);
%\draw[-] (1) -- (0);
%\draw[-] (2) -- (0);
%\end{tikzpicture}
%\qquad\qquad
%\begin{tikzpicture}[baseline=0]
%\node (all) at (0,2) {\underline{$\{\emptyset, 1,2, 12\}$}};
%\node (12) at (0,1) {$\{\emptyset, 1,2\}$};
%\node (1) at (-1,0) {\underline{$\{\emptyset, 1\}$}};
%\node (2) at (1,0) {\underline{$\{\emptyset, 2\}$}};
%\node (0) at (0,-1) {\underline{$\{\emptyset\}$}};
%\node (empty) at (0,-2) {$\emptyset$};
%\draw[-] (all) -- (12);
%\draw[-] (12) -- (1);
%\draw[-] (12) -- (2);
%\draw[-] (1) -- (0);
%\draw[-] (2) -- (0);
%\draw[-] (0) -- (empty);
%\end{tikzpicture}
%\]
%respectively, where the join-irreducible elements in $J(\mcB_2)$ are underlined.
%Note that $\emptyset = 1 \wedge 2$ and $12 = 1 \vee 2$ in $\mcB_2$.
%\end{ex}

A \defn{linear extension} of a poset $\mcP$ is a total order on the elements of $\mcP$ that is compatible with the partial order; that is, if we write the elements of $\mcP$ as a tuple $(p_1, \dotsc, p_N)$ to denote the total order, then we have $p_i < p_j$ (in $\mcP$) for all $i < j$.
Alternatively, a linear extension is an order preserving bijection $L \colon \mcP \to \{1, 2, \dotsc, \abs{\mcP}\}$ (under the natural order).
A \defn{reverse plane partition (RPP)} on $\mcP$ is an order reversing map 
$\rho \colon \mcP \to \ZZ_{>0}$; that is, if $a \leq b$ in $\mcP$, then $\rho(a) \geq \rho(b)$ (in the usual order on $\ZZ_{>0}$).
There is an easy bijection between RPPs and weak chains $\mcC := \mcI_1 \subseteq \cdots \subseteq \mcI_k$ in $\mcJ(\mcP)$ that, roughly speaking, simply stacks the order ideals to obtain the RPP.
More precisely, given a weak chain $\mcC$, we define the corresponding RPP by
\[
\rho_{\mcC}(b) = \abs{\{j \mid b \in \mcI_j \}}.
\]
We can construct the inverse bijection by using the order reversing property of $\rho$ to write it as a sum of indicator functions for the order ideals.

Following a construction by Robert Steinberg given in~\cite[Sec.~11]{Proctor84}, for a minuscule node $r$ of $\g$, the \defn{minuscule poset} $\mcP_r$ is the poset on the elements
\[
\{ \alpha \in \Phi^+ \mid \inner{\clfw_r}{-\alpha^{\vee}} < 0 \},
\]
where $\alpha \leq \beta$ if and only if $\beta - \alpha \in Q^+$.
Furthermore, if $\alpha \lessdot \beta$, then $\beta - \alpha = \alpha_i$ for some $i \in I_0$.
The elements are precisely $\Phi^J$ for $J = I_0 \setminus \{r\}$, and convex orders of $\Phi^J := \Phi \setminus \Phi_J$ are exactly the linear extensions of this poset.
Thus reduced expressions of $w^J$ are in bijection with linear extensions of $\mcP_r$.
From~\cite{Proctor99}, we can describe a minuscule poset as any self-dual $d$-complete poset.
An RPP for the poset $\mcP_r$ will be referred to as a \defn{minuscule RPP}.

A \defn{heap} (over $I_0$) is a poset $\mcP$ with a surjective map $\pi \colon \mcP \to I_0$ such that
\begin{itemize}
\item[(H1)] $\pi^{-1}(i)$ is totally ordered in $\mcP$,
\item[(H2)] if $i \sim j$, then $\pi^{-1}(i) \cup \pi^{-1}(j)$ is totally ordered in $\mcP$, and
\item[(H3)] the transitive closure of the above relations gives $\mcP$.
\end{itemize}
A \defn{minuscule heap} will be a heap on a minuscule poset.
Minuscule heaps were given a combinatorial characterization by Wildberger~\cite{Wildberger03}.
When $\g$ is simply-laced type, these were connected with quiver representations for a corresponding Dynkin quiver in~\cite{GPT18}.

Fix a linear extension $(p_1, \dotsc, p_{\ell})$ of $\mcP_r$.
The labels of a minuscule heap correspond to the indices $(i_1, \dotsc, i_{\ell}) = (\pi(p_1), \dotsc, \pi(p_{\ell}))$ for the minimal length coset representative $w^J = s_{i_{\ell}} \dotsm s_{i_1}$ (note the reversal of indices).
Note that if $p_i$ and $p_j$ are incomparable, then we must have $\pi(p_i) \not\sim \pi(p_j)$.
Furthermore, this is exactly the bijection between reduced expressions of $w^J$ and linear extensions of $\mcP_r$ previously mentioned.

\begin{ex}
Consider type $E_6$ and take $r = 6$.
Therefore, $J = \{1, 2, 3, 4, 5\}$ and $w^J$ is the element $w^{(2)}$ in Example~\ref{ex:nice_decomposition}.
The minuscule heap labeled by the linear extension corresponding to the reduced word in Example~\ref{ex:nice_decomposition} is given by Figure~\ref{fig:minuscule_heap_e6}.
Note that taking the dual poset gives the minuscule heap for $r =1$ with $w^{I \setminus \{1\}} = (w^{(2)})^{-1}$, and the dual linear extension $i \mapsto 17 - i$ corresponds to the reversal of the reduced word given in Example~\ref{ex:nice_decomposition}.
\end{ex}

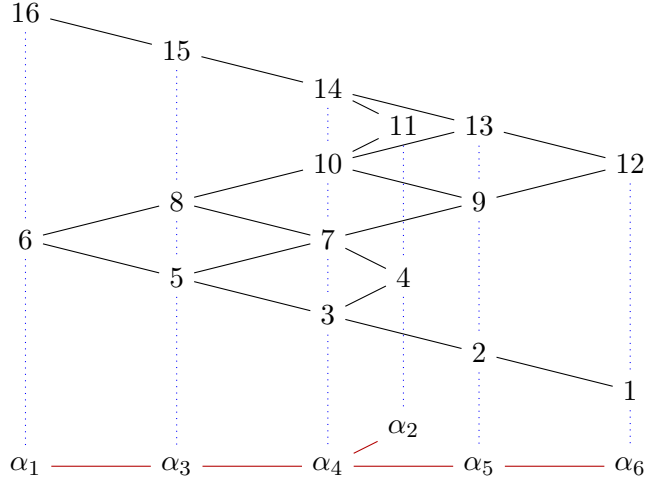
\begin{figure}
\[
\begin{tikzpicture}[>=latex, xscale=2,baseline=0]
%w^J = s_1 s_3 s_4 s_5 s_6 s_2 s_4 s_5 s_3 s_4 s_1 s_3 s_2 s_4 s_5 s_6
% Normal poset like display
%\node (0) at (0,0) {$1$};
%\node (1) at (0,1) {$2$};
%\node (2) at (0,2) {$3$};
%\node (3) at (-1,3) {$5$};
%\node (3p) at (1,3) {$4$};
%\node (4) at (-1,4) {$6$};
%\node (4p) at (1,4) {$7$};
%\node (5) at (-1,5) {$8$};
%\node (5p) at (1,5) {$9$};
%\node (6) at (-1,6) {$10$};
%\node (6p) at (1,6) {$12$};
%\node (7) at (-1,7) {$11$};
%\node (7p) at (1,7) {$13$};
%\node (8) at (0,8) {$14$};
%\node (9) at (0,9) {$15$};
%\node (10) at (0,10) {$16$};
\node (0) at (4,0) {$1$};
\node (1) at (3,.5) {$2$};
\node (2) at (2,1) {$3$};
\node (3) at (1,1.5) {$5$};
\node (3p) at (2.5,1.5) {$4$};
\node (4) at (0,2) {$6$};
\node (4p) at (2,2) {$7$};
\node (5) at (1,2.5) {$8$};
\node (5p) at (3,2.5) {$9$};
\node (6) at (2,3) {$10$};
\node (6p) at (4,3) {$12$};
\node (7) at (2.5,3.5) {$11$};
\node (7p) at (3,3.5) {$13$};
\node (8) at (2,4) {$14$};
\node (9) at (1,4.5) {$15$};
\node (10) at (0,5) {$16$};
\foreach[var=\j, evaluate=\i as \j using int(\i+1)] \i in {0,...,9} {
  \draw[-] (\i) -- (\j);
  \ifthenelse{2 < \i \AND \i < 7}{
    \draw[-] (\i p) -- (\j p);
  }{}
}
\draw[-] (2) -- (3p);
\draw[-] (7p) -- (8);
\draw[-] (3) -- (4p);
\draw[-] (4p) -- (5);
\draw[-] (5p) -- (6);
\draw[-] (6) -- (7p);
\node (d1) at (0, -1) {$\alpha_1$};
\node (d2) at (2.5, -1+.5) {$\alpha_2$};
\node (d3) at (1, -1) {$\alpha_3$};
\node (d4) at (2, -1) {$\alpha_4$};
\node (d5) at (3, -1) {$\alpha_5$};
\node (d6) at (4, -1) {$\alpha_6$};
\draw[-,color=darkred] (d1) -- (d3) -- (d4) -- (d5) -- (d6);
\draw[-,color=darkred] (d4) -- (d2);
\draw[-,dotted,color=blue] (10) -- (4) -- (d1);
\draw[-,dotted,color=blue] (7) -- (3p) -- (d2);
\draw[-,dotted,color=blue] (9) -- (5) -- (3) -- (d3);
\draw[-,dotted,color=blue] (8) -- (6) -- (4p) -- (2) -- (d4);
\draw[-,dotted,color=blue] (7p) -- (5p) -- (1) -- (d5);
\draw[-,dotted,color=blue] (6p) -- (0) -- (d6);
\end{tikzpicture}
\]
\caption{
A minuscule heap $\mcP_6$ for type $E_6$ for the node $r = 6$ with the projection $\pi \colon \mcP_6 \to I_0$ (here $i \in I_0$ is written as $\alpha_i$) indicated by the dotted lines. % (equivalently $J = \{1,2,3,4,5\}$)
%with $\mcP_6$ labeled by the linear extension corresponding to $\cosetrepr{w_0} = w_0^J = s_6 s_5 s_4 s_2 s_3 s_1 s_4 s_3 s_5 s_4 s_2 s_6 s_5 s_4 s_3 s_1$.
}
\label{fig:minuscule_heap_e6}
\end{figure}

\subsection{Minuscule crystals}

In this section, we will assume $r$ is a minuscule node and $J = I_0 \setminus \{r\}$.
We describe results on crystals $B(s\clfw_r)$.
The $U_q(\g)$-representation $V(\clfw_r)$ is a \defn{minuscule representation}, which means $W$ acts transitively on the crystal basis $B(\clfw_r)$ or, equivalently, $\langle \clfw_r, \alpha^{\vee} \rangle \leq 1$ for all $\alpha \in \Phi^+$~\cite{Seshadri78}.
Therefore, we can parameterize elements in $B(\clfw_r)$ by $W^J$ (equivalently $W / W_J$) since the stabilizer of $\clfw_r$ is $W_J$.
Indeed, for any $w \in W$ and reduced expression $s_{i_1} \dotsm s_{i_\ell} = \cosetrepr{w}$, we have $b = \ft_{i_1} \cdots \ft_{i_\ell} u_{\clfw_r} \in B(\clfw_r)$, which is independent of the choice of reduced expression.
More specifically, the crystal structure is given locally by $\ft_i b = b'$ if and only if $s_i b = b' \in W / W_J$ and $\ell(b') > \ell(b)$.

We also note the following consequence of the transitive Weyl group action, which is likely well-known to experts.

\begin{prop}
\label{prop:minuscule_Decomposition}
For any $K \subseteq I_0$, the restriction of $B(\clfw_r)$ to $U_q(\g_K)$-crystals is a disjoint union of minuscule $U_q(\g_K)$-representations.
\end{prop}

\begin{rem}
\label{rem:crystal_poset}
An irreducible highest weight crystal $B(\lambda)$ is naturally a connected ranked poset with the Hasse diagram given by the crystal graph after forgetting the edge labels.
The rank function is given by the depth, or the sum of the coefficients of $\lambda - \wt(b) \in Q^+$ expressed in terms of the simple roots.
If $\lambda = \clfw_r$ is a minuscule weight, then this poset corresponds to (strong) Bruhat order on $W^J$.
In general, highest weight crystals are not lattices; for example, $B(2\clfw_2)$ in type $A_3$ or $B(\clfw_4)$ in type $E_6$.
\end{rem}

Recall that linear extensions in the minuscule poset $\mcP_r$ corresponds to reduced expressions of $w^J$, which also correspond to maximal chains in $B(\clfw_r)$.
Therefore, we have that $B(\clfw_r)$, when considered as a poset by Remark~\ref{rem:crystal_poset}, is $\mcJ(\mcP_r)$.
This can also be seen directly with elements of $W^J$ being precisely order ideals in $\mcP_r$ using the heap $(\mcP_r, \pi)$.
More precisely, given a linear extension of an order ideal $(p_1, p_2, \dotsc, p_{\ell}) \subseteq \mcP_r$, we can form a reduced word for an element in $W^J$ by $s_{\pi(p_{\ell})} s_{\pi(p_2)} \dotsm s_{\pi(p_1)}$.
We leave it as an exercise to the reader to describe the crystal structure explicitly, but essentially it is given by $\ft_i$ tries to add an element along $\pi^{-1}(i)$ if possible to the order ideal.
%(A precise description is given below in a slightly different language.)
In particular, the minuscule crystals in type $E_{6,7}$ were constructed using heaps by Green~\cite{Green07,Green08}.
Thus, we can construct $\mcP_r$ as the join-irreducible elements of $B(\clfw_r)$ using the label of the incoming crystal edge to describe the surjection $\pi$ to construct the heap.
This process is also clearly reversible.

Next, we want to extend this crystal structure to $B(s\clfw_r)$.
We will do so by using the set $\mcR^{r,s}$ of minuscule RPPs for $\mcP_r$ with the entries being at most $s$; that is, all order reversing functions $\rho \colon \mcP_r \to \{0, 1, \dotsc, s\}$.
We can translate this into order ideals by mapping $\rho \in \mcR^{r,s}$ to a weak chain of order ideals $\mcI_1 \subseteq \mcI_2 \subseteq \cdots \subseteq \mcI_s$ by $p \in \mcI_{s-j}$ if and only if $j \leq \rho(p)$.
Note that this chain does not need to be maximal, and we may have $\mcI_{s-j} = \emptyset$.
It is clear this is bijective.
In other words, minuscule RPPs on $\mcP_r$ are in bijection with weak chains in the poset $\mcJ(\mcP_r)$.
Therefore, we can use the following result to describe the crystal structure.

\begin{prop}[{\cite[Prop.~7.29]{Scrimshaw20}}]
\label{prop:only_comparable}
Let $r$ be a minuscule node. Then we have
\[
\{ b_1 \otimes \cdots \otimes b_s \in B(\clfw_r)^{\otimes s} \mid b_1 \leq \cdots \leq b_s \} \iso B(s\clfw_r).
\]
\end{prop}

On the other hand, we can directly and explicitly define this crystal structure on the set of RPPs of $\mcP_r$ using the minuscule heap $(\mcP_r, \pi)$ as follows.
Define $\ft_i$ (resp.\ $\et_i$) as the operation that adds (resp.\ subtracts) $1$ from the largest entry from $\rho$ in $\pi^{-1}(i)$ where possible and being $\zero$ if no such move is possible.
We can think of $\ft_i$ as dropping a box down $\pi^{-1}(i)$ and $\et_i$ as dropping an ``anti-box.''
However, we need impose bracketing rule, where we read bottom-to-top along $\pi^{-1}(i)$ contributing `$\bclose$' for each addable box and then `$\bopen$' for each removable box (in that order).

While this paper was being written, this crystal structure was independently given in~\cite[Thm~1.2]{DEKMF22}, where it was remarked~\cite[Rem.~2.30]{DEKMF22} that Proposition~\ref{prop:only_comparable} can also be derived from~\cite{L94}.

\begin{prop}
\label{prop:direct_limit_finite_type}
Let $J = I_0 \setminus \{r\}$, where $r$ is a minuscule node.
We have $\mcR^{r,s} \iso B(s\clfw_r)$.
Moreover, this gives a directed system defined by $s \mapsto s + k$ for any $s, k \in \ZZ_{\geq 0}$ and the direct limit is isomorphic to $\B^J$.
\end{prop}

\begin{proof}
This is immediate from the definitions as action of the crystal operators is independent of $s$.
\end{proof}

\begin{prop}
\label{prop:heap_string}
Fix a linear extension $\ii = (i_1, \dotsc, i_N)$ of $\mcP_r$.
For any RPP $\rho \in \mcR^{r,s}$, the entries read according to $\ii$ is the $\ii$-polyhedral data of $\rho$.
\end{prop}

\begin{proof}
Since all reduced expressions for $w^J = \cosetrepr{w_0}$ are fully-commutative, the $\ii$-polyhedral data for any $\ii \in R((w^J)^{-1})$ only changes by reordering according to the choice of linear extension~\cite{Nakashima99,NZ97}.
Now suppose the claim holds for all elements of depth $k$.
We first show this for the case $s = 1$, where we simply are adding a box to the largest possible element $j \in \pi^{-1}(i)$.
This means every entry below $j$ must contain a box, so the claim follows from the polyhedral realization.
For general $s$, then our bracketing rule on the RRPs corresponds to the tensor product bracketing rule, which is a translation of the polyhedral model's crystal structure on the entry-wise addition (\textit{cf.}~\cite[Rem.~2.1.2]{KN94}).
\end{proof}

\begin{ex}
Consider type $A_4$ with $r = 2$, and set $\ii = (2,3,4,1,2,3)$.
We note that for all $s \geq 3$, the element $f_1^2 f_4 f_3^2 f_2^3 u_{s\Lambda_2}$ corresponds to the RPP
\begin{align*}
\begin{array}{cccc}
& & 0 \\
& 0 & & 1 \\
2 & & 2 \\
& 3
\end{array}
\!\in \mcR^{2,s}
\longleftrightarrow
 b_2(-3) \otimes b_3(-2) \otimes b_4(-1) \otimes b_1(-2) \otimes b_2(0) \otimes b_3(0) \otimes u_{\infty}.
\end{align*}
The abbreviation of the element in the $\ii$-polyhedral model on the right is $(3,2,1,2,0,0)$.
Then the corresponding $2$-strings through these elements are
\[
\begin{array}{c@{\,{\color{red}\xrightarrow[\hspace{15pt}]{2}}\,}c@{\,{\color{red}\xrightarrow[\hspace{15pt}]{2}}\,}c@{\,{\color{red}\xrightarrow[\hspace{15pt}]{2}}\,}c@{\,{\color{red}\xrightarrow[\hspace{15pt}]{2}}\,}c}
\begin{array}{cccc}
& & 0 \\
& 0 & & 1 \\
2 & & 2 \\
& 3
\end{array}
&
\begin{array}{cccc}
& & 0 \\
& 1 & & 1 \\
2 & & 2 \\
& 3
\end{array}
&
\begin{array}{cccc}
& & 0 \\
& 1 & & 1 \\
2 & & 2 \\
& 4
\end{array}
&
\begin{array}{cccc}
& & 0 \\
& 1 & & 1 \\
2 & & 2 \\
& 5
\end{array}
&
\cdots,
\\
(3,2,1,2,0,0)
&
(3,2,1,2,1,0)
&
(4,3,1,2,1,0)
&
(5,3,1,2,1,0)
&
\cdots.
\end{array}
\]
Consider the two values corresponding to $\pi_i^{-1}(2)$, and note that we do not make the top such entry into a $2$ in this $2$-string because of the pairing between an addable box at the top with a removable box at the bottom.
\end{ex}

Recall that in type $A_{n-1}$, the $\ii$-polyhedral data and $\ii$-Lusztig data can be naturally arranged in rectangular $r \times (n-r)$ matrices with the upper left corner representing $i_1 = r$ and $\beta_1 = \alpha_r$, respectively.
As such, we can apply the Robinson--Schensted--Knuth (RSK) bijection to the $\ii$-Lusztig data, where the output pair of semistandard tableaux gives an $r \times (n-r)$ matrix by joining their Gelfand--Tsetlin descriptions such that it matches the corresponding RPP.
See below for an example, and we refer the reader to, \textit{e.g.},~\cite[Sec.~4]{CITI} for the general details.

\begin{ex}
\newcommand{\tge}{\scalebox{.5}{$\geq$}}%
Consider type $A_5$ (or $\g = \mathfrak{sl}_6$), $r = 2$, and $\ii = (2, 1, 3, 4, 5, 2, 3, 4)$, then the root order for the $\ii$-Lusztig data is
\[
\alpha_2, \quad \alpha_{12}, \quad \alpha_{23}, \quad \alpha_{24}, \quad \alpha_{25}, \quad \alpha_{13}, \quad \alpha_{14}, \quad \alpha_{15},
\]
where $\alpha_{ij} = \alpha_i + \cdots + \alpha_j$.
Then, the RSK bijection maps the $\ii$-Lusztig data to the RPP (equivalently $\ii$-polyhedral data) as follows:
\begin{align*}
\begin{bmatrix}
c_1 & c_3 & c_4 & c_5 \\
c_2 & c_6 & c_7 & c_8
\end{bmatrix}
\xrightarrow{RSK} (P, Q)
& \longleftrightarrow
\begin{array}{cccccccc}
i_1 &\tge& i_6 &\tge& 0 && 0 \\
& i_3 &\tge& i_7 &\tge& 0 \\
&& i_4 &\tge& i_8 \\
&&& i_5
\end{array},
\begin{array}{ccc}
i_1 &\tge& i_6 \\
& i_2
\end{array}
\\ & \longleftrightarrow
\begin{bmatrix}
i_1 & i_3 & i_4 & i_5 \\
i_2 & i_6 & i_7 & i_8
\end{bmatrix}
\longleftrightarrow
\begin{tikzpicture}[scale=.5,baseline=15]
\foreach \i/\x\y in {1/0/0,2/-1/1,3/1/1,4/2/2,5/3/3,6/0/2,7/1/3,8/2/4}
  \node (i\i) at (\x,\y) {$i_{\i}$};
\foreach \i in {1,2,3,4,5}
  \node (b\i) at (\i-2,-1) {$\alpha_{\i}$};
\draw[-] (i1) -- (i3) -- (i4) -- (i5) -- (i8);
\draw[-] (i1) -- (i2) -- (i6) -- (i7) -- (i8);
\draw[-] (i3) -- (i6);
\draw[-] (i4) -- (i7);
\draw[blue,dotted] (b1) -- (i2);
\draw[blue,dotted] (b2) -- (i1) -- (i6);
\draw[blue,dotted] (b3) -- (i3) -- (i7);
\draw[blue,dotted] (b4) -- (i4) -- (i8);
\draw[blue,dotted] (b5) -- (i5);
\end{tikzpicture}
\end{align*}
\end{ex}

\begin{cor} \label{cor: RSK between Lusztig and string}
Let $\g$ be of type $A_{n-1}$, and fix some $r \in I_0$.
For any $\ii \in R\bigl((w^J)^{-1}\bigr)$, the RSK bijection is the $U_q(\g)$-crystal isomorphism between the $\ii$-polyhedral data and $\ii$-Lusztig data for $B(s\clfw_r)$.
\end{cor}

\begin{proof}
The minuscule heap, which is the restriction of the root poset to $\Phi^{J+}$, is a product of two chains of length $r$ and $n-r$, we can naturally identify the entries of a $r \times (n-r)$ matrix with $\Phi^{J+}$.
In particular, we associate the top-left entry of the matrix with $\alpha_r$.
A straightforward computation shows the PBW crystal associated to $(\cosetrepr{w_0})^{-1} w_{J'}$ with $J'$ being the image of $J$ under the natural order $2$ automorphism on the Dynkin diagram of type $A_n$, corresponds to the crystal structure given by~\cite{Kwon18}. Hence, the claim follows from Proposition~\ref{prop:heap_string}. 
\end{proof}

There is an explicit crystal isomorphism between RPPs and semistandard tableaux of rectangle shape for type $A_n$.
We rotate the RPP 90 degrees counterclockwise and  replace each entry $x$ by $s - x$, then we obtain the tableau by considering the result as a Gelfand--Tsetlin pattern using the canonical bijection.
This was described in~\cite{DEKMF22}.
Alternatively, we can read the RPP along diagonals $(d_r, \dotsc, d_1)$ starting from the bottom and going to the northeast.
We will denote the entries in the $k$-th diagonal by $d_k = (d_{k,k}, d_{k,k+1}, \dotsc, d_{k,n+1-k})$.
The number of $i$'s in row $k$ is equal to $s - \sum_{j=k}^i d_{k,j}$ (note that necessarily for $i < k$, there are no such $i$'s).
From this, we see the bracketing rule is precisely the bracketing rule from semistandard tableaux; see also~\cite[Prop.~3.12]{CITI} for Proposition~\ref{prop:heap_string}.
Moreover, certain string data can be extracted (see, \textit{e.g.},~\cite[Prop.~2.3.3]{WY21}).

%%%%%%%%%%%%%%%%%%%%%%%%%%%%%%%%%%%%%%%%
\section{Uniform construction of special KR crystals} \label{sec: construction of special KRs}

\subsection{Kirillov--Reshetikhin crystals}
An important class of finite dimensional representations of $U_q'(\aff{\g})$ are the Kirillov--Reshetikhin representations $W^{r,s}$, where $r \in I_0$ and $s \in \ZZ_{>0}$, which are classified by their Drinfel'd polynomials (\textit{e.g.},~see \cite[Sec.~6.9]{ChHe10}).
The $U_q'(\gaff)$-module $W^{r,s}$ is conjectured~\cite{HKOTY99,HKOTT02} to have a corresponding crystal\footnote{Strictly speaking, these are crystal \emph{pseudo}bases, which satisfy the crystal basis axioms up to a negligible sign.} called a \defn{Kirillov--Reshetikhin crystal} (KR crystal for short) and denoted by $B^{r,s}$.
This conjecture has been proven in most cases~\cite{BS20,KKMMNN92,Okado07,OS08,NS21}.
In particular, one important property when $r$ is a special node is that $B^{r,s} \iso B(s\clfw_r)$ as $U_q(\g)$-crystals for all $s$ (\textit{e.g.}~see \cite{Chari01}), which yields a uniform proof of their existence from~\cite{KKMMNN92}.
If $r$ is special node, then we refer to $B^{r, s}$ as a \defn{special KR crystal}.

%===================
\subsection{Construction of special KR crystals}

As before, take $J = I_0 \setminus \{r\}$ for some special node $r$.
Take a reduced expression $\ii_0$ of $w_0$ such that $\ii_0 = \ii^J \cdot \ii_J$, where $\ii^J \in R((w^J)^{-1})$.
Then we have 
\[
	\B_{\ii_0} \cong \B_J \otimes \B^J,
\]
where $\B^J$ (resp.~$\B_J$) is the subcrystal of $\B_{\ii_0}$ associated to $\ii^J$ (resp.~$\ii_J$); more precisely, 
\begin{align*}
	\B^J &= \{ \cc = (c_\beta) \in \B_{\ii_0} \, \mid \, c_\beta = 0 \text{ unless } \beta \in (\Phi^J)^+  \}, \\
	\B_J &= \{ \cc = (c_\beta) \in \B_{\ii_0} \, \mid \, c_\beta = 0 \text{ unless } \beta \in \Phi_J^+ \}.
\end{align*}
Note that
\begin{equation} \label{eq:B upper J}
\B^J = \{ \cc \in \B_{\ii_0} \, \mid \, \varepsilon_i^*(b) = 0 \text{ for all } i \in J \}.
\end{equation}

\begin{rem} \label{rem:parabolic Verma}
The crystal $\B^J$ also comes from the crystal basis of a $U_q(\g)$-module.
Let $U_q(\mathfrak{p})$ be the subalgebra of $U_q(\g)$ generated by $U_q(\g_J)$ and the positive half of $U_q(\g)$.
For $\lambda \in P^+$, we denote by $V_J(\lambda)$ the finite-dimensional irreducible $U_q(\g_J)$-module with highest weight $\lambda$.
The \defn{parabolic Verma module} is
\begin{equation*}
	M^J(\lambda) = U_q(\g) \otimes_{U_q(\mathfrak{p})} V_J(\lambda),
\end{equation*}
where we inflate $V_J(\lambda)$ to a $U_q(\mathfrak{p})$-module by having $X v = 0$ for all $X \in \mathfrak{p} \setminus \g_J$.
Note that 
\begin{equation*}
	M^J(\lambda) \iso U_q^-(w^J) \otimes_{\QQ(q)} V_J(\lambda)
\end{equation*}
as $\QQ(q)$-vector spaces.
Here $U_q(w^J)$ is the $\QQ(q)$-subalgebra of the negative half of $U_q(\g)$ generated by PBW vectors associated to $w^J$.
Then it follows from \cite{JKU24} that the $U_q(\g)$-module $M^J(\lambda)$ has a crystal base for classical types.
In particular, $M^J(0) \iso U_q^-(w^J)$ with its crystal graph given by $\B^J$ (\textit{cf.}~\cite{Kwon14}), and hence we also refer to $\B^J$ as the crystal of the parabolic Verma module (of highest weight $0$).
\end{rem}

Following \cite{Kwon13, Kwon18}, we can promote the crystal $\B^J$ to a $U_q(\aff{\g})$-crystal by defining $0$-crystal operators $\et_0$ and $\ft_0$ as follows.
Recall \eqref{eq: def of Theta}.
For $\cc = (c_\beta) \in \B^J$, we define
\begin{subequations}
\label{eq:affine structure}
\begin{align}
	\et_0 \cc &= \cc + \oneb_{\Theta}, \qquad \quad
	\ft_0 \cc =
	\begin{cases}
		\cc - \oneb_{\Theta} & \text{if $c_{\Theta} > 0$,} \\
		\zero & \text{otherwise,}
	\end{cases} \\
	\varphi_0(\cc) &= \max \{ \, k \mid \ft_0^k \cc \neq 0 \, \}, \qquad
	\varepsilon_0(\cc) = \varphi_0(\cc) - \langle \wt(\cc),\, h_0 \rangle,
\end{align}
\end{subequations}
where $\cc \pm \oneb_{\Theta} = (c'_\beta)$ with $c'_\Theta = c_\Theta \pm 1$ and $c'_\beta = c_\beta$ for all $\beta \neq \Theta$.
Let $\aff{\B}^J$ denote this $U_q'(\aff{\g})$-crystal, and define the subcrystal
\[
\aff{\B}^{J,s} := \{ b \in \aff{\B}^J \mid \varepsilon_r^{\ast}(b) \leq s \}.
\]
Note that
\begin{equation}
\label{eq:g0_decomp}
B(s\clfw_r) \iso T_{s\clfw_r} \otimes \aff{\B}^{J,s} = T_{s\clfw_r} \otimes \B^{J,s} \subset T_{s\clfw_r} \otimes \B^J
\end{equation}
as $U_q(\g_0)$-crystals from~\cite[Prop.~8.2]{K95} and~\eqref{eq:B upper J}. 
\smallskip

Now, we are in position to state the main result of this paper.

\begin{thm} \label{thm:main1}
	Let $\aff{\g}$ be an affine Lie algebra.
	Let $r$ be a special node and $J = I_0 \setminus \{r\}$. Then we have
	\[
	T_{s\clfw_r} \otimes \aff{\B}^{J,s} \iso B^{r,s}
	\]
	as $U_q'(\aff{\g})$-crystals.%, where $B^{r,s}$ is the KR crystal with respect to $r$ and $s \in \ZZ_{>0}$.
\end{thm}

We will give the proof in Section~\ref{proof of main theorem1}.
However, one corollary of Theorem~\ref{thm:main1} with Proposition~\ref{prop:direct_limit_finite_type} is that $\aff{\B}^J$ can be described as the direct limit of $T_{s\clfw_r} \otimes B^{r,s}$ as $s \to \infty$, which was the interpretation from~\cite{Kwon18, JK19}.

\begin{rem}
\label{rem:various}
\mbox{}
\begin{enumerate}[(1)]
	\item The crystal operators $\ft_i$, for $i \in I_0$, on $\B_{\ii_0}$ can be realized using different reduced expressions $\ii_0$~\cite{BZ01} (see also~\cite{Kamnitzer10}).
	More explicit crystal structures, such as given by a bracketing rule, using constructions have been described in~\cite{Re97,SST18}.
	There are also more type-specific descriptions for type $A$~\cite{Kwon18}, for type $D$~\cite{JK19}, and for type $E_{6,7}$~\cite{Jang22}.
	\smallskip

	\item Note that $B^{r,s}$ is a perfect crystal (a technical condition that is equivalent to $B^{r,s} \otimes B(\lambda) \iso B(\mu)$ for a given $\mu \in \aff{P}^+$ of level $s$ and $r$, where $\lambda$ depends on $\mu$ and $r$) in our cases by \cite{FOS10} for nonexceptional types and \cite{Hiroshima21} for type $E_{6,7}^{(1)}$.
In fact, the author in \cite{Hiroshima21} uses the result \cite{Jang22} (equivalently, Theorem \ref{thm:main1} for type $E_{6,7}^{(1)}$) to deduce that the KR crystals $B^{r,s}$ for type $E_{6,7}^{(1)}$ with special node $r$ are perfect.
	\smallskip
	
	\item In~\cite{HJ12} (see also \cite{Wang23}), Hernandez and Jimbo gave a representation-theoretic interpretation of the limit of the normalized $q$-characters of the KR modules in terms of an integrable $U_q(\mathfrak{b})$-module, which is called the \defn{prefundamental representation}. On the other hand, the crystal $\B^J$ can be viewed as the direct limit of $T_{s\clfw_r} \otimes B^{r,s}$ as $s \to \infty$ \cite{Kwon13,Kwon18}, and as the crystal of the unipotent quantum coordinate ring $U_q(w^J)$ associated with $w^J$. Motivated from these facts, the family of prefundamental representations attached to special nodes is uniformly constructed on $U_q(w^J)$ with an explicit $U_q(\mathfrak{b})$-actions \cite{JKP21, JKP25, Jan25}. Thus the crystal $\aff{\B}^J$ can be regarded as a crystal of the cominuscule prefundamental representation (cf.~Remark \ref{rem:parabolic Verma}).
\end{enumerate}
\end{rem}

\begin{ex}
Consider $\aff{\g}$ of type $D_3^{(2)}$ with $r = 2$ (so $J = \{1\}$).
Then the crystal $\aff{\B}^J$ up to depth $5$ is given by Figure~\ref{fig:affine_KR_limit_ex}.
Hence, we compute $\aff{\B}^{J,1}$:
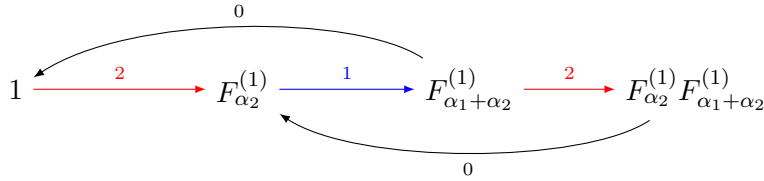
\[
\begin{tikzpicture}[>=latex,xscale=3]
\node (hw) at (0,0) {$1$};
\node (2) at (1,0) {$F_{\alpha_2}^{(1)}$};
\node (12) at (2,0) {$F_{\alpha_1+\alpha_2}^{(1)}$};
\node (212) at (3,0) {$F_{\alpha_2}^{(1)} F_{\alpha_1+\alpha_2}^{(1)}$};
\draw[->,red] (hw) -- node[midway,anchor=south] {\tiny $2$} (2);
\draw[->,blue] (2) -- node[midway,anchor=south] {\tiny $1$} (12);
\draw[->,red] (12) -- node[midway,anchor=south] {\tiny $2$} (212);
\draw[->,black] (212) .. controls (2.5,-1) and (1.5,-1) .. node[midway,anchor=north] {\tiny $0$} (2);
\draw[->,black] (12) .. controls (1.5,1) and (.5,1) .. node[midway,anchor=south] {\tiny $0$} (hw);
\end{tikzpicture}
\]
and $\aff{\B}^{J,2}$ (which is isomorphic to $B^{2,2}$ given in Example~\ref{ex:virtual_crystal}):
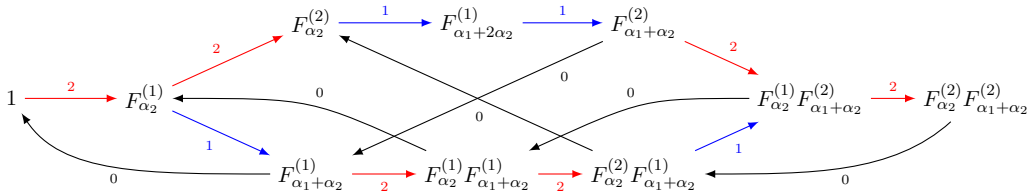
\[
\begin{tikzpicture}[>=latex, xscale=2.2, yscale=1, every node/.style={scale=0.75}]
\node (1) at (1.2, 0) {$1$};
\node (2) at (2, 0) {$F_{\alpha_2}^{(1)}$};
\node (3) at (3, 1) {$F_{\alpha_2}^{(2)}$};
\node (4) at (3, -1) {$F_{\alpha_1+\alpha_2}^{(1)}$};
\node (5) at (4, 1) {$F_{\alpha_1+2\alpha_2}^{(1)}$};
\node (6) at (4, -1) {$F_{\alpha_2}^{(1)} F_{\alpha_1+\alpha_2}^{(1)}$};
\node (7) at (5, 1) {$F_{\alpha_1+\alpha_2}^{(2)}$};
\node (8) at (5, -1) {$F_{\alpha_2}^{(2)} F_{\alpha_1+\alpha_2}^{(1)}$};
\node (9) at (6, 0) {$F_{\alpha_2}^{(1)} F_{\alpha_1+\alpha_2}^{(2)}$};
\node (10) at (7, 0) {$F_{\alpha_2}^{(2)} F_{\alpha_1+\alpha_2}^{(2)}$};
\draw[->,red] (1) -- node[midway,anchor=south] {\tiny $2$} (2);
\draw[->,red] (2) -- node[midway,anchor=south east] {\tiny $2$} (3);
\draw[->,blue] (2) -- node[midway,anchor=north east] {\tiny $1$} (4);
\draw[->,red] (4) -- node[midway,anchor=north] {\tiny $2$} (6);
\draw[->,red] (6) -- node[midway,anchor=north] {\tiny $2$} (8);
\draw[->,blue] (8) -- node[midway,anchor=north west] {\tiny $1$} (9);
\draw[->,blue] (3) -- node[midway,anchor=south] {\tiny $1$} (5);
\draw[->,blue] (5) -- node[midway,anchor=south] {\tiny $1$} (7);
\draw[->,red] (7) -- node[midway,anchor=south west] {\tiny $2$} (9);
\draw[->,red] (9) -- node[midway,anchor=south] {\tiny $2$} (10);
\draw[->,black] (10) .. controls (6.5,-1) and (6,-1) .. node[midway,anchor=north] {\tiny $0$} (8);
\draw[->,black] (9) .. controls (5, 0) .. node[midway,anchor=south east] {\tiny $0$} (6);
\draw[->,black] (7) -- node[pos=.2,anchor=north west] {\tiny $0$} (4);
\draw[->,black] (8) -- node[pos=.4,anchor=north east] {\tiny $0$} (3);
\draw[->,black] (6) .. controls (3, 0) .. node[midway,anchor=south west] {\tiny $0$} (2);
\draw[->,black] (4) .. controls (2,-1) and (1.5,-1) .. node[midway,anchor=north] {\tiny $0$} (1);
\end{tikzpicture}
\]
\end{ex}

\begin{figure}
\[
\begin{tikzpicture}[>=latex,xscale=1.5,yscale=1.5,baseline=0]
\node (hw) at (0,0) {$1$};
\node (2) at (0,-1) {$F_{\alpha_2}^{(1)}$};
\node (12) at (-1,-2) {$F_{\alpha_1+\alpha_2}^{(1)}$};
\node (22) at (1,-2) {$F_{\alpha_2}^{(2)}$};
\node (212) at (-2,-3) {$F_{\alpha_2}^{(1)} F_{\alpha_1+\alpha_2}^{(1)}$};
\node (122) at (0,-3) {$F_{\alpha_1+2\alpha_2}^{(1)}$};
\node (222) at (2,-3) {$F_{\alpha_2}^{(3)}$};
\node (2212) at (-3,-4) {$F_{\alpha_2}^{(2)} F_{\alpha_1+\alpha_2}^{(1)}$};
\node (1122) at (-1,-4) {$F_{\alpha_1+\alpha_2}^{(2)}$};
\node (1222) at (1,-4) {$F_{\alpha_2}^{(1)} F_{\alpha_1+2\alpha_2}^{(1)}$};
\node (2222) at (3,-4) {$F_{\alpha_2}^{(4)}$};
\node (22212) at (-4,-5) {$F_{\alpha_2}^{(3)} F_{\alpha_1+\alpha_2}^{(1)}$};
\node (21122) at (-2,-5) {$F_{\alpha_2}^{(1)} F_{\alpha_1+\alpha_2}^{(2)}$};
\node (11222) at (0,-5) {$F_{\alpha_1+2\alpha_2}^{(1)} F_{\alpha_1+\alpha_2}^{(1)}$};
\node (12222) at (2,-5) {$F_{\alpha_2}^{(2)} F_{\alpha_1+2\alpha_2}^{(1)}$};
\node (22222) at (4,-5) {$F_{\alpha_2}^{(5)}$};
\draw[->,red] (hw) -- node[midway,anchor=west] {\tiny $2$} (2);
\draw[->,blue] (2) -- node[midway,anchor=south east] {\tiny $1$} (12);
\draw[->,red] (2) -- node[midway,anchor=south west] {\tiny $2$} (22);
\draw[->,red] (12) -- node[midway,anchor=south east] {\tiny $2$} (212);
\draw[->,blue] (22) -- node[midway,anchor=south east] {\tiny $1$} (122);
\draw[->,red] (22) -- node[midway,anchor=south west] {\tiny $2$} (222);
\draw[->,red] (212) -- node[midway,anchor=south east] {\tiny $2$} (2212);
\draw[->,blue] (122) -- node[midway,anchor=south east] {\tiny $1$} (1122);
\draw[->,red] (122) -- node[midway,anchor=south west] {\tiny $2$} (1222);
\draw[->,blue] (222) -- node[midway,anchor=south east] {\tiny $1$} (1222);
\draw[->,red] (222) -- node[midway,anchor=south west] {\tiny $2$} (2222);
\draw[->,red] (2212) -- node[midway,anchor=south east] {\tiny $2$} (22212);
\draw[->,blue] (2212) -- node[midway,anchor=south west] {\tiny $1$} (21122);
\draw[->,red] (1122) -- node[midway,anchor=south east] {\tiny $2$} (21122);
\draw[->,blue] (1222) -- node[midway,anchor=south east] {\tiny $1$} (11222);
\draw[->,red] (1222) -- node[midway,anchor=south west] {\tiny $2$} (12222);
\draw[->,blue] (2222) -- node[midway,anchor=south east] {\tiny $1$} (12222);
\draw[->,red] (2222) -- node[midway,anchor=south west] {\tiny $2$} (22222);
\draw[->,black] (12) .. controls (-1, -1) .. node[midway,anchor=east] {\tiny $0$} (hw);
\draw[->,black] (212) .. controls (-2, -2) and (-1,-1) .. node[midway,anchor=east] {\tiny $0$} (2);
\draw[->,black] (2212) .. controls (-1, -3.5) and (0,-2) .. node[midway,anchor=south east] {\tiny $0$} (22);
\draw[->,black] (22212) .. controls (1, -4.2) and (.7, -3) .. node[midway,anchor=south east] {\tiny $0$} (222);
\draw[->,black] (21122) -- node[midway,anchor=east] {\tiny $0$} (212);
\draw[->,black] (1122) -- node[midway,anchor=east] {\tiny $0$} (12);
\draw[->,black] (11222) -- node[pos=.3,anchor=east] {\tiny $0$} (122);
\end{tikzpicture}
\]
\caption{The crystal $\aff{\B}^J$ in type $D_3^{(2)}$ with $J = \{1\}$ up to depth $5$. The $0$-arrows originating from elements of (classical) depth greater than $5$ are not drawn.}
\label{fig:affine_KR_limit_ex}
\end{figure}
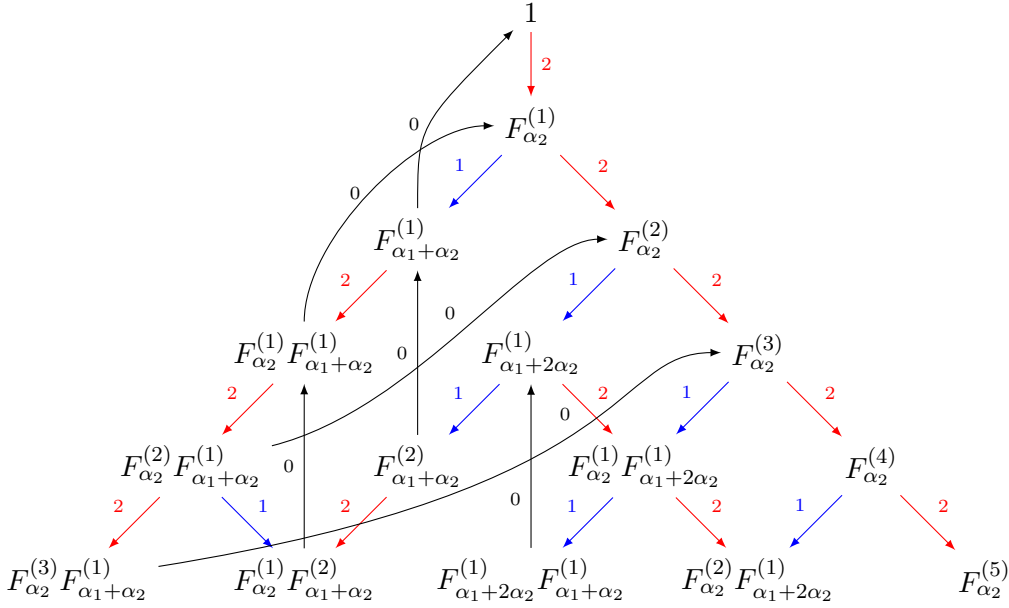

In the viewpoint of Section \ref{sec:minuscule crystals}, one may think of two problems as follows.

\begin{prob}
\label{prob:KR_RPP}
Directly describe the affine crystal operators in terms of the RPPs.
\end{prob}

We expect the solution to Problem~\ref{prob:KR_RPP} to apply $f_0$ on the corresponding tensor product from Proposition~\ref{prop:only_comparable} and then sorting so that the result is in order.
Note that this is not simply adding/removing $1$ from every entry in the RPP.

\begin{prob}
\label{prob:KR_quiver}
RPPs on minuscule posets were connected with quiver representations in~\cite{GPT18}, which we can then use to give another model for $B^{r,s}$.
Directly describe the affine crystal operators in terms of the corresponding quiver representations.
\end{prob}

%===================
\subsection{Proof of Theorem~\ref{thm:main1}} \label{proof of main theorem1}

Let $\virtual{\g}$ be a Lie algebra with index set $\virtual{I}$, root lattice $\virtual{Q}$ with simple roots $\{\virtual{\alpha}_j\}_{j \in \virtual{i}}$, and weight lattice $\virtual{P}$ with fundamental weights $\{\virtual{\Lambda}_j\}_{j \in \virtual{i}}$.
Suppose there exists a map $\phi \colon \virtual{I} \to I$ with \defn{scaling factors} $\{c_i \in \ZZ_{>0}\}_{i \in I}$ such that the linear function $\Phi \colon \aff{P} \to \virtual{P}$ satisfies
\begin{equation}
\label{eq:virtual_compatibility}
\Phi(\Lambda_i) = c_i \sum_{j \in \phi^{-1}(i)} \Lambda_j,
\qquad\qquad
\Phi(\alpha_i) = c_i \sum_{j \in \phi^{-1}(i)} \alpha_j.
\end{equation}
Note that the first condition uniquely defines the morphism and only the second needs to be checked.
For this paper, we will primarily consider $\virtual{\g}$ of type $A_{2n-1}^{(2)}$ (resp.\ $D_{n+1}^{(2)}$) for $\aff{\g}$ of type $B_n^{(1)}$ (resp.\ $C_n^{(1)}$) and take $\virtual{I} = I$ with $\phi(i) = i$ with $c_i = 1$ except for $c_n = 2$ (resp. $c_0 = c_n = 2$).

Using this data, there exists a way to construct the $U_q(\aff{\g})$-crystal $B(\lambda)$ from the $U_q(\virtual{\g})$-crystal $B(\Phi(\lambda))$ known as a \defn{virtual crystal} (see, \textit{e.g},~\cite{K96,SalisburyS18}).
We define the $U_q(\aff{\g})$-crystal by the map $u_{\lambda} \mapsto u_{\Phi(\lambda)}$ with the remaining crystal structure determined by
\begin{gather*}
e_i \mapsto \prod_{j \in \phi^{-1}(i)} \virtual{e}_j^{c_i},
\qquad
f_i \mapsto \prod_{j \in \phi^{-1}(i)} \virtual{f}_j^{c_i},
\\
\varepsilon_i(b) := \virtual{\varepsilon}_k(b),
\qquad
\varphi_i(b) := \virtual{\varphi}_k(b),
\qquad
\wt := \Phi^{-1} \circ \virtual{\wt},
\end{gather*}
where $\virtual{x}$ denotes the corresponding $U_q(\virtual{\g})$-crystal object and $k$ is any element in $\phi^{-1}(i)$.
Equation~\eqref{eq:virtual_compatibility} ensures all of these are well-defined~\cite{K96,OSS03III,OSS03II,SalisburyS18,SS2006}.
This is also known to extend to KR crystals in many cases~\cite{OSS03III,OSS03II,PS15,SS2006}.

\begin{ex}
\label{ex:virtual_crystal}
Consider type $C_2^{(1)}$, which we realize inside of $D_3^{(2)}$ by the scaling factors $(2, 1, 2)$.
Then the KR crystal $B^{2,1}$ in type $C_2^{(1)}$ is given by
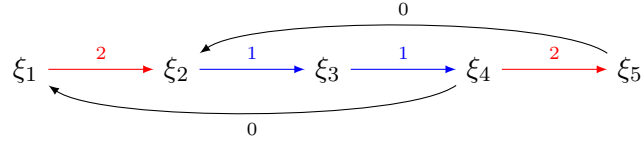
\[
\begin{tikzpicture}[>=latex, xscale=2, yscale=.7]
\foreach \i in {1,2,3,4,5} {
  \node (\i) at (\i, 0) {$\xi_{\i}$};
}
\draw[->,red] (1) -- node[midway,anchor=south] {\tiny $2$} (2);
\draw[->,blue] (2) -- node[midway,anchor=south] {\tiny $1$} (3);
\draw[->,blue] (3) -- node[midway,anchor=south] {\tiny $1$} (4);
\draw[->,red] (4) -- node[midway,anchor=south] {\tiny $2$} (5);
\draw[->,black] (5) .. controls (4.5,1) and (2.5,1) .. node[midway,anchor=south] {\tiny $0$} (2);
\draw[->,black] (4) .. controls (3.5,-1) and (1.5,-1) .. node[midway,anchor=north] {\tiny $0$} (1);
\end{tikzpicture}
\]
This is realized as the virtual crystal inside of $\widetilde{B}^{1,2} = B^{2,2}$ in type $D_3^{(2)}$:
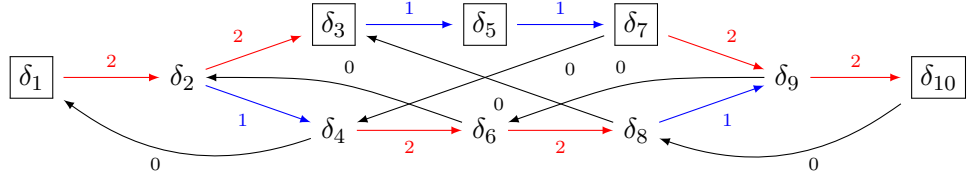
\[
\begin{tikzpicture}[>=latex, xscale=2, yscale=.7]
\node (1) at (1, 0) {$\boxed{\delta_1}$};
\node (2) at (2, 0) {$\delta_2$};
\node (3) at (3, 1) {$\boxed{\delta_3}$};
\node (4) at (3, -1) {$\delta_4$};
\node (5) at (4, 1) {$\boxed{\delta_5}$};
\node (6) at (4, -1) {$\delta_6$};
\node (7) at (5, 1) {$\boxed{\delta_7}$};
\node (8) at (5, -1) {$\delta_8$};
\node (9) at (6, 0) {$\delta_9$};
\node (10) at (7, 0) {$\boxed{\delta_{10}}$};
\draw[->,red] (1) -- node[midway,anchor=south] {\tiny $2$} (2);
\draw[->,red] (2) -- node[midway,anchor=south east] {\tiny $2$} (3);
\draw[->,blue] (2) -- node[midway,anchor=north east] {\tiny $1$} (4);
\draw[->,red] (4) -- node[midway,anchor=north] {\tiny $2$} (6);
\draw[->,red] (6) -- node[midway,anchor=north] {\tiny $2$} (8);
\draw[->,blue] (8) -- node[midway,anchor=north west] {\tiny $1$} (9);
\draw[->,blue] (3) -- node[midway,anchor=south] {\tiny $1$} (5);
\draw[->,blue] (5) -- node[midway,anchor=south] {\tiny $1$} (7);
\draw[->,red] (7) -- node[midway,anchor=south west] {\tiny $2$} (9);
\draw[->,red] (9) -- node[midway,anchor=south] {\tiny $2$} (10);
\draw[->,black] (10) .. controls (6.5,-1) and (6,-2) .. node[midway,anchor=north] {\tiny $0$} (8);
\draw[->,black] (9) .. controls (5, 0) .. node[midway,anchor=south east] {\tiny $0$} (6);
\draw[->,black] (7) -- node[pos=.2,anchor=north west] {\tiny $0$} (4);
\draw[->,black] (8) -- node[pos=.4,anchor=north east] {\tiny $0$} (3);
\draw[->,black] (6) .. controls (3, 0) .. node[midway,anchor=south west] {\tiny $0$} (2);
\draw[->,black] (4) .. controls (2,-2) and (1.5,-1) .. node[midway,anchor=north] {\tiny $0$} (1);
\end{tikzpicture}
\]
where we have boxed the entries corresponding to the type $C_1^{(1)}$ crystal.
\end{ex}
\medskip

We prove our main theorem by using the approach in~\cite{BS21,FOS09,Jang22,JS10}, where we branch $B^{r,s}$ into Levi subalgebras and show that there is a unique way to recombine these.
In particular, we will use the analog of~\cite[Thm.~3.15]{JS10}.
However, we require two technical results, where we provide uniform proofs of something previously shown using type specific information.
The first is a uniform proof that the decomposition of $B(s\clfw_r)$ into $U_q(\aff{\g}_J)$-crystals, where $J = I_0 \setminus \{r\}$, is multiplicity free (see, \textit{e.g.},~\cite[Rem.~5.2]{Jang22}).
Let $\{ \check{\Lambda}_j \mid j \in I \setminus \{r\}\}$ denote the fundamental weights of $\aff{\g}_r$.
The second is that the crystal $T_{s\clfw_r} \otimes \aff{\B}^{J,s} \iso B(s\check{\Lambda}_0)$ as $U_q(\aff{\g}_r)$-crystals under the $\tau_r$-twisted isomorphism (see, \textit{e.g.}, the proof of~\cite[Thm.~5.1]{Jang22}).

Let us begin by stating our method of proof, which is~\cite[Thm.~3.15]{JS10} but combined with~\cite[Lemma~3.12]{JS10} to be an assumption.
In particular, if we show these assumptions, then we will obtain Theorem~\ref{thm:main1} as a result.

\begin{thm}[{\cite[Lemma~3.12, Thm.~3.15]{JS10}}]
\label{thm:reconstruction}
Let $a, b \in I$ and $J = I \setminus \{a, b\}$.
Let $B$ and $B'$ be $U_q(\aff{\g})$-crystals such that
\begin{align*}
B & \iso B' & & \text{as $U_q(\aff{\g}_a)$-crystals,} \\
B & \iso B' & & \text{as $U_q(\aff{\g}_b)$-crystals,} \\
B & \iso B' \iso \bigoplus_{\mu \in M} B(\mu_{\downarrow}) & & \text{as $U_q(\aff{\g}_J)$-crystals,}
\end{align*}
where $M \subseteq \aff{P}^+$ and $\mu_{\downarrow}$ denotes the restriction to $\aff{P}_J$.
Then there exists a unique $U_q(\aff{\g})$-crystal isomorphism $\psi \colon B \to B'$.
\end{thm}

\begin{rem}
The third condition does \emph{not} imply that the decomposition into $U_q(\aff{\g}_J)$-crystals is multiplicity free, as was noted in~\cite[Sec.~3.4]{BS21}.
Multiplicity arises when the lowest weight element is also a $J$-highest weight element, as for the case of $B(\clfw_7)$ in type $E_7$.
\end{rem}

The remainder of this section is showing these conditions are satisfied for $B = T_{s\clfw_r} \otimes \aff{\B}^{J,s}$ and $B' = B^{r,s}$.

\begin{lem}
\label{lemma:total_order}
Let $r$ be a special node and $J = I_0 \setminus \{r\}$. Let $b, b' \in B(\clfw_r)$ be $J$-highest weight elements. Then $b$ and $b'$ are comparable.
\end{lem}

\begin{proof}
In order for two $J$-highest weight elements to be incomparable, there needs to exist a crystal operator that does not commute with $f_r$. However, since every reduced word of $W^J$ is fully-commutative, all crystal operators must commute.
\end{proof}

We would want to extend Proposition~\ref{prop:only_comparable} to include all special nodes, but it does not hold for the cominuscule nodes (that are not minuscule) as the following example shows.
Indeed, for type $B_n$, we only have that the connected component generated by $u_{\clfw_r}^{\otimes s}$ is contained within the subset, and for type $C_n$, we do not even have containment for $n \geq 4$.
Thus, we will need a separate proof for these two cases.

\begin{ex}
\label{ex:ordering_failure}
We will take $s = 2$ throughout this example.
In type $B_2$, the cominuscule node is $r = 1$, and we have $B(\clfw_1)$ given by the crystal graph
\[
1 \; {\color{blue} \xrightarrow[\hspace{20pt}]{1}} \; 2 \; {\color{red} \xrightarrow[\hspace{20pt}]{2}} \; 0 \; {\color{red} \xrightarrow[\hspace{20pt}]{2}} \; \btw \; {\color{blue} \xrightarrow[\hspace{20pt}]{1}} \; \bon.
\]
The element $0 \otimes 0$ is not in the connected component of $1 \otimes 1 \in B(\clfw_1)^{\otimes 2}$.
(Instead it is in the component corresponding to $B(\clfw_2)$.)

For type $C_4$, the cominuscule node is $r = 4$.
We have the element
\[
\ft_2^4 \ft_1 \ft_3^3 \ft_2 \ft_4^2 \ft_3^2 \ft_4 (u_{\clfw_4} \otimes u_{\clfw_4})
=
(\ft_2^2 \ft_3^2 \ft_4 u_{\clfw_4}) \otimes (\ft_2^2 \ft_1 \ft_3 \ft_2 \ft_4 \ft_3^2 \ft_4 u_{\clfw_4}),
\]
and a direct examination of the crystal graph for $B(\clfw_4)$ shows these two tensor factors are incomparable.
\end{ex}

\begin{lem}
\label{lemma:hw_decomp_elements}
Let $r$ be a special node and $J = I_0 \setminus \{r\}$.
Let
\[
b = b_1 \otimes \cdots \otimes b_{s-1} \otimes b_s \in B(s \clfw_r) \subseteq B(\clfw_r)^{\otimes s}
\]
be a $J$-highest weight element.
Then $b_k$ is $J$-highest weight for all $1 \leq k \leq s$.
\end{lem}

\begin{proof}
When $r$ is a cominuscule node, the result will follow from the realization as a virtual crystal and the corresponding statement in the ambient crystal (as $r$ corresponds to a minuscule node).
Thus, we can assume that $r$ is a minuscule node.

Recall that Proposition~\ref{prop:only_comparable} states $b_1 \leq \cdots \leq b_s$.
We prove the claim by induction on $s$ with the base case of $s = 1$ being trivial.
Consider a $J$-highest weight element $b = b' \otimes b_s$, and by induction, $b'$ is also $J$-highest weight.
If $b_s = u_{\clfw_r}$, then the claim holds with $b_1 = \cdots = b_s$, and so we now assume $b_s \neq u_{\clfw_r}$.
Since $b_{s-1} \leq b_s$, there exists $i_1, \dotsc, i_k \in I_0$ such that $b_s = f_{i_k} \cdots f_{i_1} b_{s-1}$.
By the tensor product rule, if $i_k \in J$, then
\begin{align*}
e_{i_k} b = e_{i_k}(b' \otimes b_s) = b' \otimes (e_{i_k} f_{i_k} \cdots f_{i_1} b_{s-1})
& = b' \otimes (f_{i_{k-1}} \cdots f_{i_1} b_{s-1}) \neq \zero.
\end{align*}
We have $i_k = r$ if and only if $b_s$ is $J$-highest weight, and thus $b$ is $J$-highest weight if and only if $b_s$ is $J$-highest weight.
\end{proof}

Now we prove the main technical result to show the Theorem~\ref{thm:main1}.

\begin{lem}
\label{lemma:decomp_by_weight}
Let $r$ be a special node and $J = I_0 \setminus \{r\}$.
Then the $J$-highest weight elements of $B(s\clfw_r)$ are uniquely determined by their weight.
\end{lem}

\begin{proof}
As in the proof of Lemma~\ref{lemma:hw_decomp_elements}, it is sufficient to show the claim when $r$ is a minuscule node by using virtual crystals when $r$ is cominuscule.

By Lemma~\ref{lemma:hw_decomp_elements}, it is sufficient to show that the sets
\[
X_b := \{ i \in I_0 \mid \varphi_i(b) = 1 \} \text{ for all } b \in B(\clfw_r) \text{ is a $J$-highest weight element},
\]
are disjoint as the weight is simply counting the different $J$-highest weight elements $b_k \in B(\clfw_r)$ occurring $b = b_1 \otimes \cdots \otimes b_s$.
Indeed, the crystal axiom (1) from Definition~\ref{def:abstract_crystal} states that $\varphi_i(b)$ contributes $\clfw_i$ to the weight.

By Proposition~\ref{prop:minuscule_Decomposition}, any $J$-connected component of $B(\clfw_r)$ must be a minuscule representation.
We first consider the case when $\g_J$ is a simple Lie algebra (equivalently the corresponding Dynkin diagram is connected).
A minuscule representation must have a highest weight of the form $\clfw_{r'}$ for some $r'$ that is a minuscule in $\g_J$.
Thus for any $J$-highest weight element $b \in B(\clfw_r)$, we must have $\abs{X_b} \leq 1$.
If $X_b = \emptyset$, then $b$ is the unique \emph{lowest} weight element $u_{\emptyset}$ of $B(\clfw_r)$, and hence $b \in B(\clfw_r) = \emptyset$ if and only if $b = u_{\emptyset}$.
Hence, we only need to consider the case $\abs{X_b} = 1$.

We first consider the case $\varphi_r(b) = 1$, so necessarily $\varphi_i(b) = 0$ for all $i \in J$.
Since any $r$-string must have length at most $1$ in a minuscule crystal, we must have $\varepsilon_r(b) = 0$.
Hence, $b$ is a highest weight element, which means $b = u_{\clfw_r}$.

Next we assume $X_b = X_{b'} = \{i\}$ for some $i \neq r$, and let $w_b$ and $w_{b'}$ be the corresponding minimal length coset representatives that give a path to $b$ and $b'$, respectively.
By Lemma~\ref{lemma:total_order}, these elements are comparable, and we can assume $b \geq b'$ without loss of generality.
Let $\beta = \wt(b) - \wt(b') \in Q^-$, and so $w_{\beta} = w_b w_{b'}^{-1}$.
However, $w_{\beta}$ stabilizes $\Lambda_i - \Lambda_r$, which implies that $w_{\beta} = 1$ as the weight must shift by $\beta$.
Hence, we must have $b = b'$ since elements are parameterized by their Weyl group elements.

For the case when $\g_J$ has reducible type, we can apply the above argument on each component.
Therefore, the sets $\{X_b\}$ are disjoint as claimed, with the weights of the $J$-highest weight elements given by
\[
\wt(b) = \begin{cases}
\clfw_r & \text{if } b = u_{\clfw_r}, \\
-\clfw_r + \sum_{i \in X_b} \clfw_i & \text{otherwise}.
\end{cases}
\]
Thus, the weight of any $J$-highest weight element $b \in B(s\clfw_r) \subseteq B(\clfw_r)^{\otimes s}$ is determined by the number $J$-highest weight of $B(\clfw_r)$ that appear.
\end{proof}

\begin{lem}
\label{lemma:gr_decomp}
Let $r$ be a special node.
We have
\[
T_{s\clfw_r} \otimes \aff{\B}^{J,s} \iso B(s \check{\Lambda}_0)
\]
as $U_q(\aff{\g}_r)$-crystals.
\end{lem}

\begin{proof}
We have $\alpha_0 = -\Theta + \delta$.
Note that $\tau_r \Delta_0 = \Delta_r$ and $\proj \Delta_r = \{\alpha_i \mid i \in J \} \cup \{-\Theta\} = w^J \Delta$ by Proposition~\ref{prop:twisted_roots}. Hence the crystal operator $f_0 = f_{-\Theta}$ as we have $\cc_{\beta} \mapsto \cc_{w^J \beta}$ for the simple roots $w^J \Delta$ from the conjugated Cartan subalgebra. Hence the claim follows from the fact that $B(s\clfw_r) \iso_{\tau_r} B(s\check{\Lambda}_0)$ and~\eqref{eq:g0_decomp}.
\end{proof}

Now, we are ready to prove Theorem \ref{thm:main1}.

\begin{proof}[Proof of Theorem~$\ref{thm:main1}$]
Our approach is to apply Theorem~\ref{thm:reconstruction}.
The first condition of Theorem~\ref{thm:reconstruction} is
\begin{equation}
\label{eq:g0_isomorphisms}
T_{s\clfw_r} \otimes \aff{\B}^{J,s} \iso B(s\clfw_r) \iso B^{r,s}
\end{equation}
as $U_q(\g_0)$-crystals by~\eqref{eq:g0_decomp} and $B^{r,s}$ is a minuscule KR crystal.
The second condition is Lemma~\ref{lemma:gr_decomp} and $B^{r,s} \iso B(s\check{\Lambda}_0)$ as $U_q(\aff{\g}_r)$-crystals from the diagram automorphism $\tau_r$ acting on $B^{r,s}$.
The third condition follows from Lemma~\ref{lemma:decomp_by_weight}, the level-zero condition on the weight as $U_q'(\aff{\g})$-crystals, and the isomorphisms~\eqref{eq:g0_isomorphisms}.
\end{proof}

%%%%%%%%%%%%%%%%%%%%%%%%%%%%%%%%%%%%%%%%
\section{Combinatorial characterization of \texorpdfstring{$\B^{J,s}$}{BJs}} \label{sec:path characterization}

Recall that
\[
	\B^J \hookrightarrow \B_{\mathbf{i}_0}, \qquad
	T_{s\clfw_r} \otimes \aff{\B}^{J,s} \iso B^{r,s},
\]
where $\aff{\B}^{J,s} = \{ \cc \in \B^J \mid \varepsilon_r^{\ast}(\cc) \le s \}$.
In this section, we give a combinatorial formula for $\varepsilon_r^{\ast}(\cc)$.

\subsection{Paths on root arrays}

First, we introduce certain directed graphs and paths on them for types $A_n$ and $D_n$.
For a directed graph $G = (V, E)$, a (directed) \defn{path} from $v \to v'$ is a sequence of vertices $\pp = (p_1, p_2, \dotsc, p_s)$ (for some $s > 0$) such that $p_1 = v$ and $p_s = v'$ with $(p_k, p_{k+1}) \in E$ for all $k = 1, \dotsc, s-1$.

%%%%%%%%%%%%%%%%%%%%%%%%%%%%%%%%%%%%%%%%
\subsubsection{Type \texorpdfstring{$A_n$}{An}} \label{subsubsec:A-paths}

Suppose $\g = \mathfrak{sl}_n$.
For $i \in I_0$, we write $\alpha_i = \epsilon_i - \epsilon_{i+1}$, where $P = \bigoplus_{i=1}^n \ZZ \epsilon_i$ with $(\epsilon_i, \epsilon_j) = \delta_{ij}$.
Note that any positive root is of the form $\epsilon_i - \epsilon_j$ for $1 \le i < j \le n$.
Fix any node $r \in I_0$, all of which are special.
Then one can check 
\[
	(\Phi^J)^+ = \{ \epsilon_i - \epsilon_j \, \mid \, 1 \le i \le r < j \le n \}.
\]
Take $\ii^J \in R((w^J)^{-1})$ such that the convex order on $(\Phi^J)^+$ induced from $\ii^J$ by \eqref{eq:beta_k} is given by
\[
	\epsilon_a - \epsilon_b \prec \epsilon_c - \epsilon_d
	\quad \iff \quad
	b < d \text{ or } (b = d \text{ and } a > c).
\]

\begin{df} \label{df:A-path}
Let $\mathscr{G}(A_n, r)$ denote the digraph on the roots $(\Phi^J)^+$ with edges from $\epsilon_{a-1} - \epsilon_b$, for some $a < b$, to either $\epsilon_a - \epsilon_b$ or $\epsilon_{a-1} - \epsilon_{b+1}$.
Let $\mathscr{R}(A_n, r)$ denote the set of paths in $\mathscr{G}(A_n, r)$ from $\epsilon_1 - \epsilon_{r+1} \to \epsilon_r - \epsilon_n$.
\end{df}

The digraph $\mathscr{G}(A_n, r)$ is similar to the Hasse diagram of the minuscule poset $\mcP_r$, but the correspondence with the roots is different.
Subsequently, we draw them with a different convention.

\begin{ex}
Let $n = 4$ and $r = 2$.
Then $\ii^J = (2,1,3,2)$ and the convex order is given by 
\(
	\epsilon_2 - \epsilon_3 \prec \epsilon_1 - \epsilon_3 \prec \epsilon_2 - \epsilon_4 \prec \epsilon_1 - \epsilon_4.
\)
The two paths in $\mathscr{R}(A_3, 2)$ are
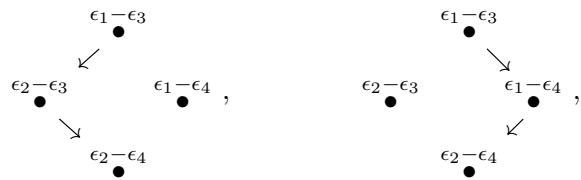
\[	
	\vcenter{
	\xymatrix@C=0em @R=0.5em{
		& \overset{\epsilon_1 - \epsilon_3}{\bullet} \ar@{->}[dl] & \\
		\overset{\epsilon_2 - \epsilon_3}{\bullet} \ar@{->}[dr] & &\!\!\!\overset{\epsilon_{1} - \epsilon_{4}}{\bullet} \\
		& \overset{\epsilon_2 - \epsilon_4}{\bullet} & 
	}},
	\qquad \qquad
	\vcenter{
	\xymatrix@C=0em @R=0.5em{
		& \overset{\epsilon_1 - \epsilon_3}{\bullet} \ar@{->}[dr] & \\
		\overset{\epsilon_2 - \epsilon_3}{\bullet} & &\!\!\!\overset{\epsilon_1 - \epsilon_4}{\bullet} \ar@{->}[dl] \\
		& \overset{\epsilon_2 - \epsilon_4}{\bullet} & 
	}},
\]
where we have omitted edges not in the path.
Note that we will use this same convention in other examples below.
\end{ex}

\subsubsection{Type \texorpdfstring{$D_n$}{Dn}}
Suppose $\g = \mathfrak{so}_{2n}$ for $n \ge 4$.
Under the same convention as in Section \ref{subsubsec:A-paths}, we write $\alpha_i = \epsilon_i - \epsilon_{i+1}$ for $i \neq n$, and $\alpha_n = \epsilon_{n-1} + \epsilon_n$.
When $r = 1$, we have
\[
	(\Phi^J)^+ = \{ \epsilon_1 \pm \epsilon_i \, \mid \, 2 \le i \le n \}.
\]
Take $\ii^J \in R(w^J)$ such that the convex order on $(\Phi^J)^+$ induced from $\ii^J$ is given by $\epsilon_1 - \epsilon_n \prec \epsilon_1 + \epsilon_n$ and 
\[
	\begin{aligned}
	\epsilon_1 - \epsilon_a \prec \epsilon_1 - \epsilon_b \iff a < b, \\
	\epsilon_1 + \epsilon_a \prec \epsilon_1 + \epsilon_b \iff a > b.
	\end{aligned}
\]

\begin{df} \label{df:D-path}
Let $\mathscr{G}(D_n, 1)$ denote the digraph on the roots $(\Phi^J)^+$ with edges
\begin{enumerate}[(1)]
	\item $\epsilon_1 \pm \epsilon_a \to \epsilon_1 \pm \epsilon_{a \mp 1}$ for all $a < n-1$;
	\item $\epsilon_1 - \epsilon_{n-1} \to \epsilon_1 - \epsilon_n$; and $\epsilon_1 \pm \epsilon_n \to \epsilon_1 + \epsilon_{n-1}$.
\end{enumerate}
Let $\mathscr{R}(D_n, 1)$ denote the set of paths $\epsilon_1 - \epsilon_2 \to \epsilon_1 + \epsilon_3$ and $\epsilon_1 - \epsilon_3 \to \epsilon_1 + \epsilon_2$.
\end{df}

The graph $\mathscr{G}(D_n, 1)$ is
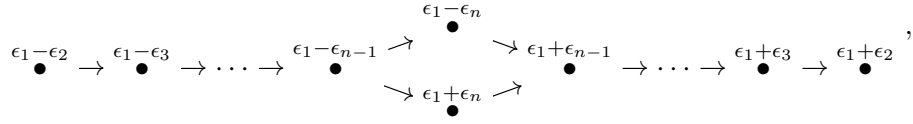
\[
	\xymatrix@R=-0.5em @C=0.8em{
		 	             &          &        &          & \overset{\epsilon_1 - \epsilon_n}{\bullet} \ar@{->}[dr]  \\
		 \overset{\epsilon_1 - \epsilon_2}{\bullet} \ar@{->}[r] &  \overset{\epsilon_1 - \epsilon_3}{\bullet} \ar@{->}[r] & \cdots \ar@{->}[r] & \overset{\epsilon_1 - \epsilon_{n-1}}{\bullet} \ar@{->}[ur] \ar@{->}[dr] & & \overset{\epsilon_1 + \epsilon_{n-1}}{\bullet} \ar@{->}[r] & \cdots \ar@{->}[r] & \overset{\epsilon_1 + \epsilon_3}{\bullet} \ar@{->}[r] & \overset{\epsilon_1 + \epsilon_2}{\bullet}   \\
		 	             &          &        &          &\overset{\epsilon_1 + \epsilon_n}{\bullet} \ar@{->}[ur]
	},
\]
which coincides with the crystal graph for $B(\Lambda_1)$ of type $D_{n-1}$ for $n \geq 4$.

\begin{ex}
There are always four paths in $\mathscr{R}(D_n, 1)$ for all $n \ge 4$.
For example, the paths in $\mathscr{R}(D_4, 1)$ are
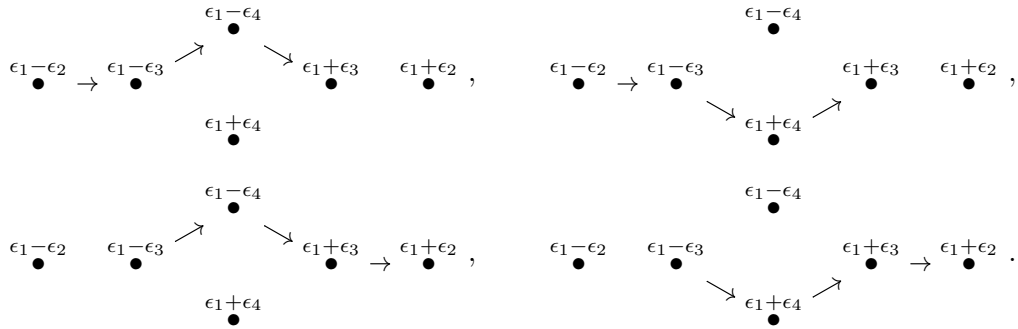
\[
\begin{split}
	\vcenter{
	\xymatrix@R=0em @C=0.65em{
	&          & \overset{\epsilon_1-\epsilon_4}{\bullet} \ar@{->}[dr]  \\
	\overset{\epsilon_1-\epsilon_2}{\bullet} \ar@{->}[r] & \overset{\epsilon_1-\epsilon_3}{\bullet} \ar@{->}[ur] & & \overset{\epsilon_1+\epsilon_3}{\bullet} & \overset{\epsilon_1+\epsilon_2}{\bullet}     \\
	&          & \overset{\epsilon_1 + \epsilon_4}{\bullet}
	}},
	\qquad
	\vcenter{
	\xymatrix@R=0em @C=0.65em{
	&          & \overset{\epsilon_1-\epsilon_4}{\bullet}  \\
	\overset{\epsilon_1-\epsilon_2}{\bullet} \ar@{->}[r] & \overset{\epsilon_1-\epsilon_3}{\bullet} \ar@{->}[dr] & & \overset{\epsilon_1+\epsilon_3}{\bullet} & \overset{\epsilon_1+\epsilon_2}{\bullet}     \\
	&          & \overset{\epsilon_1 + \epsilon_4}{\bullet} \ar@{->}[ur]
	}}, \\
	\vcenter{
	\xymatrix@R=0em @C=0.65em{
	&          & \overset{\epsilon_1-\epsilon_4}{\bullet} \ar@{->}[dr]  \\
	\overset{\epsilon_1-\epsilon_2}{\bullet} & \overset{\epsilon_1-\epsilon_3}{\bullet} \ar@{->}[ur] & & \overset{\epsilon_1+\epsilon_3}{\bullet} \ar@{->}[r] & \overset{\epsilon_1+\epsilon_2}{\bullet}     \\
	&          & \overset{\epsilon_1 + \epsilon_4}{\bullet}
	}},
	\qquad
	\vcenter{
	\xymatrix@R=0em @C=0.65em{
	&          & \overset{\epsilon_1-\epsilon_4}{\bullet}  \\
	\overset{\epsilon_1-\epsilon_2}{\bullet} & \overset{\epsilon_1-\epsilon_3}{\bullet} \ar@{->}[dr] & & \overset{\epsilon_1+\epsilon_3}{\bullet} \ar@{->}[r] & \overset{\epsilon_1+\epsilon_2}{\bullet}     \\
	&          & \overset{\epsilon_1 + \epsilon_4}{\bullet} \ar@{->}[ur]
	}}.
\end{split}
\]
\end{ex}

\begin{rem}
While we do not explicitly need them in this paper, the corresponding combinatorics for the other special nodes when $\g$ is of type $D_n$ or $E_{6,7}$ are known.
	\begin{itemize}
		\item[(D)] We refer the reader to \cite[Sec.~3.2]{JK19} for the definition of $\mathscr{R}(D_n,n)$, which is denoted there by $\Delta_n$, and to \cite[Def.~3.10]{JK19} for (double) paths on $\mathscr{G}(D_n,n)$.
		The (double) paths on $\mathscr{G}(D_n,n-1)$ are defined analogously.
		%For $r \in \{n-1,n\}$, we simply refer to a double path on $\mathscr{R}(D_n,r)$ as a \emph{path}.
		
		\item[(E)] For type $E_6$ (resp.~$E_7$) and a special node $r$, we define the set of a (triple, resp.~quadruple) paths $\mathscr{R}(E_6,r)$ (resp.~$\mathscr{R}(E_7,r)$) are defined in~\cite[Def.~6.2]{Jang22} and subject to a number of technical conditions.
	\end{itemize}
\end{rem}

%%%%%%%%%%%%%%%%%%%%%%%%%%%%%%%%%%%%%%%%
\subsubsection{Types \texorpdfstring{$B_n$}{Bn} and  \texorpdfstring{$C_n$}{Cn}}

Suppose $\g = \mathfrak{so}_{2n+1}$ or $\mathfrak{sp}_{2n}$.
Set $\kappa = 1$ if $\g = \mathfrak{so}_{2n+1}$, and $\kappa = 2$ if $\g = \mathfrak{sp}_{2n}$.
Under the same convention as in Section \ref{subsubsec:A-paths}, we write $\alpha_i = \epsilon_i - \epsilon_{i+1}$ for $i \neq n$, and $\alpha_n = \kappa \epsilon_n$.
Then we have
\[
	(\Phi^J)^+ = \begin{cases}
		\{ \epsilon_1 \pm \epsilon_i \, \mid \, 2 \le i \le n \} \cup \{ \kappa \epsilon_1  \} & \text{if $r = 1$, } \\
		\{ \epsilon_i + \epsilon_j \, \mid \, 1 \le i < j \le n \} \cup \{ \kappa \epsilon_i \, \mid \, 1 \le i \le n  \} & \text{if $r = n$. }
	\end{cases}
\]
For $r \in \{ 1, n \}$, we take $\ii^J \in R(w^J)$ such that the convex order on $(\Phi^J)^+$ induced from $\ii^J$ is as follows:
\begin{enumerate}
	\item if $r = 1$, then $\epsilon_1 - \epsilon_n \prec \kappa \epsilon_1 \prec \epsilon_1 + \epsilon_n$ and 
	\begin{align*}
		\epsilon_1 - \epsilon_a \prec \epsilon_1 - \epsilon_b \iff a < b, && \qquad
		\epsilon_1 + \epsilon_a \prec \epsilon_1 + \epsilon_b \iff a > b.
	\end{align*}
	
	\item if $r = n$, then 
	\begin{align*}
	  \epsilon_a+\epsilon_b\prec\epsilon_c+\epsilon_d
	  &\iff b>d\ \text{or}\ (b=d\text{ and }a>c), \\
	  \kappa \epsilon_a\prec \kappa \epsilon_c
	  &\iff a>c, \\
	  \kappa \epsilon_a\prec\epsilon_c+\epsilon_d
	  &\iff a\ge d, \\
	  \epsilon_a+\epsilon_b\prec \kappa \epsilon_c
	  &\iff b>c.
	\end{align*}
\end{enumerate}
For $X \in \{ B, C \}$ and $r \in \{ 1, n \}$, we define $\mathscr{R}(X_n, r)$ as follows, which is identified with $(\Phi^J)^+$.

\begin{df}
For $X_n = B_n, C_n$ and $r = 1, n$, let $\mathscr{G}(X_n, r)$ denote the digraph on the roots $(\Phi^J)^+$ with edges as follows.
\begin{itemize}
	\item[($1$)] If $r = 1$, then the edges are $\epsilon_1 \pm \epsilon_a \to \epsilon_1 \pm \epsilon_{a \mp 1}$ for all $a < n$ and $\epsilon_1 - \epsilon_n \to \kappa \epsilon_1 \to \epsilon_1 + \epsilon_n$.
	Let $\mathscr{R}(X_n, 1)$ be the pair of paths $\epsilon_1 - \epsilon_2 \to \epsilon_1 + \epsilon_3$ and $\epsilon_1 - \epsilon_3 \to \epsilon_1 + \epsilon_2$.
	
	\item[($n$)] If $r = n$, then the edges are from $\epsilon_a + \epsilon_b$, for $a < b$, given by the following two cases. If $\abs{a - b} > 1$, then it can go to either $\epsilon_{a+1} + \epsilon_b$ or $\epsilon_a + \epsilon_{b-1}$ and if $\abs{a - b} = 1$, then it can go to either $\kappa \epsilon_a$ or $\kappa \epsilon_{a+1}$.
	Let $\mathscr{R}(X_n, n)$ be the set of paths $\epsilon_1 + \epsilon_n \to \kappa \epsilon_i$ for some $i \in I_0$.
\end{itemize}
Let
 a \defn{path} in $\mathscr{R}(X_n, r)$ is a sequence $\pp = (\beta_k \in (\Phi^J)^+)_{k=1}^s$ for some $s \ge 1$ defined as follows.
\end{df}

The graph $\mathscr{G}(X_n, 1)$ is simply the path
\[
	\xymatrix@R=-0.5em @C=0.7em{
		 \overset{\epsilon_1 - \epsilon_2}{\bullet} \ar@{->}[r] &  \overset{\epsilon_1 - \epsilon_3}{\bullet} \ar@{->}[r] & \cdots \ar@{->}[r] & \overset{\epsilon_1 - \epsilon_n}{\bullet} \ar@{->}[r] & \overset{\kappa \epsilon_1}{\bullet} \ar@{->}[r] & \overset{\epsilon_1 + \epsilon_n}{\bullet} \ar@{->}[r] & \cdots \ar@{->}[r] & \overset{\epsilon_1 + \epsilon_3}{\bullet} \ar@{->}[r] & \overset{\epsilon_1 + \epsilon_2}{\bullet} 
	}.
\]
The graph $\mathscr{G}(X_n, n)$ has size given by the $(n-1)$-th triangular number and locally has edges given by the triangles, with $a < b - 1$,
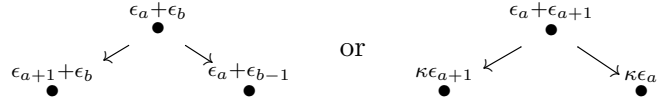
\begin{equation*} %\label{eq:elementary triangle}
\begin{split}
	\vcenter{
	\xymatrix@C=0.5em @R=0.2em{
		& \overset{\epsilon_a + \epsilon_b}{\bullet} \ar@{->}[dr] \ar@{->}[dl] & \\
		\overset{\epsilon_{a+1} + \epsilon_b}{\bullet} & &\!\!\!\overset{\epsilon_a + \epsilon_{b-1}}{\bullet}
	}}
	\quad
	\text{ or }
	\quad
	\vcenter{
	\xymatrix@C=0.5em @R=0.2em{
		& \overset{\epsilon_a+\epsilon_{a+1}}{\bullet} \ar@{->}[dr] \ar@{->}[dl] & \\
		\overset{\kappa \epsilon_{a+1}}{\bullet} & &\!\overset{\kappa \epsilon_{a}}{\bullet}
	}}
\end{split}
\end{equation*}

\begin{ex} \label{ex:paths in B3, C3}
Let $X = B, C$.
\begin{enumerate}[(1)]
	\item The two paths in $\mathscr{R}(X_3, 1)$ are
		\[
			\xymatrix@C=1em{
				\overset{\epsilon_1-\epsilon_2}{\bullet} \ar@{->}[r] & \overset{\epsilon_1-\epsilon_3}{\bullet} \ar@{->}[r] & \overset{\kappa \epsilon_1}{\bullet} \ar@{->}[r] & \overset{\epsilon_1+\epsilon_3}{\bullet} & \overset{\epsilon_1+\epsilon_2}{\bullet} 
			}\,, \qquad
			\xymatrix@C=1em{
				\overset{\epsilon_1-\epsilon_2}{\bullet} & \overset{\epsilon_1-\epsilon_3}{\bullet} \ar@{->}[r] & \overset{\kappa \epsilon_1}{\bullet} \ar@{->}[r] & \overset{\epsilon_1+\epsilon_3}{\bullet} \ar@{->}[r] & \overset{\epsilon_1+\epsilon_2}{\bullet} 
			}.
		\]
		
	\item The paths in $\mathscr{R}(X_3, n)$ are
		\[
			\vcenter
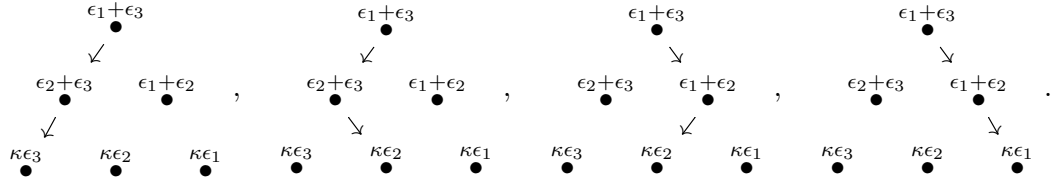
{
			\xymatrix@C=-0.95em @R=0.6em{
				& & \overset{\epsilon_1+\epsilon_3}{\bullet} \ar@{->}[dl] & & \\
				& \overset{\epsilon_2+\epsilon_3}{\bullet} \ar@{->}[dl] & & \overset{\epsilon_1 + \epsilon_2}{\bullet} & \\
				\overset{\kappa \epsilon_3}{\bullet} & & \overset{\kappa \epsilon_2}{\bullet} & & \overset{\kappa \epsilon_1}{\bullet}
			}},
			\quad\!
			\vcenter{
			\xymatrix@C=-0.95em @R=0.5em{
				& & \overset{\epsilon_1+\epsilon_3}{\bullet} \ar@{->}[dl] & & \\
				& \overset{\epsilon_2+\epsilon_3}{\bullet} \ar@{->}[dr] & & \overset{\epsilon_1+\epsilon_2}{\bullet} & \\
				\overset{\kappa \epsilon_3}{\bullet} & & \overset{\kappa \epsilon_2}{\bullet} & & \overset{\kappa \epsilon_1}{\bullet}
			}},
			\quad\!
			\vcenter{
			\xymatrix@C=-0.95em @R=0.5em{
				& & \overset{\epsilon_1+\epsilon_3}{\bullet} \ar@{->}[dr] & & \\
				& \overset{\epsilon_2+\epsilon_3}{\bullet} & & \overset{\epsilon_1+\epsilon_2}{\bullet} \ar@{->}[dl] & \\
				\overset{\kappa \epsilon_3}{\bullet} & & \overset{\kappa \epsilon_2}{\bullet} & & \overset{\kappa \epsilon_1}{\bullet}
			}},
			\quad\!
			\vcenter{
			\xymatrix@C=-0.95em @R=0.5em{
				& & \overset{\epsilon_1+\epsilon_3}{\bullet} \ar@{->}[dr] & & \\
				& \overset{\epsilon_2+\epsilon_3}{\bullet} & & \overset{\epsilon_1+\epsilon_2}{\bullet} \ar@{->}[dr] & \\
				\overset{\kappa \epsilon_3}{\bullet} & & \overset{\kappa \epsilon_2}{\bullet} & & \overset{\kappa \epsilon_1}{\bullet}
			}}.
		\]
\end{enumerate}
\end{ex}

%%%%%%%%%%%%%%%%%%%%%%%%%%%%%%%%%%%%%%%%
\subsection{Path description of \texorpdfstring{$\varepsilon_r^*$}{ep r*}}

Let $\gaff$ be an affine Lie algebra, and fix a special node~$r$ (see Table~\ref{tab:dynkins}).
%In this case, we have the classical Lie algebra $\g$ is of type $X_n$ for $X = A,B,C,D,E$.
If $\g$ is of type $B_n$, for $\beta \in \Phi^+$, we define $\gamma_\beta = 2$ if $\beta$ is long, and $\gamma_\beta = 1$ if $\beta$ is short.
For $\cc = \left( c_\beta \right) \in \B^J$ and a path $\pp \in \mathscr{R}(X_n, r)$, we define
\begin{equation} \label{eq:norm of cc on pp}
	\Abs{\cc}_{\pp} 
	= 
	\begin{cases}
		\displaystyle 
		\sum_{\beta \text{ lying on } \pp} \gamma_\beta c_\beta & \text{if $\g$ is of type $B_n$,} \\[1ex]
		\displaystyle 
		\sum_{\beta \text{ lying on } \pp} c_\beta & \text{otherwise.}
	\end{cases}
\end{equation}

\begin{thm} \label{thm:path formula}
	Suppose that $\g$ is a simple finite dimensional Lie algebra, and set $J = I_0 \setminus \{r\}$, where $r$ is a special node.
	For $\cc \in \B^J$, we have
	\[
		\varepsilon_r^{\ast}(\cc) 
		= \max \{ \Abs{\cc}_{\pp} \mid {\pp \in \mathscr{R}(X_n, r)} \}.
	\]
\end{thm}

\begin{proof}
	For type $A_n$, the above formula can be essentially obtained from \cite{Greene74, Kwon13} by realizing $\B^J$ as the set of $r \times (n-r+1)$ non-negative integral matrices (see \cite[Sec.~4.3]{Kwon18}), or directly proven following \cite{JK19} by means of \cite{BZ01}, and for types $D_n$ and $E_{6,7}$, it was proved in \cite{JK19, Jang22}. Note that the case of type $D_n$ with $r = 1$ can be proven similarly as in \cite{JK19} by considering the minuscule crystal $B(\varpi_1)$.
    
    Now we focus on types $B_n$ and $C_n$ with $r \in \{ 1, n \}$.
    Here we give the details only for types $B_n$ and $C_n$ with $r=n$.
    The case for $r = 1$ are proved similarly.

    Recall $\kappa = 1$ for type $B_n$ and $\kappa = 2$ for type $C_n$.
    We define the following from folding type $A_{2n-1}$ to type $B_n, C_n$ (see, \textit{e.g.},~\cite{K96,OSS03III,OSS03II}).
    Let $\mathring{\g} = \mathfrak{gl}_{2n}$, and $\mathring{I} = \{ 1, 2, \dots, 2n-1 \}$.
    Let $\mathring{P} = \bigoplus_{i=1}^{2n} \ZZ \mathring{\epsilon}_i$ be the weight lattice of $\mathring{\g}$, and let $\{\,\mathring{\alpha}_i := \mathring{\epsilon}_i - \mathring{\epsilon}_{i+1} \,\mid\, i \in \mathring{I} \,\}$ be the set of simple roots of $\mathring{\g}$.
    We denote by $\mathring{W}$ the Weyl group of $\mathring{\g}$.
    Consider the automorphism $\sigma$ of Dynkin diagram of type $A_{2n-1}$ given by $\sigma(\mathring{\alpha}_i) = \mathring{\alpha}_{2n-i}$ for $i \in \mathring{I}$.
    It is easy to see $\sigma$ induces a bijection on the set of all paths $\mathscr{R}(A_{2n-1}, n)$, which we also denote by $\sigma$ by abuse of notation.

    Let $- \colon Q \rightarrow \mathring{Q}$ be an embedding given by 
    \[
        \overline{\alpha_i} = 
        \begin{cases}
            \mathring{\alpha}_i + \mathring{\alpha}_{2n-i} & \text{if } 1 \le i \le n-1, \\
            \kappa \mathring{\alpha}_n & \text{if } i = n.
        \end{cases}
    \]
    It is well-known that the map $- \colon W \rightarrow \mathring{W}$ given by 
    \[
        \overline{s_i} = 
        \begin{cases}
            \mathring{s}_i \mathring{s}_{2n-i} & \text{if } 1 \le i \le n-1, \\
            \mathring{s}_n & \text{if } i = n,
        \end{cases}
    \]
    is an injective group homomorphism, where $\mathring{s}_k$ is the simple reflection in $\mathring{W}$ associated to $\mathring{\alpha}_k$ for $1 \le k \le 2n-1$. 
    Note that 
    \begin{equation} \label{eq:compatibility of Weyl actions}
        \overline{w\vphantom{\lambda}}(\overline{\lambda}) = \overline{w(\lambda)}
    \end{equation}
    for $w \in W$ and $\lambda \in Q$ \cite{FRS97}.
    One can take a reduced expression $\mathring{\ii}_0$ of the longest element $\mathring{w}_0 \in \mathring{W}$ such that 
    \begin{equation} \label{eq:type A i0J}
        \mathring{\bf i}_0^J = \mathring{\ii}_0^{(1)} \mathring{\ii}_0^{(2)} \dotsm \mathring{\ii}_0^{(n)} \in R(\mathring{w}_0^J),
    \end{equation}
    where $\mathring{\ii}_0^{(s)} = (n+s-1, n+s-2, \dots, s+1, s)$ for $1 \le s \le n$, and we can obtain any $\overline{\ii_0} \in R(\overline{w_0})$ from $\mathring{\ii}_0$ by commutation or $2$-braid moves.

    Let $\mathring{B}(\infty)$ be the crystal of $U_q^-(\mathring{\g})$. Then we define
    \[
        \boldsymbol{\et}_i = 
        \begin{cases}
            \mathring{\et}_i \mathring{\et}_{2n-i} & \text{ if }  1 \le i \le n-1, \\
            \mathring{\et}_n^\kappa & \text{ if } i =n,
        \end{cases}
        \qquad\qquad
        \boldsymbol{\ft}_i = 
        \begin{cases}
            \mathring{\ft}_i \mathring{\ft}_{2n-i} & \text{ if }  1 \le i \le n-1, \\
            \mathring{\ft}_n^\kappa & \text{ if } i =n,
        \end{cases}
    \]
    where $\mathring{\et}_i$ and $\mathring{\ft}_i$ are the Kashiwara operators on $\mathring{B}(\infty)$. 
    Also we define $\boldsymbol{\et}_i^\ast = \ast \circ \boldsymbol{\et}_i \circ \ast$ and $\boldsymbol{\ft}_i^\ast = \ast \circ \boldsymbol{\ft}_i \circ \ast$, where $\ast$ is the bijection on $\mathring{B}(\infty)$ induced from the anti-involution $\ast$ on $U_q^-(\mathring{\g})$.
    It is well-known\footnote{This can also be easily deduced from~\cite{SalisburyS15II,SalisburyS16II}.} that there exists an injective map $\chi \colon B(\infty) \to \mathring{B}(\infty)$ such that
    \[
    \chi(\et_i b) = \boldsymbol{\et}_i \chi(b),
    \qquad \chi(\et_i^\ast b) = \boldsymbol{\et}_i^\ast \chi(b),
    \qquad \chi(\ft_i b) = \boldsymbol{\ft}_i \chi(b),
    \qquad \chi(\ft_i^\ast b) = \boldsymbol{\ft}_i^\ast \chi(b),
    \]
    and $\overline{\wt(b)} = \wt\bigl(\chi(b)\bigr)$ for $b \in B(\infty)$ and $i \in I$.
    In particular, for $b \in B(\infty)$, we have
    \begin{equation} \label{eq:epsilon-n-ast-chi}
        \kappa \varepsilon_n^*(b) = \varepsilon_n^*\bigl(\chi(b)\bigr).
    \end{equation}

    For $\cc = (c_k) \in \B_{\mathring{\ii}_0}$, set $c_k = c_{ij}$ if $\mathring{\beta}_k = \mathring{\epsilon}_i - \mathring{\epsilon}_j$ for some $1 \le i < j \le 2n$, where $\mathring{\beta}_k = \mathring{s}_{j_1} \cdots \mathring{s}_{j_{k-1}}(\mathring{\alpha}_{j_k})$ comes from the convex ordering associated to $\mathring{\ii}$.
    By \cite[Thm.~ 4.6]{Kwon18a} (by means of \eqref{eq:compatibility of Weyl actions}), the map $\chi \colon B(\infty) \to \mathring{B}(\infty)$ can be characterized in terms of $\B_{\ii_0}$ and $\B_{\mathring{\ii}_0}$ so that image of $\chi \colon \B_{\ii_0} \to \B_{\mathring{\ii}_0}$ is given by
    \begin{equation} \label{eq:image-chi}
        \chi\left( \B_{\ii_0} \right) = 
        \{
            \cc \in \B_{\mathring{\ii}_0} \,\mid\,
            \sigma(\cc) = \cc \text{ and } \kappa | c_{k\, 2n-k+1} \, (1 \leqslant k \leqslant n)\,
        \},
    \end{equation}
    where $\sigma$ is the bijection on $\B_{\ii_0}$ given by $\sigma(\cc) = (c_{ij}^\sigma)$ with $c_{ij}^\sigma = c_{2n-j+1, 2n-i+1}$ $(1 \le i < j \le 2n)$ for $\cc = (c_{ij}) \in \B_{\mathring{\ii}_0}$. 
    
    Also, it follows from \cite[Eq.~(4.5)]{Kwon18a} that $\chi(\cc) = (d_1, \dots, d_M)$ is given by 
    \begin{equation} \label{eq:relationship of c and d}
        d_k= \begin{cases}
        c_j & \text { if } \mathring{\beta}_k \neq \sigma(\mathring{\beta}_k) \text { and } \mathring{\beta}_k+\sigma(\mathring{\beta}_k)=\overline{\beta_j} \text { for some } 1 \leqslant j \leqslant N, \\ 
        \kappa c_j & \text { if } \mathring{\beta}_k=\sigma(\mathring{\beta}_k) \text { and } \epsilon \mathring{\beta}_k=\overline{\beta_j} \text { for some } 1 \leqslant j \leqslant N,\end{cases}
    \end{equation}
    for $1 \le k \le M$.
    By \eqref{eq:epsilon-n-ast-chi}, for $\cc \in \B^J$, we have
    \begin{equation} \label{eq:epsilon ast from B to A}
        \kappa \varepsilon_n^\ast(\cc) 
        = \varepsilon_n^\ast\bigl(\chi(\cc)\bigr)
        = \max \{\Abs{\chi(\cc)}_{\mathring{\pp}} \mid {\mathring{\pp} \in \mathscr{R}(A_{2n-1}, n)} \}.
    \end{equation}
    Note that it follows from \eqref{eq:image-chi} that the above maximum should be attained at a path $\mathring{\pp}$ satisfying $\mathring{\pp} = \sigma(\mathring{\pp})$.
    
    Set $\mathscr{R}(A_{2n-1}, n)^{\sigma} := \{\mathring{\pp} \in \mathscr{R}(A_{2n-1}, n) \mid \mathring{\pp} = \sigma(\mathring{\pp}) \}$.
    Let $\pp = (\beta_{k_1}, \dots, \beta_{k_m}) \in \mathscr{R}(X_n, n)$ be given. Put $\overline{\pp} = (\overline{\beta_{k_1}}, \dots, \overline{\beta_{k_m}})$.
    Recall that $\overline{\beta_{k_l}}$ is expressed as a sum of one or two positive roots of type $A_{2n-1}$ (see \eqref{eq:relationship of c and d}).
    We define $\mathring{\pp} = (\mathring{\beta}_{j_1}, \dots,  \mathring{\beta}_{j_{m'}})$ (for some $m'$) to be the sequence obtained from $\overline{\pp}$ by replacing each term $\overline{\beta_{k_l}}$ with its one or two positive root summands.
    By \eqref{eq:compatibility of Weyl actions} and~\eqref{eq:type A i0J}, we have $\mathring{\pp} \in \mathscr{R}(A_{2n-1}, n)$.
    Then it is straightforward to check that the map 
    \begin{equation} \label{eq:path-bijection}
    \begin{split}
        \xymatrix@R=0em{
            \mathscr{R}(X_n, n) \ar@{->}[r] & \mathscr{R}(A_{2n-1}, n)^\sigma \\ 
            \pp \ar@{->}[r] & \overline{\pp}
        }
    \end{split}
    \end{equation}
    is bijective.
    By \eqref{eq:image-chi}, \eqref{eq:relationship of c and d}, \eqref{eq:epsilon ast from B to A}, and \eqref{eq:path-bijection}, we have 
    \[
        \varepsilon_n^\ast(\cc)  = 
        \frac{1}{\kappa} \max \{\Abs{\chi(\cc)}_{\mathring{\pp}} \mid {\mathring{\pp} \in \mathscr{R}(A_{2n-1}, n)^\sigma} \} = 
        \{ \max \Abs{\cc}_{\pp} \mid {\pp \in \mathscr{R}(X_n, n)} \}
    \]
    as desired.

	In the case of $r = 1$, let $\mathring{\g}$ be of type $D_{n+1}$, that is, $\mathring{\g} = \mathfrak{so}_{2(n+1)}$, and consider the automorphism $\sigma$ of the Dynkin diagram of $\mathring{\g}$ defined by 
    \[
    \sigma(\mathring{\alpha}_i) = \mathring{\alpha}_i \quad (i \neq n, n+1), \qquad 
    \sigma(\mathring{\alpha}_n) = \mathring{\alpha}_{n+1}, \qquad 
    \sigma(\mathring{\alpha}_{n+1}) = \mathring{\alpha}_n,
    \]
    In this setting, a similar argument to the one above then applies.
\end{proof}

\begin{ex}
  We give an example of the constructions used in the proof of Theorem~\ref{thm:path formula}.
  Recall Example~\ref{ex:paths in B3, C3}, where
  \[
    \vcenter{
      \xymatrix@C=0.5em @R=0.5em{
        & & \overset{\beta_3}{\bullet} \ar@{->}[dl]
             & & \\
        & \overset{\beta_2}{\bullet} \ar@{->}[dl]
          & & \overset{\beta_5}{\bullet} & \\
        \overset{\beta_1}{\bullet}
          & & \overset{\beta_4}{\bullet}
          & & \overset{\beta_6}{\bullet}
      }
    },
    \qquad\qquad
    \vcenter{
      \xymatrix@C=0.5em @R=0.5em{
        & & \overset{\beta_3}{\bullet} \ar@{->}[dl]
             & & \\
        & \overset{\beta_2}{\bullet} \ar@{->}[dr]
          & & \overset{\beta_5}{\bullet} & \\
        \overset{\beta_1}{\bullet}
          & & \overset{\beta_4}{\bullet}
          & & \overset{\beta_6}{\bullet}
      }
    }.
  \]
  are two paths in $\mathscr{R}(C_3, 3)$.
  For this case, we have $\ii^J_0 = (3, 2, 3, 1, 2, 3) \in R(w_0^J)$, and therefore $\overline{\ii}_0^J = (3, 2,4, 3, 1,5, 2,4, 3)$ and $\mathring{\ii}_0^J = (3,2,1,4,3,2,5,4,3)$, which we can see differ only by $2$-term braid moves.
  The convex order we consider is $\beta_1 = \alpha_3$, $\beta_2 = \alpha_2 + \alpha_3$, $\beta_3 = \alpha_1 + \alpha_2 + \alpha_3$, $\beta_4 = 2\alpha_2 + \alpha_3$, $\beta_5 = \alpha_1 + 2\alpha_2 + \alpha_3$, and $\beta_6 = 2\alpha_1 + 2\alpha_2 + \alpha_3$.
  For $\pp = (\beta_3, \beta_2, \beta_1) \in \mathscr{R}(C_3, 3)$, we associate
  \[
  	\mathring{\pp} = (\mathring{\alpha}_3,\, \mathring{\alpha}_2 + \mathring{\alpha}_3, \mathring{\alpha}_3 + \mathring{\alpha}_4, 
	\mathring{\alpha}_1 + \mathring{\alpha}_2 + \mathring{\alpha}_3, \mathring{\alpha}_3 + \mathring{\alpha}_4 + \mathring{\alpha}_5),
  \]
  where $\sigma(\mathring{\alpha}_3) = \mathring{\alpha}_3$, $\sigma(\mathring{\alpha}_2 + \mathring{\alpha}_3) = \mathring{\alpha}_3 + \mathring{\alpha}_4$ and $\sigma(\mathring{\alpha}_1 + \mathring{\alpha}_2 + \mathring{\alpha}_3) = \mathring{\alpha}_3 + \mathring{\alpha}_4 + \mathring{\alpha}_5$.
  Drawing this and the other path given above as paths in $\mathscr{R}(A_5, 3)$, we have
  \[
    \vcenter
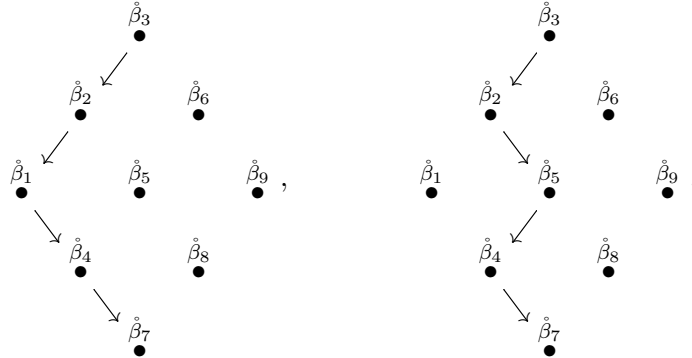
{
      \xymatrix@C=0.5em @R=0.5em{
        & & \overset{\mathring{\beta}_3}{\bullet} \ar@{->}[dl]
             & & \\
        & \overset{\mathring{\beta}_2}{\bullet} \ar@{->}[dl]
          & & \overset{\mathring{\beta}_6}{\bullet} & \\
        \overset{\mathring{\beta}_1}{\bullet} \ar@{->}[dr]
          & & \overset{\mathring{\beta}_5}{\bullet}
          & & \overset{\mathring{\beta}_9}{\bullet} \\
          &        \overset{\mathring{\beta}_4}{\bullet} \ar@{->}[dr]
          & & \overset{\mathring{\beta}_8}{\bullet} \\
                    & & \overset{\mathring{\beta}_7}{\bullet}
      }
    },
    \qquad\qquad
    \vcenter{
      \xymatrix@C=0.5em @R=0.5em{
        & & \overset{\mathring{\beta}_3}{\bullet} \ar@{->}[dl]
             & & \\
        & \overset{\mathring{\beta}_2}{\bullet} \ar@{->}[dr]
          & & \overset{\mathring{\beta}_6}{\bullet} & \\
        \overset{\mathring{\beta}_1}{\bullet}
          & & \overset{\mathring{\beta}_5}{\bullet} \ar@{->}[dl]
          & & \overset{\mathring{\beta}_9}{\bullet} \\
          &        \overset{\mathring{\beta}_4}{\bullet} \ar@{->}[dr]
          & & \overset{\mathring{\beta}_8}{\bullet} \\
                    & & \overset{\mathring{\beta}_7}{\bullet}
      }
    }.
  \]
  where we have labeled the nodes according to the convex order associated to $\mathring{\ii}^J_0$.
\end{ex}

\begin{rem}
    For $(X_n, r) = (B_n, n)$ or $(C_n, 1)$, the same formula for $\varepsilon_r^\ast(\cc)$ ($\cc \in \B^J$) can also be obtained from the transition map between the string and Lusztig parameterizations in terms of $\ii$-trails~\cite[Thm.~3.7]{BZ01} as in~\cite{JK19}. 
    Specifically,~\cite{JK19} characterizes the $\ii_0^{\rm op\ast}$-trails from $\varpi_n$ (resp.~$s_n\varpi_n$) to $w_0\varpi_n$ in type $D_n$, where $\ii_0^\ast = (i_1^\ast, \dots, i_N^\ast)$ by $w_0(\alpha_i) = -\alpha_{i^\ast}$ $(i \in I)$, and $\ii_0^{\rm op} = (i_N, \dots, i_1)$.
Note that when the fundamental ${}^L\g$-module $V(\varpi_r^\vee)$ is minuscule (where $\varpi_r^\vee$ denotes the $r$-th fundamental weight of ${}^L\g$), the $\ii$-trails can be described in terms of the crystal $B(\varpi_r^\vee)$;
this fact is essential in \cite{JK19}.
Thus the approach of \cite{JK19} cannot be applied directly to the case of $(X_n, r) = (B_n, 1)$ or $(C_n, n)$, since $V(\varpi_r^\vee)$ is not minuscule in these cases.
However, by analyzing $V(\varpi_r^\vee)$, we show that it suffices to consider the \emph{extremal} $\ii_0^{\rm op\ast}$-trails from $s_r\varpi_r^\vee$ to $w_0\varpi_r^\vee$, which are explicitly characterized in~\cite[Prop.~9.2]{BZ01}.
This reduction enables us to apply the approach of~\cite{JK19} to the case of $(X_n, r) = (B_n, 1)$ or $(C_n, n)$, yielding the same formula for $\varepsilon_r^\ast(\cc)$.
\end{rem}

\bibliographystyle{alpha}
\bibliography{kr_crystals}{}

\end{document}